\documentclass[12pt,          
               phd,           
               onehalfspacing 
               ]{ncsuthesis}

\usepackage{booktabs}  
\usepackage{amsmath}
\usepackage{textcomp}  
\usepackage{xcolor}
\usepackage{lipsum}    
\usepackage{longtable}

\usepackage{fancyvrb}

\usepackage{amsthm}
\usepackage{ amssymb }
\usepackage{diagbox}
\usepackage{url,  mathrsfs}
\usepackage{hyperref}
\usepackage{fancyhdr} 

\usepackage{url,  mathrsfs}
\usepackage{hyperref}
\usepackage{amsmath}
\usepackage{amsfonts}
\usepackage{graphicx,  epstopdf}
\usepackage{subfig}
\usepackage{color}
\usepackage{bbm}
\usepackage{subfloat}
\usepackage[draft]{fixme}
\usepackage{bm} 
\usepackage{bbm} 
\usepackage{epigraph}
\usepackage{endnotes}
\usepackage{ifpdf}
\usepackage{array}

\usepackage{multicol}
\usepackage{multirow}
\usepackage{graphicx}
\usepackage{tikz}

\usepackage{xcolor}

\newtheorem{theorem}{Theorem}[section]
\newtheorem{proposition}[theorem]{Proposition}
\newtheorem{lemma}[theorem]{Lemma}
\newtheorem{conjecture}[theorem]{Conjecture}

\newtheorem{corollary}[theorem]{Corollary}

\theoremstyle{definition}

\newtheorem{definition}[theorem]{Definition}

\newtheorem{example}[theorem]{Example}

\newtheorem{remark}[theorem]{Remark}

\newcommand{\B}{\mathcal{B}}
\newcommand{\p}{\mathcal{P}}
\newcommand{\R}{\mathbb{R}}
\newcommand{\nest}[2]{\mathcal{K}_{#1} #2} 

\newcommand{\conichull}{\operatorname{Cone}}

\newcommand{\danger}[1]{}

\dispositionformat{\bfseries}            

\headingformat{\large\MakeUppercase}   

\committeesize{5}

\chair{Nathan Reading}
\memberI{Seth Sullivant}
\memberII{Laura Colmenarejo}
\memberIII{Corey Jones}   
\memberIV{Stephen Terry}    

\student{Jordan Grady}{Almeter} 

\program{Mathematics}

\thesistitle{$\p$-graph associahedra and hypercube graph associahedra}

\usepackage[backend=biber, style=numeric]{biblatex}

\begin{document}
\frontmatter

\begin{abstract}

A graph associahedron is a polytope dual to a simplicial complex whose elements are induced connected subgraphs called tubes. Graph associahedra generalize permutahedra, associahedra, and cyclohedra, and therefore are of great interest to those who study Coxeter combinatorics.

For any graph, any proper vertex subset which induces a connected subgraph is called a tube, and any set of compatible tubes is called a tubing. The set of tubings for any graph is a simplicial complex which is dual to a simple polytope called the graph associahedron of that graph. The graph associahedron for a graph can be realized by repeatedly truncating certain faces of a simplex in accordance with tubes of that graph. The graph associahedron is further generalized by nestohedra, whose nested complexes are further generalized by nested complexes of semilattices.

This thesis characterizes nested complexes of simplicial complexes, which we call $\Delta$-nested complexes. From here, we can define $\p$-nestohedra by truncating simple polyhedra, and in more specificity define $\p$-graph associahedra, which are realized by repeated truncation of faces of simple polyhedra in accordance with tubes of graphs.

We then define hypercube-graph associahedra as a special case. Hypercube-graph associahedra are defined by tubes and tubings on a graph with a matching of dashed edges, with tubes and tubings avoiding those dashed edges. These simple rules make hypercube-graph tubings a simple and intuitive extension of classical graph tubings. We explore properties of $\Delta$-nested complexes and $\p$-nestohedra, and use these results to explore properties of hypercube-graph associahedra, including their facets and faces, as well as their normal fans and Minkowski sum decompositions. We use these properties to develop general methods of enumerating $f$-polynomials of families of hypercube-graph associahedra. Several of these hypercube-graphs correspond to previously-studied polyhedra, such as cubeahedra, the halohedron, the type $A_n$ linear $c$-cluster associahedron, and the type $A_n$ linear $c$-cluster biassociahedron. We provide enumerations for these polyhedra and others.
%
%
%

\end{abstract}

\makecopyrightpage

\maketitlepage


\begin{biography}
The author was born and raised in Chesapeake, Virginia before getting his Bachelor's in Mathematics at the College of William and Mary.
\end{biography}

\begin{acknowledgements}
Thanks to friends Wyn Nelson, Eric Geiger, and Hitesh Tolani for their friendship during my time at NCSU. Thank you to my advisor Nathan Reading for his patience and guidance during this time. Thank you to my family for all of their support. Thank you to my love, Richard Hess.
\end{acknowledgements}

\thesistableofcontents


\thesislistoffigures

\mainmatter

\chapter{INTRODUCTION}
\label{chap-one}


\section{Hypercube graph associahedra}

There exist a number of constructions related to polytopes known as associahedra, permutohedra, and nestohedra in the literature. Figure \ref{fig:generalizations} shows the relationships between existing constructions, and new constructions defined for this paper. In the traditional view of things, graph associahedra are polytopes which generalize associahedra and related polytopes. Then, nestohedra generalize graph associahedra, and generalized permutohedra generalize nestohedra. In addition, semilattice-nested complexes generalize nestohedra, with the caveat that these complexes are not polytopal. In this figure, we have drawn thick-outlined boxes around constructions defined in this paper. As we see, hypercube-graph associahedra are developed as a relative to graph associahedra, using $\p$-graph associahedra and $\p$-graph nestohedra as a common generalization. The dashed edges indicate a generalization involving duality, as the $\p$-nestohedron itself is not generalized by $\Delta$-nested complexes, but the $\p$-nestohedron's nested complex is generalized by $\Delta$-nested complexes, and the same go for $\p$-graph associahedra and $\Delta$-graph nested complexes.

\begin{figure}[htp]
\centering
\includegraphics[width=\textwidth]{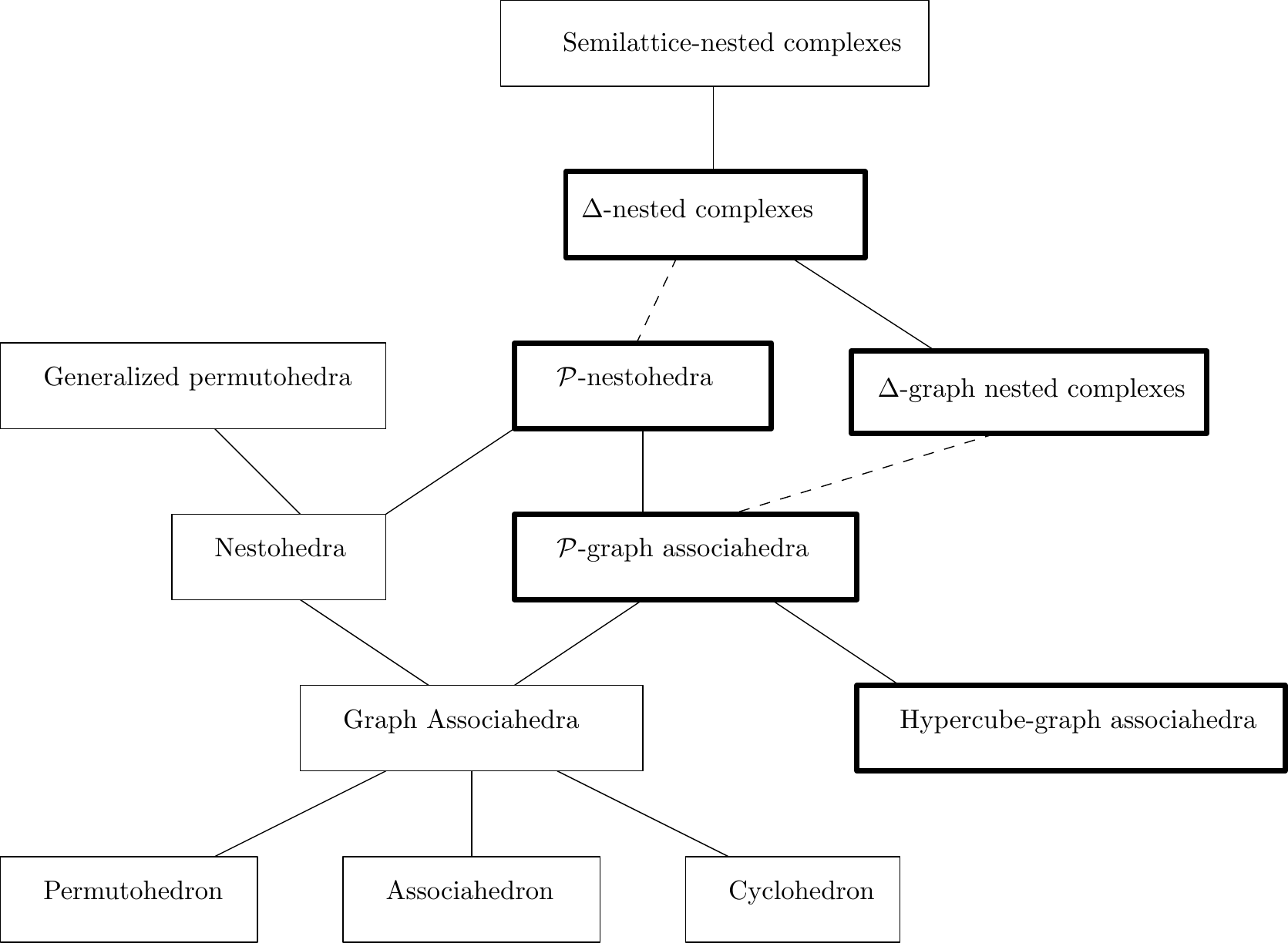}
\caption{A summary of classes of polyhedra and simplicial complexes in the literature and defined in this paper.}
\label{fig:generalizations}
\end{figure}

\subsection{Motivation and Background}

The \emph{associahedron} is a simple polytope first described by Dov Tamari, and later discovered independently by Jim Stasheff in the 1960s. The associahedron has many applications, and Jim Stasheff presents an overview of the history and the applications of the associahedron \cite[How I `met' Dov Tamari]{dovtamari}. Associahedra can be realized as compactifications of configuration spaces. The paper defining graph associahedra is focused on the compactification of moduli spaces \cite{carrdevadoss}, where the graph associahedra of the graphs of Coxeter groups are of particular interest for understanding Coxeter complexes.

Our motivation in defining hypercube graph associahedra was focused on the realization of graph associahedra as generalized permutahedra. Generalized permutahedra are polyhedra whose normal fans coarsen a type $A$ Coxeter fan, so there is a natural connection between the type $A_n$ root system and an $n$-dimensional graph associahedron. The hypercube-graph associahedron is a generalized type $B$ permutahedron, and if the graph associahedron is understood as a type $A$ object, then the hypercube-graph associahedron is understood as a type $B$ analogue. We had initially aimed to provide a generalization for all Coxeter types, but eventually realized that the construction we have defined is better understood as a generalization from the simplex to the broader case of any simple polyhedron, which is why we study $\p$-nestohedra and $\p$-graph associaehdra in this thesis.

The study of graph associahedra has spurred interest in several related polytopes. These polytopes include the \emph{multiplihedron} and the \emph{composihedron}. One other related polytope is the \emph{cubeahedron} defined in \cite{devadoss2011}. In that paper, it is shown that there are three cases where the moduli spaces of stable bordered marked surfaces are polytopal, corresponding to the associahedron, the cyclohedron, and the halohedron. The graph cubeahedron is introduced in order to characterize the halohedron, and it is found that the associahedron and the halohedron can be realized as cubeahedra. In Proposition \ref{prop:cubeahedra-hypercube}, we find that the cubeahedron is a special case of the hypercube-graph associahedron. By studying halohedra as hypercube-graph associahedra, we have been able to provide results enumerating the faces of the halohedron, as in Theorem \ref{thm:halohedron-formula}.

\subsection{Simple Polyhedra}

A \emph{simplicial complex} $\Delta$ is a collection of subsets, called \emph{faces}, of a base set $\mathcal{S}$ such that, if $X \subseteq \mathcal{S}$ is a face of $\Delta$ and $Y \subset X$, then $Y \in \Delta$.

A \emph{polyhedron} is a collection of points in a vector space $\R^n$ defined as the intersection of a finite set of inequalities of the form $\{x \in \R^n|c_i x \le b_i\}$. A \emph{polytope} is a bounded polyhedron.

A \emph{face} $F$ of a polyhedron $P$ is a set of points such that there exists an inequality $cx \le b$ such that $cx=b$ for all $x \in F$, and $cx \le b$ for all $x \in P$. A \emph{facet} of an $n$-dimensional polyhedron $P$ in $\R^n$ is an $(n-1)$-dimensional face of $P$.


A \emph{simple polyhedron} is a polyhedron where every nonempty co-dimension $d$ face is contained in exactly $d$ facets. We will write the family of simple polyhedra with the calligraphic $\p$.


Define the \emph{dual simplicial complex} of a simple polyhedron $P$ as a simplicial complex $\Delta(P)$ on the set of facets $\mathcal{S}$ of $P$. A set $X \subseteq \mathcal{S}$ is in $\Delta(P)$ if and only if the intersection of all facets in $X$ is nonempty. Note that the dual simplicial complex is isomorphic as a poset to the dual of the facial lattice of $P$, minus the empty element. 


\subsection{Graph associahedra}

This subsection recalls the notion of a graph associahedron, which was introduced in \cite{devadoss2011}. After this introduction, we will use the term \emph{classical graph associahedron} to specify this construction, as opposed to other constructions introduced in this thesis.

A \emph{graph} is an object $G=(V,E)$ where $V$ is a set of vertices and $E$ is a set of edges which are subsets of $V$ each containing $2$ elements. In some contexts, this is called a simple graph; in this thesis, we assume all graphs are simple. A \emph{subgraph} is a graph $(V',E')$ with $V' \subseteq V$ and $E' \subseteq E$. The \emph{induced subgraph} of a vertex subset $S \subseteq V$ of a graph $G$ is the maximal subgraph of $G$ on vertex set $S$. It consists of the graph on $S$ containing all edges in $E$ that are subsets of $S$. We notate this graph $G|_S$.

Given a connected graph $G$ on a vertex set $[n+1]=\{1,\ldots,n+1\}$, we define a \emph{tube} as any proper subset $t \subset [n+1]$ such that $G|_t$ is a connected induced subgraph of $G$. An example graph with several tubes is shown in Figure \ref{fig:ygraph1}.

\begin{figure}
\centering
\includegraphics[width=0.7\linewidth]{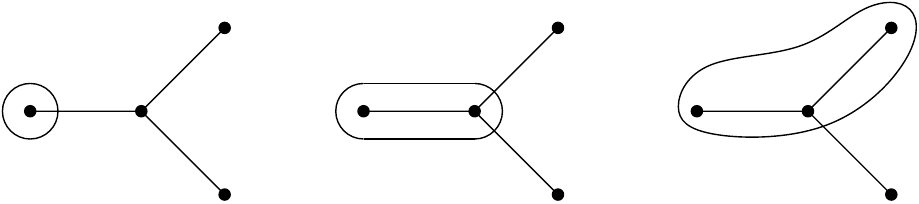}
\caption{Graph with tubes.}
\label{fig:ygraph1}
\end{figure}

We can define a \emph{tubing} as any collection of \emph{pairwise-compatible tubes}, where two tubes $t_1, t_2$ are compatible if and only if $t_1 \subset t_2, t_2 \subset t_1$, or $t_1, t_2$ are disjoint and not \emph{adjacent}; that is, there exist no edges between any vertex in $t_1$ and any vertex in $t_2$. Examples of compatible tubes and tubings are shown in Figure \ref{fig:ygraph2-intro}, and two pairs of incompatible tubes are shown in Figure \ref{fig:ygraph-new}.

\begin{figure}
\centering
\includegraphics[width=0.7\linewidth]{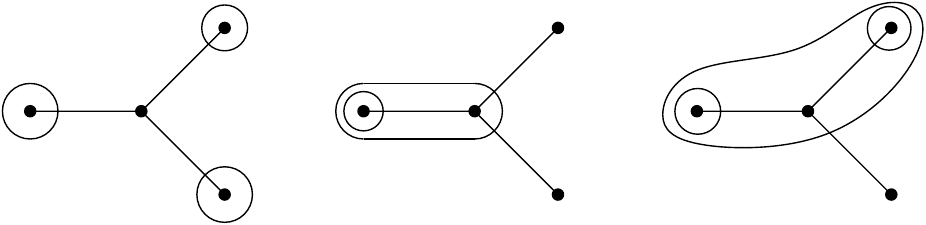}
\caption{Graph with tubings.}
\label{fig:ygraph2-intro}
\end{figure}

\begin{figure}
\centering
\includegraphics[width=0.45\linewidth]{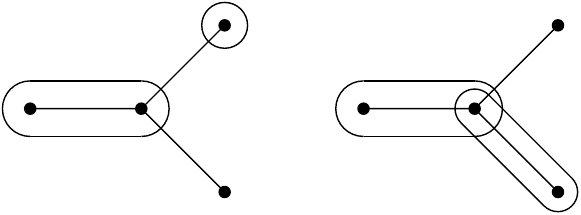}
\caption{Graph with pairs of incompatible tubes.}
\label{fig:ygraph-new}
\end{figure}

The collection of tubings for a graph $G$ is a simplicial complex, which we call the \emph{tubing complex} of $G$. This tubing complex is dual to a simple polytope, called the \emph{graph associahedron} of $G$.
One realization of the graph associahedron is obtained by repeatedly \emph{truncating} (i.e. slicing off) faces of a simplex, as in \cite{carrdevadoss}. 
In this construction, every vertex of $G$ is associated with a facet of an $n$-dimensional simplex. Nonempty ntersections of facets are then naturally associated with proper vertex subsets of $G$. The graph associahedorn is then constructed by truncating faces corresponding to tubes of $G$ in ascending order of dimension.

\subsection{Well-known graph associahedra}

This polyhedron is known as the graph associahedron because it is a generalization of the \emph{associahedron}. The associahedron is a polyhedron with many applications in combinatorics. Its $f$-vector is related to many enumeration problems. For example, it is associated with triangulations of an $(n+3)$-gon. The associahedron and its normal fan are heavily tied to type $A$ Coxeter combinatorics. The $n$-dimensional associahedron can be realized as the graph-associahedron of a path graph on $n+1$ vertices. The \emph{permutahedron}, also known as the type $A_n$ permutahedron, is the convex hull of the orbit of a generic point under action by the symmetry group of a simplex. The permutahedron in $n$ dimensions can be realized as the graph associahedron of a complete graph on $n+1$ vertices.

The \emph{cyclohedron} is another well-known graph associahedron. It can be realized as the graph-associahedron of a cycle graph on $n+1$ vertices. It is associated with symmetric triangulations of a $(2n+2)$-gon. It is also associated with type $B$ Coxeter combinatorics.

\subsection{Hypercube-Graph Associahedra}

In this thesis, we define \emph{$\p$-nestohedra} and \emph{$\p$-graph associahedra} in Chapter \ref{chap:meth}, and characterize their normal fans in Chapter \ref{chap:normal-fans}. We realize $\p$-graph associahedra by repeated truncation of simple polyhedra in a manner analogous to the realization of graph associahedra by repeated truncation of a simplex. However, we are specifically interested in the special case where this simple polyhedron is a hypercube. Hypercube-graph associahedra are described in detail in Chapter \ref{chap:hypercubes} of this thesis, using results from prior chapters to define them as $\p$-graph associahedra. In this introduction, we describe hypercube-graph associahedra without proofs or justifications.

A \emph{hypercube-graph} is a graph on vertex set $\pm[n]=\{-n,\ldots,-1,1,\ldots,n\}$, with dashed edges running between vertex pairs $\{i,-i\}$ for each $i \in [n]$. Figure \ref{fig:y-hypercubegraph1} shows a hypercube graph.

\begin{figure}[h]
\centering
\includegraphics[width=0.35\linewidth]{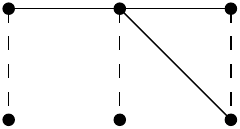}
\caption{Example hypercube-graph on $\pm[3]$.}
\label{fig:y-hypercubegraph1}
\end{figure}

A \emph{hypercube-graph tube} is a vertex subset $t \subset \pm[n]$ such that $t$ induces a connected subgraph which contains no dashed edges, as illustrated in Figure \ref{fig:y-hypercubegraph2}. A \emph{hypercube-graph tubing} is a collection of hypercube-graph tubes which satisfy the usual pairwise tube compatibility rules, but adds the condition that there may not be dashed edges between hypercube-graph tubes. Figure \ref{fig:y-hypercubegraph3} shows a pair of tubings, contrasted with several pairs of tubes which fail for various reasons--adjacency, intersection, or dashed edges between tubes.

\begin{figure}
\centering
\includegraphics[width=0.7\linewidth]{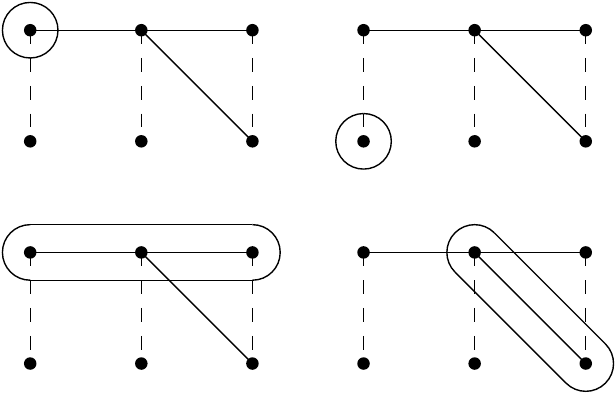}
\caption{Tubes of a hypercube-graph.}
\label{fig:y-hypercubegraph2}
\end{figure}

\begin{figure}
\centering
\includegraphics[width=0.7\linewidth]{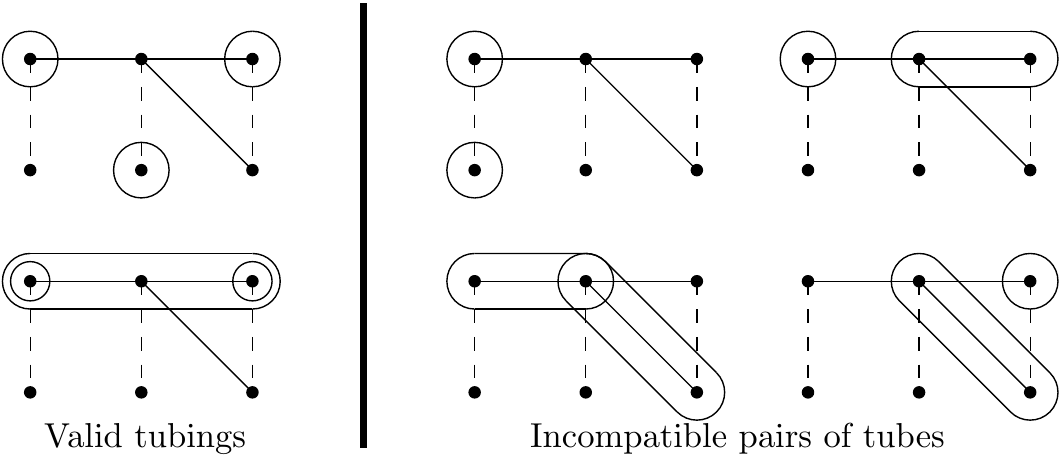}
\caption{Tubings of a hypercube-graph versus incompatible pairs of incompatible tubes}
\label{fig:y-hypercubegraph3}
\end{figure}

Hypercube-graph tubings form a simplicial complex, and this simplicial complex is dual to a simple polytope, called the \emph{hypercube-graph associahedron} of that hypercube-graph. The hypercube-graph associahedron can be constructed for a hypercube-graph $G$ by associating each pair of graph vertices $\{i,-i\}$ with a pair of opposing facets of a hypercube. 
This means that we can associate hypercube-graph tubes of $G$ with faces of the hypercube. A hypercube-graph associahedron of $G$ is constructed by repeatedly truncating faces of a hypercube associated with hypercube-graph tubes of $G$ in ascending order by dimension. Figure \ref{fig:y-hypercubegraph-truncation} shows the hypercube-graph associahedron of the graph in Figure \ref{fig:y-hypercubegraph1}. 

\begin{figure}
\centering
\includegraphics[width=0.4\linewidth]{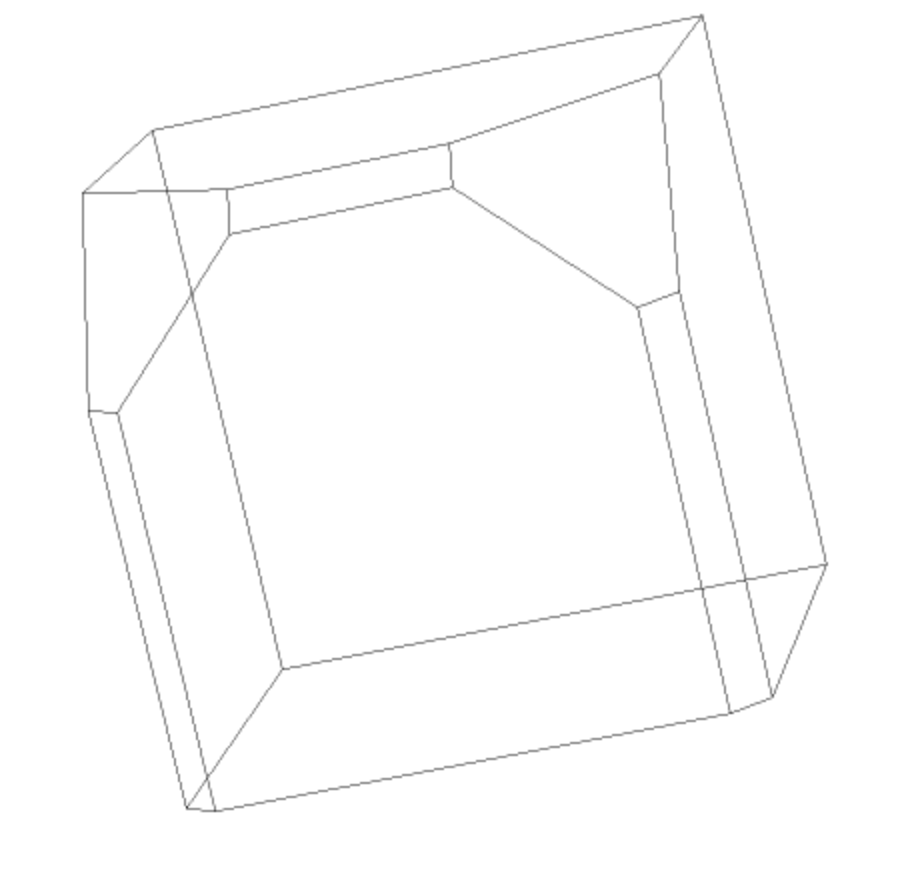}
\caption{Truncation of a hypercube to create a hypercube-graph associahedron}
\label{fig:y-hypercubegraph-truncation}
\end{figure}

\section{Gallery of hypercube-graph associahedra}

Many special cases of hypercube-graphs are studied in Chapter \ref{chap:res2}. Here, we provide an incomplete list of hypercube-graph associahedra studied in this thesis.

The maximal hypercube-graph contains all possible edges. Its hypercube-graph associahedron is shown in Figure \ref{fig:1-b3_permutahedron}, and when defined with symmetric truncations, is a type $B_3$ permutahedron.

\begin{figure}[h!]
\centering
\includegraphics[width=.4\linewidth]{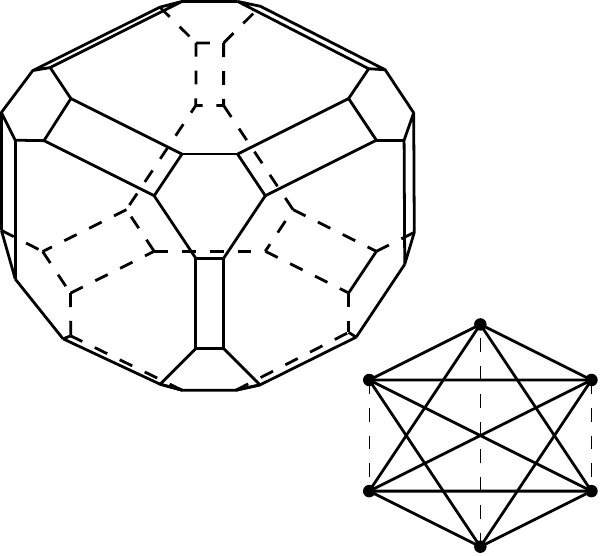}
\caption{Type $B_3$ permutahedron as a hypercube-graph associahedron.} \label{fig:1-b3_permutahedron}
\end{figure}
\begin{figure}[h!]
\centering
\includegraphics[width=.4\linewidth]{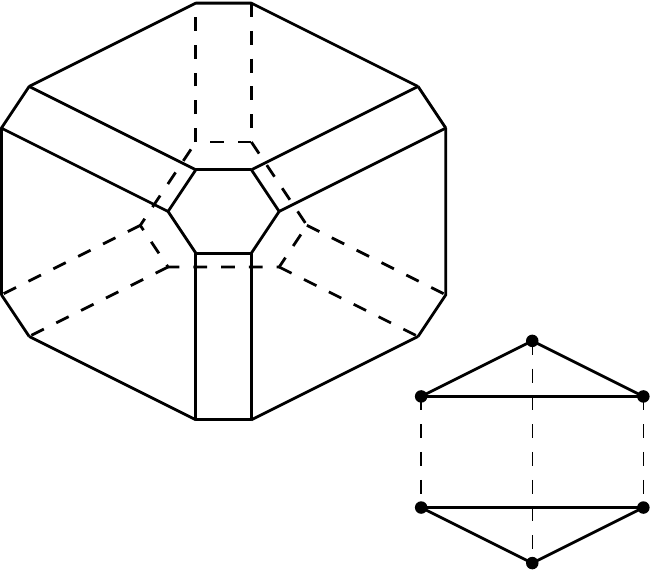}
\caption{Type $A_3$ permutahedron as a hypercube-graph associahedron.} \label{fig:1-a3_permutahedron}
\end{figure}
\begin{figure}[h!]
\centering
\includegraphics[width=.4\linewidth]{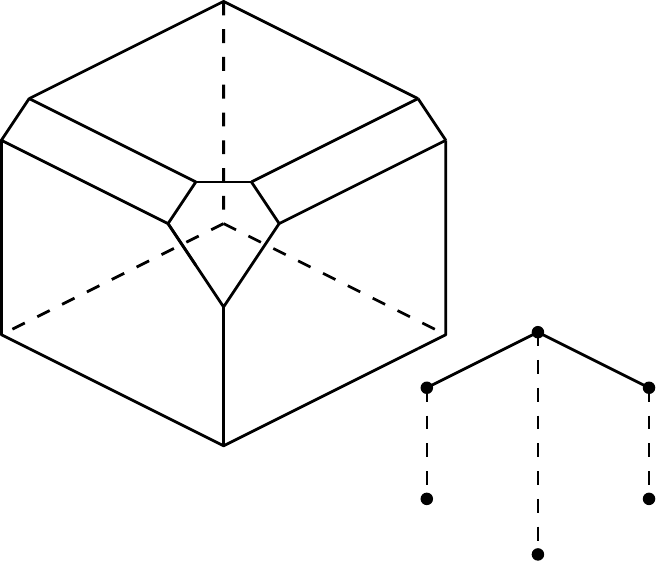}
\caption{Type $A_3$ associahedron as a hypercube-graph associahedron.} \label{fig:1-a3_associahedron}
\end{figure}
\begin{figure}[h!]
\centering
\includegraphics[width=.4\linewidth]{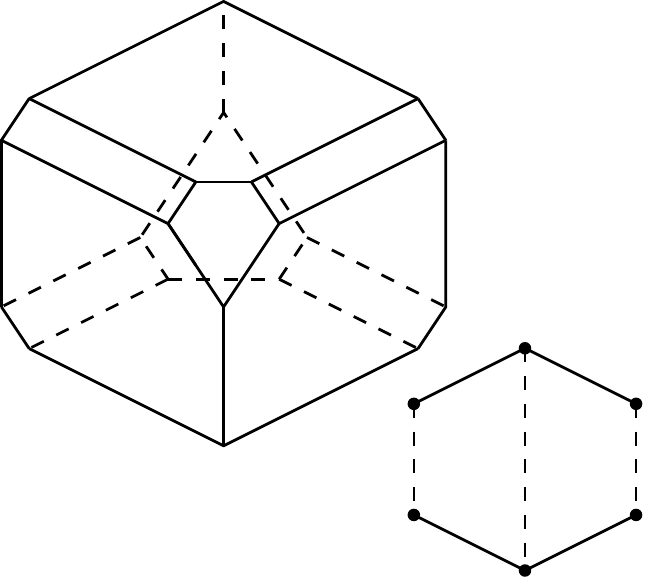}
\caption{Type $A_3$ linear biassociahedron as a hypercube-graph associahedron.} \label{fig:1-a3_biassociahedron}
\end{figure}
\begin{figure}[h!]
\centering
\includegraphics[width=.4\linewidth]{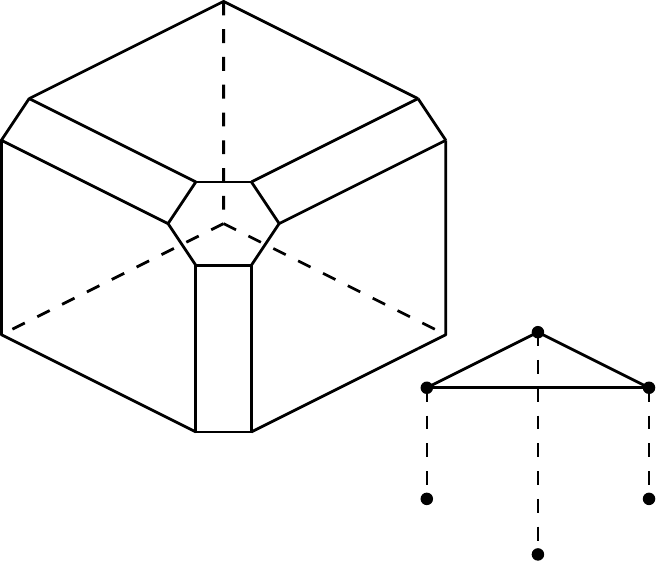}
\caption{$3$-dimensional halohedron.} \label{fig:1-halohedron}
\end{figure}

Figure \ref{fig:1-a3_permutahedron} depicts the hypercube-graph associahedron of a hypercube-graph consisting of a complete graph on positive vertices and a complete graph on negative vertices. As proven in Proposition \ref{prop:type-A-permutahedron}, this hypercube-graph associahedron is isomorphic to a type $A_n$ permutahedron.

Figure \ref{fig:1-a3_associahedron} shows the hypercube-graph associahedron of a hypercube-graph consisting of a path on positive vertices. This polyhedron has been studied before as a path cubeahedron in \cite{devadoss2011}, where it is proven to be isomorphic to an associahedron. We further prove in Proposition \ref{prop:linear-cluster-fan} that it is normal to the linear $c$-cluster fan of type $A_n$.

Figure \ref{fig:1-a3_biassociahedron} shows the hypercube-graph associahedron of a hypercube-graph consisting of a pair of paths on positive and negative vertices. We see that it is similar to the hypercube-graph associahedron shown in Figure \ref{fig:1-a3_associahedron}, but with symmetric truncations made.

Figure \ref{fig:1-halohedron} shows the halohedron as the hypercube-graph associahedron of a hypercube graph consisting of a cycle on positive vertices. Note that the cycle on 3 vertices is equal to the complete graph on 3 vertices, so this polyhedron bears resemblance to the stellahedron in 3 dimensions, but the cyclic pattern becomes apparent in higher dimensions. The halohedron was described as a cubeahedron in \cite{devadoss2011}, and an enumeration of its vertices or $f$-vectors was unknown. Theorem \ref{thm:halohedron-formula} provides an enumeration of the faces of the halohedron.

\section{The relationship between hypercube-graph associahedra and Coxeter combinatorics}

The symmetry group of the $n$-dimensional simplex is called the \emph{type $A_n$ Coxeter group}, and the hyperplane arrangement of reflections in this group creates a fan which we can call the type $A_n$ Coxeter fan. Polyhedra whose normal fans coarsen this fan are called \emph{generalized permutahedra} \cite{postnikov}.

Notably, graph associahedra are generalized permutohedra. In particular, two important type $A_n$ polyhedra can be realized as graph associahedra: the type $A_n$ permutohedron, and the type $A_n$ associahedron. Both are more commonly referred to as simply the permutohedron or the associahedron, respectively, but permutahedra and associahedra exist for other root systems.

The initial goal of this project was to generalize graph associahedra for any Coxeter type, or any root system. We do find that hypercube graph associahedra, realized as in Definition \ref{def:standard-cut}, have normal fans which coarsen the type $B_n$ Coxeter fan, which is generated by the symmetry group of a hypercube. We also find that the type $B_n$ permutohedron can be realized as a hypercube-graph associahedron. We have not been able, however, to find a type $B_n$ associahedron realized as a hypercube graph associahedron, and conjecture in Conjecture \ref{conj:no-cyclohedra} that the type $B_n$ associahedron cannot be realized as a hypercube-graph associahedron for dimensions $\ge 4$.

We note that cones in any fan can be ordered by a linear functional to define a \emph{poset of regions.} When that fan is a Coxeter fan of a finite Coxeter group of type $W$, that poset of regions is isomorphic to the Coxeter weak order of type $W$. When one fan $F$ coarsens a Coxeter fan of type $W$, that fan induces a map from the Coxeter weak order of type $W$ to the poset of regions of $F$. These maps prove interesting, as in the work of \cite{barnard2018lattices} in the classical graph associahedron case, and we believe that this is an interesting field of study for future research in the hypercube-graph associahedron case.

\section{Other results}

In this paper, we explore several other hypercube-graphs. For instance, we conjecture that the Pell Graph, defined in Subsection \ref{sub:pell-graphs}, has a poset of maximal tubings isomorphic to the lattice of sashes defined in \cite{sashes}. Furthermore, in our research we have found families of hypercube-graph associahedra which are apparently not isomorphic to known polyhedra, but which present interesting symmetries or have interesting $f$-polynomials. This includes the twisted cycle and twisted path graphs defined in Subsection \ref{sub:twisted cycle and path}, the families of double cubeahedra defined in Subsection \ref{sub:doublecubeahedra}, and the near-double path hypercube graph associahedron defined in Subsection \ref{sub:near double path definition}. We also define a complex of subtrees in \ref{sub:subtrees} which can be realized as the $\Delta$-graph nested complex of a line graph and its graphic matroid.

\chapter{$\p$-graph associahedra}
\label{chap:meth}

Semilattice-nested complexes, classical nestohedra, and classical graph associahedra were introduced for in \cite{feichtner2003}, \cite{postnikov}, and \cite{carrdevadoss}, and there have been multiple related polyhedra and generalizations defined since. This chapter defines three new concepts. Section \ref{sec:simplicial-complex-nested-complexes} defines \emph{$\Delta$-nested complexes}, the special case for semilattice-nested complexes when the underlying semilattice is a simplicial complex. Section \ref{sec:graph-nested-complexes} defines \emph{graph nested complexes}, generalizing work on graph associahedra for $\Delta$-nested complexes. Finally, Section \ref{sec:nestohedra} defines \emph{$\p$-nestohedra}, simple polyhedra which are dual to certain polyhedral $\Delta$-nested complexes and which are obtained by repeated truncation of polyhedral faces, as well as $\p$-graph associahedra, a special case.

\section{Nested complexes of simplicial complexes}\label{sec:simplicial-complex-nested-complexes}

This section summarizes some basic results on the nested complexes of simplicial complexes. Nested complexes are defined for semilattices in \cite{feichtner2003}. The posets defined by simplicial complexes are semilattices, and so we can use the semilattice definition of nested complexes to characterize building sets and nested complexes for the special case of simplicial complexes.

%

\subsection{Simplicial complexes}

A poset $L$ is called a \emph{meet-semilattice} if, for every pair of elements $x,y \in L$, there exists a unique greatest lower bound called the \emph{meet} and denoted $x \wedge y$. In this thesis, we will simply call these semilattices. A theory of semilattice building sets is developed in \cite{feichtner2003}, and is reiterated here briefly before we focus on the special case where $L$ is a simplicial complex. All semilattices have a minimal element, written in this thesis as $\hat{0}$.


For a subset of elements $B$ of a semilattice $L$ and an element $x \in B$, define $B_{\le x}=\{y\in B| y\le x\}$. Define $\max B_{\le x}$ to be the set of maximal elements in $B_{\le x}$. For two elements $a \le b$ in a poset $P$, define the interval $[a,b]=\{c \in P|a \le c \le b\}$.

The following definition can be intuitively understood as follows: if $B$ is a building set, then for every element $x \in L \backslash \{\hat{0}\}$, the interval $[\hat{0},x]$ can be decomposed into the product of intervals $[\hat{0},y]$ for all $y \in \max B_{\le x}$. Proposition \ref{complex-buildingset} provides a much simpler case, when $L$ is isomorphic to a simplicial complex, and that is the ony case we will need in this thesis.


	\begin{definition}[\cite{feichtner2003}] \label{def:semilatticebuildingset}
		Let $L$ be a semilattice. A subset $B$ of $L \backslash \{\hat{0}\}$ is called a \emph{semilattice building set} of $L$ if for any $x \in L \backslash \{\hat{0}\}$ and $\max B_{\le x} = \{x_1,\ldots, x_k\}$ there is an isomorphism of posets

	\[
		\varphi_x: \prod_{j=1}^k [\hat{0},x_j] \to [\hat{0},x]
\]

	with $\varphi_x(\hat{0}, \ldots, x_j, \ldots, \hat{0}) = x_j$ for $j=1, \ldots, k$.
	\end{definition}

	\begin{definition}[\cite{feichtner2003}]\label{latticebuildingset}
		Let $L$ be a meet-semilattice and $B$ a building set of $L$. A subset $N$ in $B$ is called \emph{semilattice nested} if, for any set of incomparable elements $N'=\{x_1, \ldots, x_t\} \subseteq N$ with $|N'|\ge 2$, the join $x_1 \vee \cdots \vee x_t$ exists and does not belong to $B$.
	\end{definition}

	\begin{definition}
		A \emph{simplicial complex} $\Delta$ on a base set $S$ is a family of subsets of $S$, or \emph{faces}, with the condition that every singleton set $\{s\}$ for $s \in S$ is a face of $\Delta$, and such that if $F$ is a face of $\Delta$ and $F' \subset F$, then $F'$ is a face of $\Delta$.
	\end{definition}

	The \emph{rank} of a simplicial complex is the size of its largest face. A simplicial complex is \emph{pure} if all of its maximal faces are of the same size. \danger{These are new definitions added here!}

	\begin{definition} \label{def:semilatticenestedcomplex}
		The \emph{$B$-nested complex} of a semilattice $L$ is the set of all semilattice nested subsets of elements of a lattice $L$ for a given building set $B$. We will use the notation $\mathcal{N}(B,L)$ to refer to this simplicial complex.
	\end{definition}

	Faces of a simplicial complex admit a partial order based on inclusion, and we know the resulting poset is a meet-semilattice, meaning we can define semilattice building sets on these posets.

	Define $\B_{\subseteq S}$ to be the set of sets $\{T \in \B| T \subseteq S\}$. For an ordered set $\B$ with relation $\leq$, define $\max (\B)$ to be the set of maximal elements in $\B$; that is, $\{S \in \B|S \not< T \forall T \in \B\}$. The following lemma comes very easily from Definition \ref{def:semilatticebuildingset}.

	\begin{lemma} \label{lemma:partition-lemma}
		Given a simplicial complex $\Delta$ on a base set $\mathcal{S}$, a set of faces $\B$ of $\Delta \backslash \hat{0}$ is a building set if and only if:
	\begin{enumerate}
	\item For each element $s \in \mathcal{S}$, the set $\{s\}$ is contained in $\B$.
			\item For every nonempty set $S$ in $\Delta$, $\max (\B_{\subseteq S})$ is a partition of $S$.
	\end{enumerate}
\end{lemma}

\begin{proof}
It is simple to show that if $L$ is a semilattice and $B$ is a semilattice-building set, then all atomic elements of $L$ must be in $B$. If $x$ is an atom in $L$ and $x \notin B$, then $ \max (B_{\le x})$ must be empty, making $\prod_{y \in \max B_{\le x}} [\hat{0},y]$ an empty product. This must contradict Definition \ref{def:semilatticebuildingset}, as $[\hat{0},x]$ is a non-singleton poset. The atomic elements of a simplicial complex are the singleton sets, and so simplicial complex building sets must contain all singleton sets.

In a simplicial complex $\Delta$, the interval $[\emptyset,S]$ for any set $S$ is equal to the Boolean lattice of subsets of $S$. We then find that $\B_{\subseteq S}$ must contain all singleton subsets of $S$. From this, we see that the union of all sets in $\max \B_{\subseteq S}$ is equal to $S$. We also find that the product of $[\emptyset,S_1],\ldots,[\emptyset,S_k]$ is equal to a Boolean lattice on $|S_1|+\cdots+|S_k|$ elements. From these, we can deduce that a set $\B$ is a building set in $\Delta$ if and only if, for every $S \in \Delta$, the set $\max (\B_{\subseteq S})$ is a partition of $S$.
\end{proof}

We now wish to prove that this is equivalent to another, very useful definition of building sets, and one which is more familiar to those who have read the classical definition in \cite{postnikov}. This proposition looks very similar to Lemma \ref{lemma:partition-lemma}, but the second condition is different.

	\begin{proposition} \label{complex-buildingset}
		Given a simplicial complex $\Delta$ on a base set $\mathcal{S}$, a set of faces $\B$ of $\Delta \backslash \hat{0}$ is a building set if and only if:
		\begin{enumerate}
			\item For each element $s \in \mathcal{S}$, the set $\{s\}$ is contained in $\B$.
			\item For two sets $S_1, S_2 \in \B$ where $S_1 \cap S_2 \ne \emptyset$, if $S_1 \cup S_2 \in \Delta$, then $S_1 \cup S_2 \in \B$.
\end{enumerate}

\end{proposition}

\begin{proof}

Lemma \ref{lemma:partition-lemma} has already established conditions for a set $\B$ to be a building set. We intend to show that if $\{s\} \in \B$ for every $s \in \mathcal{S}$, then condition (2) of Lemma \ref{lemma:partition-lemma} is equivalent to condition (2) of Proposition \ref{complex-buildingset}. Refer to these conditions as the partition condition and the intersection-union condition.

We first prove that the partition condition implies the intersection condition. If $\max \B_{\subseteq S}$ is a partition of $S$ for every nonempty face $S \in \Delta$, then consider $S=S_1 \cup S_2$ for $S_1, S_2 \in \B$ with nonempty intersection. If $S_1 \cup S_2 \in \Delta$, then $\B_{\subseteq S}$ contains both $S_1$ and $S_2$. This then means that $\max \B_{\subseteq S}$ must also contain the union $S_1 \cup S_2$. As a result, if $\B$ is a building set, then the intersection-union condition holds.

Now consider the case where $\B$ containing $\{s\}$ for each $s\in\mathcal{S}$ is a set for which the intersection-union condition holds, and assume that the partition condition does not hold. This means there exists a nonempty set $S \in \Delta$ such that $\max \B_{\subseteq S}$ is not a partition of $S$. We know that the union of all sets in $\max \B_{\subseteq S}$ is equal to $S$, so this can only fail to be a partition if there exist two sets $S_1, S_2 \in \max \B_{\subseteq S}$ such that $S_1 \cap S_2 \ne \emptyset$. However, by the intersection-union condition, the union $S_1 \cup S_2$ must be in $\B$. This provides a contradiction, as now $S_1 \cup S_2 \in \max \B_{\subseteq S}$, and either $S_1$ or $S_2$ in $\max \B_{\subseteq S}$ is a proper subset of $S_1 \cup S_2$, proving one of the subsets to be non-maximal. As a result, when $\{s\}\in\B$ for all $s \in \mathcal{S}$, the intersection-union condition and partition condition are equivalent. This then proves the proposition.
\end{proof}



Nested sets of semilattices are defined in Definition \ref{latticebuildingset}. Here, we characterize nested sets of simplicial complexes.

\begin{definition}
Two sets $S_1, S_2$ have \emph{non-trivial intersection} if $S_1 \not\subseteq S_2$, $S_2\not\subseteq S_1$, and $S_1 \cap S_2 \ne \emptyset$.
\end{definition}

 As a result, we can characterize nested sets of simplicial complexes as follows.

\begin{proposition} \label{complex-nestedset}
Consider a simplicial complex $\Delta$ and $\Delta$-building set $\B$. Any subset $N \subseteq \B$ is nested if and only if both conditions hold: \begin{enumerate}\item For every subset $N'\subseteq N$ with $|N'|\ge 2$ and every pair of sets in $N'$ is disjoint, $\bigcup N'$ is in $\Delta$ and not in $\B$.\item No two sets $S_1, S_2 \in N$ have nontrivial intersection.\end{enumerate}
	\end{proposition}

\begin{proof}

Call the first condition the disjoint union condition, and call the second condition the nontrivial intersection condition. We note that Definition \ref{def:semilatticenestedcomplex} when applied to a simplicial complex is equivalent to saying that a set $N$ is nested if and only if, for every subset of incomparable sets $N'=\{S_1,\ldots,S_k\} \subseteq N$ with $k \ge 2$, the union $\bigcup_{i=1}^k S_i$ is in $\Delta$ and not in $\B$. Note that two sets are incomparable if one is not contained in the other as a subset.

If a set $N$ is $\B$-nested, then the disjoint union condition holds for $N$, as every set $N'$ of disjoint sets is incomparable. In addition, if two sets $S_1, S_2$ have nontrivial intersection, then they are incomparable, and $S_1 \cup S_2$ is in $\B$ according to building set properties, which is a contradiction. As a result, no two sets in a nested set have nontrivial intersection.

Now say that $N$ satisfies the disjoint union condition and the nontrivial intersection condition; we wish to prove that $N$ is a nested set. Consider a subset $N'=\{S_1,\ldots,S_k\} \subseteq N$ such that $|N'|\ge 2$, and the union $\bigcup N'$ either is not in $\Delta$ or is in $\B$. We note that by the nontrivial intersection condition, no two of these sets have nontrivial intersection, and so the sets $S_1, \ldots, S_k$ all are disjoint. Now by the disjoint union condition, $\bigcup N'$ is in $\Delta$ but not in $\B$. As a result, we have proven that these two conditions are equivalent to saying that $N$ is $\B$-nested.
\end{proof}

%
%
%
%
%
%
%
%
%
%
%
%



\subsection{Links of the nested complexes of simplicial complexes}

In \cite{carrdevadoss}, every facet of a graph associahedron is associated with the graph associahedron of a new graph, called the reconnected complement. We are compelled to find a similar result for $\Delta$-nested complexes. We note that each face of a simple polyhedron is dual to the link of a set in the dual simplicial complex of that simple polyhedron. As a result, finding the links of sets in $\Delta$-nested complexes is a natural generalization of the Carr-Devadoss result, as well as the results we prove later in \ref{sub:Pcomplexlinks}.




\begin{definition}
Given a set $S$ in a simplicial complex $\Delta$, the \emph{link} of a face $S$ in $\Delta$, denoted $\Delta/S$, is the subcomplex of $\Delta$ defined by

\[
	\Delta/S = \{X| X \in \Delta, X \cup S \in \Delta,X \cap S = \emptyset\}.
\]
An equivalent but lesser-used definition is

\[
	\Delta/S = \{Y\backslash S| \; Y \in \Delta, S \subset Y\}.
\]

\end{definition}

\begin{definition}
For a simplicial complex $\Delta$ and a set $S$ in a $\Delta$-building set $\B$, the \emph{building set pseudolink} of $S$ in $\B$, denoted $\B/S$, is defined as
\[
	\B/S = \{X|X\in\B,X\cup S \in \Delta, X \cap S = \emptyset\} \cup \{Y\backslash S|\; Y \in \B, S\subset Y\}.
\]
\end{definition}


This definition shows parallels between the two definitions of simplicial complex links. Note that this is not necessarily a disjoint union.

\begin{lemma}
The pseudolink $\B/S$ of an element $S$ in a $\Delta$-building set $\B$ is a building set of the link $\Delta/S$.
\end{lemma}

\begin{proof}
Every atom $\{x\}$ of $\Delta/S$ is an atom in $\Delta$, and so must be in $\B$. From this, we see $\{x\} \cup S$ is in $\Delta$, and so $\{x\}$ is in $\B/S$. As a result, ever atom of $\Delta/S$ is in $\B/S$, and so $\B/S$ satisfies the  first condition of simplicial building sets. Now, we must prove that if $I, J$ are sets in $\B/S$ with $I \cap J \ne \emptyset$, then $I \cup J \in \B/S$ if $I\cup J \in \Delta/S$.

Note that a set $X\in \Delta/S$ is in $\B/S$ if and only if $X \in \B$ or $X \in \B/S$. Assume $I,J \in \B/S$, $I\cap J \ne \emptyset$, and $I \cup J \in \Delta/S$. We know \mbox{$I \cup J \cup S \in \Delta$}. If $I\cup S$ is in $\B$, then $(I\cup S)$ has nontrivial intersection with both sets $J$ and $J \cup S$, which are both in $\B$. As a result, $I \cup J \cup S \in \B$, and $I \cup J \in \B/S$. By the same logic, if $J \cup S \in \B$, then $I \cup J \in \B/S$. As a result, we can consider the case where $I \cup S, J \cup S \notin \B$. This means that $I, J \in \B$, and because of their nonempty intersection, $I \cup J \in \B$. As a result, either $I \cup J \cup S$ or $I \cup J$ is in $\B$, which implies $I \cup J \in \B/S$. Therefore, we have satisfied the two conditions of building sets, and $\B/S$ is a $\Delta/S$-building set.
\end{proof}

Having shown that $\B/S$ is a building set, we now know that its nested complex $\mathcal{N}(\B/S,\Delta/S)$ exists, and state the following proposition about it.

\begin{proposition} \label{prop:buildingsetlink}
The nested complex $\mathcal{N}(\B/S,\Delta/S)$ is isomorphic to the subcomplex of $\mathcal{N}(\B,\Delta)$ consisting of sets $N\in \mathcal{N}(\B,\Delta)$ such that $N$ contains no subset of $S$ as an element, and $N \cup \{S\}$ is $\B$-nested.
\end{proposition}

\begin{proof}

Define $\mathcal{M}$ to be the simplicial complex of $\B$-nested sets such that, for each $N \in \mathcal{M}$, no subset of $S$ is an element of $N$, and $N \cup \{S\}$ is $\B$-nested. Define a map $\phi(T)=T\backslash S$ for all $T \in \B$. We will first prove that this map is a bijection between the base set of $\mathcal{M}$ and the base set of $\mathcal{N}(\B/S,\Delta/S)$. We will then extend this map to a map between faces of $\mathcal{M}$ and faces of the complex $\mathcal{N}(\B/S,\Delta/S)$, with $\phi(N)=\{\phi(T)|T \in N\}$, and prove that this map is an isomorphism. 


We first prove that $\phi$ is injective on the base set of $\mathcal{M}$. Say $T_1, T_2 \in \B$ and $\{T_1,S\}, \{T_2,S\}$ are both nested, and $T_1\backslash S=T_2 \backslash S$. As a result, $\phi(T_1)=\phi(T_2)$. In order to prove that $\phi$ is injective, we must prove that $T_1=T_2$, which we do by investigating different cases. If $T_1 \cap S=\emptyset$ and $T_2 \cap S=\emptyset$, then $T_1=T_2$. Assume then that one of the sets intersects with $S$; without loss of generality, assume $T_1 \cap S \ne \emptyset$. Because $\{T_1,S\}$ is nested, by Proposition \ref{complex-nestedset}, the two sets can only have trivial intersection, and we find $S \subseteq T_1$.

In the case where $S \cap T_2\ne\emptyset$, we find that $T_1 \cap T_2 \ne \emptyset$, and $S \subseteq T_1 \subseteq T_2$ or $S \subseteq T_2 \subseteq T_1$. In either case, we find that $T_1\backslash S=T_2\backslash S$ implies $T_1=T_2$. Now as the only remaining case, we must assume that $S \subseteq T_1$ and $T_2 \cap S=\emptyset$. In this case, we find $T_1\backslash S=T_2 \backslash S$ implies $T_1$ is the disjoint union of $S$ and $T_2$. This however is a contradiction, as by Proposition \ref{complex-nestedset}, two disjoint sets $\{S,T_2\}$ in a $\B$-nested set cannot have a union in $\B$, but $S\cup T_2=T_1 \in \B$. As a result, we have proven that $\phi$ is injective into $\B/S$.


Next we prove that $\phi$ is surjective onto $\B/S$, by proving that for every set $T\in\B/S$ there is a preimage in $\mathcal{M}$. If $T \in \B/S$, then there exists at least one set $Y \in \B$ such that $Y \backslash S=T$. We are presented with three cases: either $S,Y$ are disjoint, $S \subseteq Y$, or $S, Y$ have nontrivial intersection. If $S, Y$ are disjoint, then $T=Y$, and $\phi(Y)=T$. Now consider the case that $ S \subseteq Y$. We find very easily that $\{S,Y\}$ is nested, and so $Y$ is in the base set of $\mathcal{M}$, and $\phi(Y)=T$. Finally, consider the case that $S,Y$ have nontrivial intersection. This means that $\{S,Y\}$ is not nested. However, we know from Proposition \ref{complex-buildingset} that $S \cup Y $ is in $\B$, and $\{S \cup Y,S\}$ must be $\B$-nested, meaning $S \cup Y$ is in the base set of $\mathcal{M}$ and $\phi(S\cup Y)=T$.

As a result, $\phi$ is an isomorphism between the base set of $\mathcal{M}$ and $\B/S$, the base set of $\mathcal{N}(\B/S,\Delta/S)$. We will now prove that $N$ in $\mathcal{M}$ is $\B$-nested if and only if $\phi(N)$ is $\B/S$-nested, proving $\phi$ is an isomorphism and the two complexes are isomorphic.

We will prove the forward direction first. Consider the case where $N$ is in $\mathcal{M}$, and we will prove $\phi(N)$ is $\B/S$-nested. Consider two sets $T_1, T_2$ in $N$. If $T_1, T_2$ are disjoint, then $\phi(T_1), \phi(T_2)$ are disjoint. If $T_1 \subseteq T_2$, then $(T_1 \backslash S) \subseteq (T_2 \backslash S)$. As a result, the set $\phi(N)$ has no nontrivial intersecting pairs, satisfying the second condition of \ref{complex-nestedset}. Now we will prove that $\phi(N)$ satisfies the first condition of Proposition \ref{complex-nestedset}, that for any subset of $\phi(N)$, with cardinality at least two and containing disjoint subsets, their union cannot be in $\B/S$.

If $N'=\{T_1,\ldots,T_k\} \subseteq N$ and $\phi(N')$ is a set of disjoint sets with $|N'|\ge 2$, we find that $\bigcup \phi(N')=\bigcup N' \backslash S$. Consider a case where $N$ is nested and in $\mathcal{M}$ but $\bigcup \phi(N') \in \B/S$, meaning $\phi(N)$ is not nested. Note that $\bigcup N'\backslash S$ is in $\B/S$ only if $\bigcup N'$ or $\bigcup N'\backslash S $ is in $\B$. We know that $\bigcup N'$ cannot be in $\B$ because $N$ is $\B$-nested, so assume $S \subseteq \bigcup N'$. Note as well that $S$ can only be a subset of one element of $N'$; otherwise, $\phi(N')$ is not disjoint. Without loss of generality, write $S \subset T_1$. We now find that $\bigcup N' \backslash S \in \B$, and $T_1 \in \B$ both have nontrivial intersection, and the union of these two sets must be in $\B$. However, their union is $\bigcup N'$, which cannot be in $\B$ if $N$ is nested, so this is a contradiction. As a result, $\phi(N)$ must be $\B/S$-nested for every set $N \in \mathcal{M}$.

Now consider a case where $N$ is a subset of the base set of $\mathcal{M}$. We now wish to prove that if $\phi(N)$ is $\B/S$-nested, then $N$ is $\B$-nested, and therefore in $\mathcal{M}$.

If $\phi(N)$ is $\B/S$-nested, consider two sets $T_1, T_2 \in N$. If $T_1, T_2$ have nontrivial intersection, then one of the following is true: either $\phi(T_1), \phi(T_2)$ have nontrivial intersection with each other, $T_1$ or $T_2$ have nontrivial intersection with $S$, or $\phi(T_1),\phi(T_2)$ are disjoint. The first case is a contradiction because $\phi(N)$ is $\B/S$-nested, and no two sets in a nested set can have nontrivial intersection. The second case implies that either $\{T_1,S\}$ or $\{T_2,S\}$ is not nested, but when defining the base set of $\mathcal{M}$ we assumed $\{S,T\}$ was nested for each set $T$ in the base set, so this is a contradiction. The third case implies that $T_1 \cap T_2 \ne \emptyset$ and $(T_1 \backslash S) \cap (T_2 \backslash S) = \emptyset$, and because neither set has trivial intersection with $S$, implies that $T_1 \cap T_2 = S$. This then means that $T_1 \cup T_2 \in \B$, and $T_1 \cup T_2 \backslash S \in \B/S$, which is a contradiction because $\phi(N)$ is $\B/S$-nested.

If $\phi(N)$ is $\B/S$-nested, consider a disjoint set $N'=\{T_1,\ldots,T_k\} \subseteq N$ with $|N'|\ge 2$. We find that $\phi(N')$ is a set of disjoint sets, and is a subset of $\phi(N)$. If $\bigcup N' \in \B$, we find that $\bigcup \phi(N') \in \B/S$, which is a contradiction because $\phi(N)$ is nested.

 As a result, we have proven that the map $\phi$ is a bijection between the base sets of $\mathcal{M}$ and $\mathcal{N}(\B/S,\Delta/S)$, and it extends to an isomorphism between the two simplicial complexes.
\end{proof}

This isomorphism is used in Proposition \ref{prop:singletonlink}, but first we must state some definitions for simplicial complexes.

\begin{definition}
The \emph{restriction} of a $\Delta$-building set $\B$ to a simplicial subcomplex $\Delta' \subseteq \Delta$ is the set $\B \cap \Delta' = \{S \in \B|S \in \Delta'\}$.
\end{definition}

It is trivial to apply Proposition \ref{complex-buildingset} to find that the restriction of a $\Delta$-building set $\B$ to a subcomplex $\Delta'$ is a $\Delta'$-building set.'

\begin{proposition} \label{prop:restriction}
If $\Delta' \subseteq \Delta$ and $\B$ is a $\Delta$-building set, then $N$ is $(\B \cap \Delta')$-nested if and only if $N$ is $\B$-nested and $\bigcup N \in \Delta'$.
\end{proposition}

\begin{proof}
This follows from Proposition \ref{complex-nestedset}. If $N$ is $\B\cap \Delta'$-nested, then there is no subset $N'\subseteq N$ of disjoint sets with $|N'|\ge 2$ such that $\bigcup N' \in (\B \cap \Delta')$, and since $\bigcup N' \in \Delta'$, we know $\bigcup N' \notin \B$. In the case that $N$ is $\B$-nested and $\bigcup N \in \Delta'$, we know that for any subset of disjoint sets $N' \subseteq N$ with $|N'| \ge 2$ that $\bigcup N' \in \Delta'$ and $\bigcup N' \in \B$, so $\bigcup N' \in (\B \cap \Delta')$. In either case, the second condition of \ref{complex-nestedset} holds.
\end{proof}

\begin{definition}
The \emph{Cartesian product} of two simplicial complexes $\Delta_1, \Delta_2$ with disjoint vertex sets is equal to the simplicial complex $\{S_1 \cup S_2|S_1 \in \Delta_1, S_2 \in \Delta_2\}$.
\end{definition}

We state the two following definitions for clarity.

\begin{definition}
The \emph{subset complex} of a set $S$ is the simplicial complex \mbox{$\{T|T \subseteq S\}$}, and the \emph{proper subset complex} of a set $S$ is the simplicial complex \mbox{$\{T|T \subset S\}$.}
\end{definition}

The subset complex of a set $S$ is often identified with a Boolean lattice, whereas the proper subset complex is often identified with the boundary complex of a simplex. With these terms defined, we can now state the following proposition.

\begin{proposition} \label{prop:singletonlink}
For any set $S$ in a $\Delta$-building set $\B$, the link of the singleton face $\{S\}$ in the $\mathcal{N}(\B,\Delta)$ is isomorphic to the Cartesian product of the nested complex of $\B$ restricted to the proper subset complex on $S$, and the nested complex $\mathcal{N}(\B/S,\Delta/S)$.
\end{proposition}


\danger{Massively rewritten again, this time simpler and with more named things.}

\begin{proof}
Proposition \ref{complex-nestedset} implies that if a set $S$ is in a $\B$-nested set $N$, then no other set in $N$ has nontrivial intersection with $S$. This means that if $S' \in N$ and $S' \neq S$, then either $S' \subsetneq S$, $S \subsetneq S'$, or $S \cap S'=\emptyset$. Call the set consisting of the first case $N_1=\{S'\in N|S'\subsetneq S\}$, and $N_2=\{S'\in N| S' \not\subseteq S\}$ is the set consisting of the latter two cases. We find $N=N_1 \cup N_2 \cup \{S\}$ for any $\B$-nested set $N$ containing $S$.

It is clear that $N_1$ is a member of the restriction of $\B$ to the proper subset complex of $S$. For this proof, call the proper subset complex $P(S)$, and we say $N_1 \subseteq \mathcal{N}(\B\cap P(S),P(S))$. It is also clear that $N_2$ is a member of $\mathcal{M}$, the simplicial complex used in the proof of Proposition \ref{prop:buildingsetlink}. We note that these two complexes are completely disjoint, as $N_1$ must contain only sets which are subsets of $S$, and $N_2$ must contain only sets which are not subsets of $S$. As a result, the link of $S$ in $\mathcal{N}(\B,\Delta)$ is contained in the simplicial complex Cartesian product $\mathcal{N}(\B\cap P(S),P(S))\times \mathcal{M}$. In order to prove that they are equal, we prove that for $N_1 \in \mathcal{N}(\B\cap P(S),P(S))$ and $N_2 \in \mathcal{M}$, we find $N_1 \cup N_2 \cup \{S\}$ is $\B$-nested. This however is simple. We note $N_1, N_2$ must already be $\B$-nested, and if $T_1, T_2$ have nontrivial intersection for $T_1 \in N_1, T_2 \in N_2$, then we would find $S$ and $T_2$ have nontrivial intersection, meaning $N_2 \notin \mathcal{M}$. Similarly, for some set of disjoint sets $N'\subseteq N$ with $|N'|\ge 2$ and $\bigcup N'\in \B$, we would find that if $N'=N_1' \cup N_2'$, we would find the set $\{S\} \cup N_2'$ would be disjoint, have magnitude $\ge 2$, and would have a union in $\B$, meaning $N_2 \notin \mathcal{M}$.

As a result, the link of $S$ in $\mathcal{N}(\B,\Delta)$ is the Cartesian product of these two simplicial complexes, and because $\mathcal{M}$ is isomorphic to the pseudolink nested complex, we have proven the proposition.
\end{proof}

\section{Graph nested complexes} \label{sec:graph-nested-complexes}

	Graph associahedra and related graphic building sets and their nested complexes have been studied in many papers. A description of graph associahedra and nested complexes is given both in the introduction of this thesis, and in Subsection \ref{sub:simplex-nestohedra}.This section does the work to generalize this construction to define $\Delta$-nested complexes based on graphs.

	The study of graph associahedra and their associated complexes uses a set of terminology not used in the study of other nestohedra and nested complexes. In nested complex terminology, elements of a building set $\B$ are not given a name, and collections of compatible elements of $\B$ are called nested sets. We will define a building set whose elements are connected induced subgraphs, which are called \emph{tubes}, and collections of compatible tubes are called \emph{tubings}. This terminology is used in the paper introducing graph associahedra, \cite{carrdevadoss}.

	For clarity, we will refer to constructions defined in prior work as graph associahedra, tubings, tubes, etc. as \emph{classical} graph associahedra, classical tubings, etc.



\subsection{$\Delta$-graph Tubings}

	\begin{definition}
		For a simplicial complex $\Delta$ on a base set $\mathcal{S}$, a $\Delta$-graph is a pair $(G,\Delta)$ such that $G$ is a graph on $\mathcal{S}$ such that, if $\{i,j\}$ is an edge in $G$, then $\{i,j\}$ is a face of $\Delta$.
	\end{definition}

	\begin{definition}
		A \emph{tube} of a $\Delta$-graph $(G,\Delta)$ is any face of $\Delta$ which induces a connected subgraph of $G$.
	\end{definition}

	\begin{definition}
		Two tubes $t_1, t_2$ of a $\Delta$-graph $(G,\Delta)$ are \emph{weakly compatible} if one of the following is true:
\begin{enumerate}
\item  $t_1 \subset t_2$
\item $t_2 \subset t_1$
\item $t_1 \cap t_2 = \emptyset$ and there exist no edges between vertices in $t_1$ and $t_2$.
\end{enumerate}
	\end{definition}

	For the traditional graph case, when $G$ is connected, a collection of tubes is a tubing if and only if all tubes are pairwise weakly compatible. However, nested sets of simplicial complexes add an extra requirement for tubings.



	\begin{definition}
		A set of tubes $T$ is \emph{strongly compatible} if the union $\bigcup_{t \in T} t$ is a face of $\Delta$, and it is pairwise weakly compatible.
	\end{definition}

	Note that weak compatibility is only defined for pairs of tubes. Strong compatibility, on the other hand, applies to sets of tubes. A set of tubes must be pairwise weakly compatible. However, note that it is possible for a set of tubes to be pairwise strongly compatible, but not be strongly compatible.

	\begin{definition}
		A \emph{tubing} is any strongly compatible set of tubes.
	\end{definition}

	\begin{definition}
		The \emph{graphical building set} $\B_{(G,\Delta)}$ of a $\Delta$-graph $(G,\Delta)$ is the set of all tubes of $(G,\Delta)$.
\end{definition}

\begin{proposition}
The graphical building set $\B_{(G,\Delta)}$ of a $\Delta$-graph $(G,\Delta)$ is a $\Delta$-building set.
\end{proposition}

\begin{proof}
First, we see that every singleton set in $\Delta$ induces a single vertex subgraph, which is connected, so every singleton set of $\Delta$ is a tube.

Next, consider two tubes $t_1, t_2 \in \B_{(G,\Delta)}$. If $t_1 \cap t_2 \ne \emptyset$, then $t_1 \cup t_2$ must induce a connected subgraph. If $t_1 \cup t_2 \in \Delta$, then $t_1 \cup t_2$ is a tube in $\B_{(G,\Delta)}$.
\end{proof}



	\begin{proposition}
		For a $\Delta$-graph $(G,\Delta)$, the set of tubings of $(G,\Delta)$ is equal to the set of nested sets of the graphical building set of $(G,\Delta)$.
	\end{proposition}

	\begin{proof}

	First, we prove that a tubing $T$ is a nested set. First, a tubing has pairwise weak compatibility, so no two tubes have nontrivial intersection. Secondly, if $T'=\{t_1, \ldots, t_k\} \subseteq T$ is a set of disjoint tubes with $|T'|\ge 2$, then we see that due to weak compatibility, $\bigcup T'$ induces $k$ graph components with no edges between them. As a result, $\bigcup T'$ is not a tube. Because of strong compatibility, $\bigcup T' \in \Delta$. As a result, $T$ must be a nested set.

	Now we must prove that if $T$ is nested, then $T$ is a tubing. For any two tubes $t_1, t_2 \in T$, we know $t_1 \subset t_2, t_2 \subset t_1$, or the two are disjoint. If the two are disjoint, then $t_1 \cup t_2$ cannot be a tube. If $t_1 \cup t_2 \in \Delta$ cannot be a tube, then there are no edges between $t_1$ and $t_2$. This shows $T$ is pairwise weakly compatible.

	In addition, if $T$ is $\Delta$-nested, then $\bigcup T \in \Delta$. As a result, $T$ is strongly compatible, and $T$ is a tubing.

\end{proof}

\begin{definition}
The \emph{$\Delta$-building set closure} of a set $E \subseteq \Delta$ is the minimal $\Delta$-building set containing $E$.
\end{definition}

It is trivial to show that the intersection of two $\Delta$-building sets is a $\Delta$-building set, so we see that a unique minimal $\Delta$-building set must exist.

\begin{proposition} \label{prop:graphminimalbuildingset}
The graphic building set of a $\Delta$-graph $G$ is the $\Delta$-building set closure of the edge set of $G$.
\end{proposition}

\begin{proof}
Given a building set $\B$ and a subset $E \subseteq \B$, we find that $\B$ is the minimal $\Delta$-building set containing $E$ if $\B$ is a building set and, for every non-singleton set $S \in \B$, there exists a series $S_1, S_2, \ldots, S_k$ such that $\bigcup_{i=1}^k S_i=S$, and $\left(\bigcup_{i=1}^j S_i \right)$ and $S_k$ have nontrivial intersection.

Every tube $t$ induces a connected subgraph. Every connected graph contains a spanning tree. Every tree consists of edges $e_1, \ldots, e_k$ which can be ordered such that the graph consisting of edges $e_1, \ldots, e_i$ is connected for each $1\le i \le k$, and the edge $e_{i+1}$ connects a new vertex to a growing subtree. Bringing this together, if each edge is represented as a pair of vertices $e=\{u,v\}$, then $t=\bigcup_{i=1}^k e_i$, and $\bigcup_{i=1}^j e_i$ and $e_{j+1}$ intersect for each $1 \le j \le k-1$.
\end{proof}

	\begin{remark}
A theory of $\Delta$-graphs may be developed where the requirement that edges be faces of $\Delta$ is not necessary. However, if we define $\Delta$-graphs to allow such edges, we find them extraneous. If $G$ is a $\Delta$-graph, and $G'$ is the union of $G$ with a set of additional edges which are not faces in $\Delta$, then the sets of tubes and tubings of $G$ will be equal to the sets of tubes and tubings of $G'$. Furthermore, because we forbid edges not in $\Delta$, this means the edge set of $G$ is a subset of $\Delta$, and we can state Proposition \ref{prop:graphminimalbuildingset}. Finally, Chapter \ref{chap:hypercubes} makes great use of the convention that `forbidden subsets' use dashed edges, and this way one does not draw a dashed and an undashed edge between the same pair of vertices. For these reasons, we have made the decision to define $\Delta$-graphs in this way.
\end{remark}

\subsection{Reconnected complement} \label{sub:reconnectedcomplement}

Regarding classical graph tubes and tubings, the work in \cite{carrdevadoss} proves that the link of a classical $1$-tubing $\{t\}$ of a graph $G$ is combinatorially isomorphic to the product of the classical tubing complex of the graph induced by $t$ in $G$, and the classical tubing complex of the \emph{reconnected complement} of $t$ in $G$, where the reconnected complement is the graph obtained by taking the complement $G\backslash t$ of $t$ in $G$, and drawing edges connecting all pairs $(u,v)$ such that $u$ is connected to some vertex of $t$ and $v$ is connected to some vertex in $t$. In this subsection, we generalize this classical reconnected complement for $\Delta$-graphs.  

\begin{definition}
	Given a $\Delta$-graph $G$, a vertex $v \in G$ is \emph{reconnectable} to a tube $t$ if $t \cup \{v\}$ is in $\Delta$.
\end{definition}

The vertices reconnectable to a tube $t$ form the base set of the link $\Delta/t$.

\begin{definition} \label{def:reconnected-complement-graph}
The \emph{reconnected complement} of the tube $t$ in a $\Delta$-graph $G$ is the graph $G/t$ on the set of vertices not in $t$ but reconnectable to $t$, with edges defined as follows:
\begin{enumerate}
\item If $\{x,y\}$ is an edge in $G$, then $\{x,y\}$ is an edge in $G/t$.
\item If $x,y$ are both adjacent to vertices in $t$ and $\{x,y\}$ is a face of $\Delta/t$, then $\{x,y\}$ is an edge in $G/t$.
\end{enumerate}
\end{definition}

\begin{proposition} \label{prop:pseudolink-graph}
Given a $\Delta$-graph $G$ and a tube $t$ of $G$, the pseudolink of the graphic $\Delta$-building set of $G$ with respect to $t$ is equal to the graphic $(\Delta/t)$-building set of $G/t$.
\end{proposition}

\begin{proof}

If $\B_G$ is the graphic $\Delta$-building set of $G$, then $\B_G/t$ is the pseudolink of $\B_G$. We restate the definition of the pseudolink:

	\[
		\B_G/t = \{s|s \in \B_G, s \cup t \in \Delta, s \cap t = \emptyset\} \cup \{s \backslash t|s \in \B_G, t \subset s\}
\]

First, we prove that every element of $\B_G/t$ is a tube of $G/t$. Every element in $\B_G/t$ is either equal to $s$ in $\{s|s \in \B_G, s \cup t \in \Delta, s \cap t = \emptyset\}$, or $s \backslash t$ in $\{s \backslash t|s \in \B_G, t \subset s\}$.

Consider the case that $s \in \B_G, s \cup t \in \Delta,$ and $s \cap t  = \emptyset$. We find that $s,t$ are disjoint, and $s\cup t \in \Delta$ so all vertices of $s$ are reconnectable to $t$. Note that $G/t$ is a graph constructed from the graph induced by the vertices reconnectible to $t$ in $G$ by adding edges, so if $G|_s$ is connected, then $G/t|_s$ is connected. As a result, $s$ is a tube in $G/t$.

Consider the other case, that $s \in \B_G$ and $t \subset s$, and we will prove that $s \backslash t$ is a tube of $G/t$. Say that $G|_{s\backslash t}$ has $k$ components. If $k\ge 2$, the graph $G/t$ adds a clique connecting all components of $G|_{s\backslash t}$, meaning that $s \backslash t$ induces a connected graph, and $s\backslash t$ is a tube of $G/t$. As a result, $\B_G$ consists of $G/t$-tubes.

If $s$ is a tube of $G/t$, then consider the following. If $s$ is connected in $G$, then $s$ is a tube of $G$ with $s\cap t=\emptyset$, so $s \in \B_G/t$. If $s$ is not connected in $G$, then because the only edges between reconnectable vertices are between neighbors of $t$, this means every component in $G|_s$ contains a neighbor of $t$. This means $s \cup t$ is a tube, and therefore $(s \cup t)\backslash t=s$ is in $\B_G/t$. As a result, we have proven the proposition.
\end{proof}

\section{$\p$-nestohedra} \label{sec:nestohedra}




	\danger{This is a new subsection pulling some material out from the next subsection!!!!}

	\subsection{Definitions for polyhedra}

A \emph{polyhedron} is a subset of points in a vector space defined by a finite set of non-strict hyperplane inequalities of the form $\{\mathbf{c}_i \mathbf{x} \ge y_i| \mathbf{x} \in \R^n\}$ for a set of row vectors $\mathbf{c}_i \in (\R^n)^*$ and $y_i \in \R$. For every polyhedron, we may choose a minimal set of inequalities which define the polyhedron, each of which defines a facet. A full-dimensional polyhedron of dimension $n$ is \emph{simple} if every codimension $k$ face is contained in exactly $k$ facets. For the scope of this thesis, we will use $\p$ to refer to the class of full-dimensional simple polyhedra.

A \emph{polytope} is any polyhedron which is bounded. In this thesis, results are proven for the polyhedron case wherever possible.

	Consider a simple polyhedron $P$. Index the facets of $P$ by some set $\mathcal{S}$, so that the set of facets is $\{F_s|s \in \mathcal{S}\}$. For a subset $I \subseteq \mathcal{S}$, define the face $F_I$ as the intersection of all facets $F_s$ for $s \in I$. Because $P$ is simple, each non-empty face of $P$ has a unique expression as an intersection of facets of $P$.

	\begin{definition} \label{def:dualsimplicialcomplex}
		The \emph{dual simplicial complex}, written as $\Delta(P)$, of a simple polyhedron $P$ is the simplicial complex of sets $I \subseteq S$ such that the face $F_I \ne \emptyset$. 
\end{definition}

The \emph{lineality space} $lineal(P)$ of a polyhedron $P$ in $\R^n$ is the set of vectors $v \in \R^n$ such that $v+x \in P$ for all $x \in P$. This is a subspace. If $U$ is a complementary subspace to $lineal(P)$, then $P = (P \cap U) + lineal(P)$. We say that $P$ is \emph{pointed} if and only if $lineal(P)=\{0\}$, and this is equivalent to saying that $P$ has at least one vertex face.

The \emph{face lattice} of a polyhedron $P$ is the poset of all faces of $P$, including the empty face. For a simple polyhedron, the dual simplicial complex is isomorphic to the of the dual face lattice with the empty face removed. In this isomorphism, the set $S$ in the dual simplicial complex maps to the codimension-$|S|$ face $F_S$ of $P$. We define the following:

\begin{definition}
Two polyhedra $P, P'$ are \emph{combinatorially isomorphic} if their face lattices are related by an isomorphism taking each face $F$ in $P$ to a face $F'$ in $P'$ such that $dim(F)=dim(F')$. 
\end{definition}

\begin{remark} \label{remark:dual-simplicial-complex-non-uniqueness}
We note that if $P, P'$ are both simple, then $P, P'$ are combinatorially isomoprhic if and only if their dual simplicial complexes are isomorphic and $P, P'$ have either the same dimension, or the same dimension lineality space. We note that dual simplicial complex does not uniquely determine polyhedra up to combinatorial isomorphism. However, dual simplicial complexes do uniquely determine \emph{pointed} polyhedra up to combinatorial isomorphism.
\end{remark}

	\subsection{Defining $\p$-nestohedra}

	\begin{definition}
		A $P$-building set $\B$ is a building set of the dual simplicial complex of a simple polyhedron $P$. Nested sets of a $P$-building set are called \emph{$P$-nested}.
	\end{definition}


	\begin{definition}
		A \emph{$P$-nestohedron} of a $P$-building set $\B$ is a simple polyhedron whose dual simplicial complex is isomorphic to the nested complex of $\B$. Use the notation $\nest{P}{\B}$ to refer to a $P$-nestohedron of $\B$.
	\end{definition}

	

	We can define a method for constructing $P$-nestohedra by repeated truncation of $P$. Consider an inequality $ax \le b$, dividing the vector space into two halfspaces, the open halfspace $H_{>}$ defined by $ax > b$, and the closed halfspace $H_{\le}$ defined by $ax \le b$. The hyperplane $H$ is the set of points $ax=b$. We can define $H_<$ and $H_{\ge}$ similarly.

	\begin{definition}
		Consider a polyhedron $Q$ with a face $F$, and a halfspace $H_{\le}$ defined by $ax \le b$. We say that this inequality \emph{truncates} $Q$ by the face $F$ if $F\subset H_>$ and the following two conditions are true. First, if $F' \cap F \ne \emptyset$ and $F' \not\subseteq F$ for any face $F'$ of $Q$, then $H \cap F' \ne \emptyset$. Secondly, if $F' \cap F = \emptyset$ and $F'$ is a face of $Q$, then $F' \subset H_<$. If a hyperplane truncates $Q$ at a face $F$, define the \emph{truncation of $Q$ at the face $F$} as the polyhedron $Q \cap H_{\le}$. We may call this $\textrm{Tr}_H(Q)$.

\end{definition}

When $Q$ is a polytope and $F$ is a proper face, we note that any inequality which truncates $F$ separates all vertices of $F$ from all other vertices in $Q$, effectively `cutting off' that face. Figure \ref{fig:truncation-examples} shows some example truncations of faces of a cube and a square, including a facet truncation which amounts to the perturbation of a facet-defining hyperplane and therefore does not change the combinatorial structure of a simple polyhedron.

\begin{figure} [h]
\centering
\includegraphics[width=.75\textwidth]{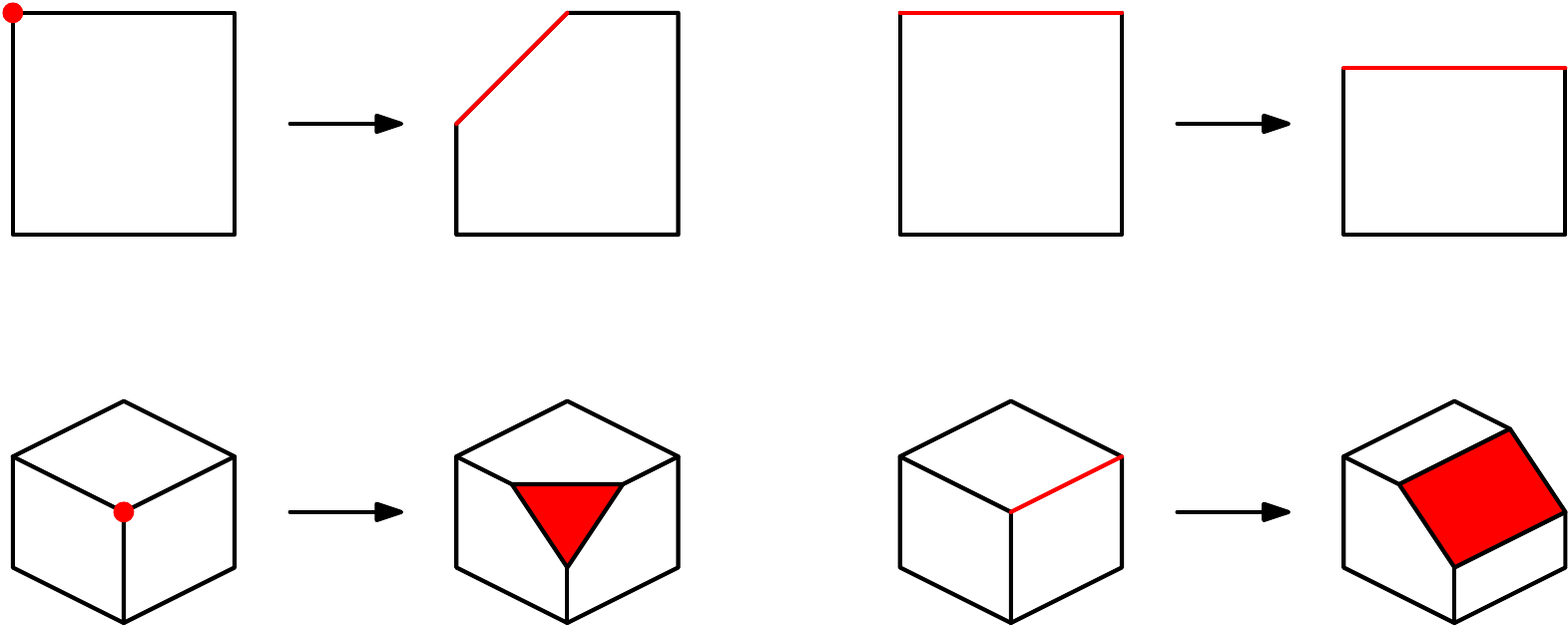}
\caption{Example truncations.} \label{fig:truncation-examples}
\end{figure}

 We note in the next proposition that we can always find a truncation for any face.

\begin{proposition}
For any face $F$ of a polyhedron $Q$ defined by an inequality $ax \le b$, there exists $\varepsilon > 0$ such that the inequality $ax \le b-\varepsilon$ truncates $F$ in $Q$.
\end{proposition}

%
%

\emph{Combinatorial blowup} is defined in \cite{feichtner2003}, where it is noted that it specializes to stellar subdivision of an abstract simplicial complex. We define a version of stellar subdivision here.

\begin{definition} \label{def:stellar-subdivision}
Given a simplicial complex $\Delta$ and a set $S$ in $\Delta$, the \emph{stellar subdivision} $\textrm{St}_S \Delta$ is a simplicial complex containing the elements
\begin{itemize}
\item $T \in \Delta$ such that $S \not\subseteq T$
\item Sets $T \cup \{h\}$ such that $S \not\subseteq T$ and $S \cup T \in \Delta$.
\end{itemize}
for some element $h \notin \Delta$. We say we replace the set $S \in \Delta$ with the element $h$.
\end{definition}

The \emph{tube promotion} process described in \cite{carrdevadoss} is a special case of a stellar subdivision operation applied to nested complexes. We note that when $P$ is a simple polyhedron, truncation is dual to stellar subdivision. \danger{I think that more broadly truncation is dual to combinatorial blowup, but I don't see that source easily.}

\begin{proposition}
If $P$ is a simple polyhedron with nonempty face $F_S$ defined by intersection of facets $\{F_s|s \in S\}$ for nonempty set $S$, and $H$ truncates $P$ at $F$, then $\Delta(\textrm{Tr}_{F_S} P)=\textrm{St}_T\left(\Delta(P)\right)$.
\end{proposition}

\begin{proof}

First, we take it as trivial that if a simple polyhedron $Q$ is $k$-dimensional, and $H$ is a hyperplane, then if $H\cap Q$ is not a face of $Q$, then $H \cap Q$ is $k-1$-dimensional, and $H_{\le} \cap Q$ is $k$-dimensional. Note $H\cap Q$ does not need to be a proper or nonempty face of $Q$. We also note that when $H$ truncates a face of a polyhedron $P$, then $H \cap P$ is not a nonempty face of $P$, and for any $k$-dimensional face $F$ such that $F \cap H$ is nonempty, $F \cap H$ is $k-1$-dimensional and $F \cap H_{\le}$ is $k$-dimensional.

Now, we wish to characterize all faces of $Tr_{F_S}(P)$. These faces are of two possible forms: the first is $F_T \cap H_{\le}$, and the second is $F_T \cap H$.

Consider the case where $S \subseteq T$. We find $F_T \subseteq F_T$, and $F_T \cap H_{\le} = F_T\cap H = \emptyset$. These faces are empty in the truncation.

Consider the case where $S \not\subseteq T$. If $S \cup T \in \Delta(P)$, then $F_S \cap F_T$ is nonempty, and by the definition of truncation, $F_T \cap H \ne \emptyset$. As a result, $F \cap H_{\le}$ is a $k$-dimensional face. If $S \cup T \notin \Delta(P)$, then $F_S \cap F_T = \emptyset$, and by the truncation definition, $F_T \subset H_<$. This means $F_T \cap H_{le}=F_T$.

As a result, we have characterized all faces of the form $F_T \cap H_{\le}$. If $S \subseteq T$, then $F_T \cap H_{\le}=\emptyset$, and if $S \not\subseteq T$ for $T \in \Delta(P)$, then $F_T \cap H_{\le}$ is a codimension-$|T|$ face.

Now consider the faces of the form $F_T \cap H$. If $S\not\subseteq T$ and $S \cup T \notin \Delta(P)$, then $F_T \subset H_<$, and $F_T \cap H=\emptyset$. If $S \not\subseteq T$ and $S \cup T \in \Delta(P)$, then as proven above, $F_T \cap H \ne\emptyset$. This face is codimension-$(|T|+1)$.

Now we consider the facets of $Tr_{F_S}P$. These are of the form $F_s \cap H_{\le}$ for all facet indices $s \in \mathcal{S}$, and the facet $H \cap P$. Each nonempty intersection of $k$ of these facets is codimension-$k$, meaning that it is simple and its non-empty face poset is indeed isomorphic to its dual simplicial complex. As a result, if we index faces of $Tr_{F_S}P$ as $F'_T=F_T \cap H_{\le}$ and $F'_{T \cup \{h\}}=F_T \cap H$, we find that $\Delta(Tr_{F_S}P)$ is exactly equal to $St_S(\Delta(P))$.
\end{proof}

We can understand the following theorem from \cite{feichtner2003} by remembering that combinatorial blowups specialize to stellar subdivision, which we have just proven is dual to truncation for simple polyhedra.


\begin{theorem}[\cite{feichtner2003}, Theorem 3.4] \label{thm:feichtnermainresult}
Let $\mathcal{L}$ be a semilattice and let $G$ be a semilattice-building set of $L$ with some chosen linear extension $G=\{G_1,\ldots,G_t\}$, with $G_i > G_j$ implying $i < j$. Let $\textrm{Bl}_x \mathcal{L}$ denote the combinatorial blowup of $\mathcal{L}$ at element $x$. Let $\textrm{Bl}_k \mathcal{L}$ denote the result of subsequent blowups $\textrm{Bl}_{G_k}(\textrm{Bl}_{G_{k-1}}(\ldots \textrm{Bl}_{G_1}\mathcal{L}))$. Then the final semilattice $\textrm{Bl}_t$ is equal to the face poset of the semilattice-nested complex of $G$.
\end{theorem}

%
%
%

%

Note that one such linear extension can easily be found for a building set, by ordering all elements in a building set by decreasing size. Applying this theorem, and knowing that in this case combinatorial blowups are dual to truncation, we prove the following.

\begin{theorem}
	For every $P$-building set $\B$, there exists a $P$-nestohedron of $\B$ obtained by listing the elements $\B=\{S_1,\ldots,S_k\}$ such that $|S_i| \ge |S_{i+1}|$ for all $i=1,\ldots, k-1$, and repeatedly truncating the faces $F_{S_1}, \ldots, F_{S_k}$ of $P$ in order.
\end{theorem}


\begin{proof}

Define $\B_0=\{\{s\}|s \in \mathcal{S}\}$ for facet set $\mathcal{S}$, and define $\B_i=\B_0 \cup \{S_1, \ldots, S_i\}$. We note that $P$ is the $P$-nestohedron of $\B_0$. Now we wish to prove that $\textrm{Tr}_{F_{S_i}} \nest{P}{\B_i}$ is equal to $\nest{P}{\B_{i+1}}$. Note that ordering $S_1, \ldots, S_k$ by size means that $S_i \supset S_j$ implies $i < j$, so Theorem \ref{thm:feichtnermainresult} applies. Because the stellar subdivision is dual to truncation,  we find $\Delta(\nest{P}{\B_{i+1}})$ is equal to the stellar subdivision $\textrm{St}_{S_{i+1}} (\Delta(\nest{P}{\B_{i}}))$, which is equal to the dual simplicial complex $\Delta(\textrm{Tr}_{F_{S_{i+1}}}\nest{P}{\B_{i}})$. This proves our theorem inductively.
\end{proof}

\subsection{Faces of $\p$-nestohedra}\label{sub:Pcomplexlinks}

Recall that we index facets of a simple polyhedron $P$ to be faces $\{F_s| s \in \mathcal{S}\}$, and faces of $P$ are of the form $F_S=\bigcap_{s \in S} F_s$ for $S \in \Delta(P)$. For a $P$-building set $\B$, we note that $\Delta(\nest{P}{\B})$ is equal to the nested complex $\mathcal{N}(\B,P)$. As a result, we can index faces of $\nest{P}{\B}$ with $\B$-nested sets, and define $\Phi_N$ to be the face of $\nest{P}{\B}$ dual to the nested set $N$.

We note some basic properties of faces of $\nest{P}{\B}$. We find that the face $\Phi_N$ is codimension-$|N|$, which means that facets of $\nest{P}{\B}$ are of the form $\Phi_{S}$, where $S \in \B$. For two nested sets $N_1, N_2$, we find that if $N_1 \cup N_2$ is $\B$-nested, then $\Phi_{N_1} \cap \Phi_{N_2}=\Phi_{N_1 \cup N_2}$, and otherwise $\Phi_{N_1}\cap \Phi_{N_2} = \emptyset$.

We also note that the dual simplicial complex of a face $\Phi_N$ is isomorphic to the link of $N$ in $\mathcal{N}(\B,P)$. Proposition \ref{prop:singletonlink} characterizes links in nested sets as isomorphic to Cartesian products of simplicial complexes, motivating us to find the following.

\begin{proposition} \label{prop:polyhedral-product}
		For two simple polyhedra $P$ and $Q$, the dual simplicial complex of the Cartesian product $P \times Q$ is combinatorially equivalent to the Cartesian product of $\Delta(P)$ and $\Delta(Q)$.
	\end{proposition}

\begin{proof}
The Cartesian product of two polyhedron is defined such that if $x \in \R^n$ is a point in a polyhedron $P$, and $y \in \R^m$ is a point in polyhedron $Q$, then $(x,y)=(x_1,\ldots,x_n,y_1,\ldots,y_m)\in \R^{n+m}$ is a point in $P \times Q$.

Note that if $c(x,y) \le b$ defines a face of $P \times Q$, then we can split $c$ and $b$ to find new inequalities $c_1 x \le b_1$ and $c_2 x \le b_2$. As a result, every face of $P \times Q$ is isomorphic to the product of a face of $P$ and a face of $Q$. It is then trivial to prove that every product of nonempty faces of $P$ and $Q$ is a face of $P \times Q$, and also that each such product is unique. This applies to the dual simplicial complexes as well.
\end{proof}

Note that if $\sum_{i=1}^n c_i x_i \le b$ defines a facet of $P$ in $\R^n$, then the same inequality defines a facet of $P \times Q$ in $\R^{n+m}$. The facet set of $P\times Q$ can be indexed by the disjoint union of the facet indexing sets of $P$ and $Q$.

The following is a result of Proposition \ref{prop:singletonlink} for the $\p$-building set case.

	\begin{theorem} \label{thm:reconnectednestohedra1}
		For every $P$-nestohedron $\nest{P}{\B}$, every facet of $\nest{P}{\B}$ is combinatorially isomorphic to the product of a $F$-nestohedron, where $F$ is some face of $P$, and a simplex-nestohedron.
\end{theorem}
	\begin{proof}
Every facet of $\nest{P}{\B}$ is of the form $\Phi_{\{S\}}$ for some set $S \in \B$, and the dual complex of $\Phi_{\{S\}}$ is equal to the link of $\{S\}$ in the nested complex of $\B$. Note, though, that we have already characterized the link of $\{S\}$ to be isomorphic to the Cartesian product of the nested complex $\mathcal{N}(\B/S,\Delta(P))$ and the nested complex of the restriction of $\B$ to the complex of proper subsets of $S$.

The link of the set $S$ in the dual complex of $P$ is the dual complex of the face $F_S$, so $\mathcal{N}(\B/S,\Delta(P))$ is the dual complex of an $F_S$-nestohedron. The complex of proper subsets of $S$ is the dual simplicial complex of a simplex, and so the nested complex of the restriction of $\B$ to this complex is the dual complex of a simplex-nestohedron.

We wish to make explicit that two edge cases, when $S$ is a singleton set or when $\B/S$ is empty, are not issues with this definition.

 First, we say that when a polyhedron $Q$ is a single point, $Q$ is a $0$-dimensional simplex which contains one facet. The dual simplicial complex of $Q$ only contains one set, the empty set, and is the proper set complex of a singleton set. A building set on $Q$ is empty, and the nested complex is empty. The nestohedron of a point is a point.

As a result, when $S$ is a singleton set, the set $\B$ restricted to the proper subset complex of $S$ is empty, and the associated simplex nestohedron is a point. When $\B/S$ is empty, the associated $F$-nestohedron is a point.
\end{proof}

This theorem can be applied to all faces of $\p$-nestohedra.

	\begin{corollary} \label{thm:reconnectednestohedra2}
		For every $P$-nestohedron $\nest{P}{\B}$, every face of $\nest{P}{\B}$ is combinatorially isomorphic to the product of a $F$-nestohedron, where $F$ is some face of $P$, and a set of simplex-nestohedra.
\end{corollary}

	\begin{proof}
	See that $P$ is a face of $P$, so $\nest{P}{\B}$ is isomorphic to the product of an $F$-nestohedron, where $F$ is some face of $P$, and an empty set of simplex-nestohedra.

	Every facet of a product of polyhedron $Q_1 \times Q_2$ is isomorphic to the product of $Q_1$ and a facet of $Q_2$, or vice versa. Every facet of a simplex-nestohedron is the product of simplex-nestohedra. As a result, if $Q$ is the product of an $F$-nestohedron and a set of simplex-nestohedra, then every facet of $Q$ is isomorphic to either an $F$-nestohedron and a set of simplex-nestohedra, or an $F'$-nestohedron and a set of simplex-nestohedra, where $F'$ is a facet of $F$.

	Note that every proper face of a polyhedron can be expressed by repeatedly taking facets of the polyhedron. As a result, this theorem is proven for every face of $\nest{P}{\B}$ inductively.
\end{proof}

We can apply Proposition \ref{prop:pseudolink-graph} to restate our theorem for $\p$-graph associahedra.

\begin{proposition}
Every facet $\Phi_{\{t\}}$ of a $P$-graph associahedron is combinatorially isomorphic to the Cartesian product of the $F_t$-graph associahedron of the graph $G/t$, and the simplex-graph associahedron of the graph $G|_t$ induced by $t$ in $G$.
\end{proposition}

Unrelated to the main argument of this section, we note this result for two disjoint polyhedra $P$ and $Q$.

\begin{proposition}\label{prop:nestohedron-of-product}

Consider the case where $P, Q$ are simple polyhedra, with $\B_P, \B_Q$ a $P,Q$-building set respectively, and disjoint facet-sets $\mathcal{S}_P, \mathcal{S}_Q$. We find that the $P\times Q$-nestohedron of $\B_P\cup \B_Q$ is equal to the Cartesian product of the $P,Q$-nestohedra of $\B_P$ and $\B_Q$.
\end{proposition}
d
\begin{proof}
This is trivial; we simply note that if $N_P$ and $N_Q$ are $\B_P$ and $\B_Q$-nested respectively, then $N_P\cup N_Q$ is $\B_P \cup \B_Q$-nested, and any $\B_P \cup \B_Q$-nested set $N$ can be decomposed in this manner.
\end{proof}

Note that this proves that every product of $\p$-nestohedra is itself a $\p$-nestohedron. However, it is not always the case that a $P\times Q$-nestohedron can be written as a product of a $P$-nestohedron and a $Q$-nestohedron; this depends on the building set of the $P\times Q$-nestohedron.


\subsection{Classical nestohedra and simplex nestohedra}\label{sub:simplex-nestohedra}

The paper \cite{postnikov} provides a definition for \emph{nestohedra}, \emph{building sets}, and \emph{nested complexes} in set-theoretic terms which closely parallel Definition \ref{complex-buildingset}. We will restate these definitions here, and explain their relationship to simplex-nestohedra. Throughout this thesis, we refer to these terms as classical building sets, classical nestohedra, and so forth to avoid confusion.

\begin{definition}
A classical building set $B$ of a set $S$ is any subset $B$ of the power set of $S$, satisfying the conditions
\begin{enumerate}
\item For each element $i\in S$, $\{i\} \in B$
\item If $I\in B, J \in B$ and $I \cap J \ne \emptyset$, then $I \cup J \in B$.
\end{enumerate}
\end{definition}

\begin{proposition}
There exists a unique set $B_{max} \subseteq B$ such that:
\begin{enumerate}
\item $B_{max}$ is a partition of $S$
\item For each set $I \in B_{max}$, $I$ is not a proper subset of any set in $B$.
\end{enumerate}
\end{proposition}

Sets in $B_{max}$ are referred to as connected components of $B$, and a classical building set $B$ on set $S$ is \emph{connected} if $B_{max}=\{S\}$.

\begin{definition}
A \emph{classical nested set} $N$ is any subset of a classical building set $B$ such that
\begin{enumerate}
\item $B_{max} \subseteq N$
\item For any two sets $I, J \in N$, we find $I\subseteq J$, $J \subseteq I$, or $I, J$ are disjoint
\item For any set of $k$ disjoint sets $I_1, \ldots, I_k \in N$, $k \ge 2$, we find $\bigcup_{i=1}^k I_i$ is not in $B$.
\end{enumerate}
\end{definition}

The collection of sets of the form $N\backslash B_{max}$, where $N$ is $B$-nested, is a simplicial complex, which we call the \emph{nested complex} of $B$.


\begin{definition}
A \emph{classical nestohedron} of a classical building set $B$ is a simple polyhedron whose dual simplicial complex is isomorphic to the nested complex of $B$.
\end{definition}

With these definitions restated for this thesis, we can prove the following statement about classical nested complexes.

\begin{proposition}\label{prop:product-of-simplices}
The nested complex of a classical building set $B$ with $B_{max}=\{S_1,\ldots,S_k\}$ is equal to the Cartesian product of nested complexes $\mathcal{N}(\B_i,\Delta_i)$, where $\Delta_i$ is the proper subset complex of $S_i$, and $\B_i$ is the $\Delta_i$-building set defined by restricting $B$ to $\Delta_i$.
\end{proposition}

\begin{proof}
Consider the case where $B_{max}$ contains only the base set $S$ of $B$. Write $\B=B \backslash \{S\}$. We see that for sets $I, J \in B$ with $I \subset S, J \subset S$ and $I \cap J \ne \emptyset$ that $I \cup J \in \B$ if and only if $I \cup J \ne S$. This corresponds to the $\Delta$-building set definition if $\Delta$ is the proper subset complex of $S$. As a result, $B$ is a classical building set if and only if $\B$ is a simplex-building set. It is then trivial to show that in this case, a set $N$ containing $S$ is $\B$-nested if and only if $N\backslash \{S\}$ is $B$-nested.

In the case where $B_{max}=\{S_1,\ldots,S_k\}$ with $k \ge 2$, define $\Delta=\prod_{i=1}^k \Delta_i$, where $\Delta_i$ is the proper subset complex of $S_i$. Define $\B_i$ as $B$ restricted to $\Delta_i$. It is then trivial to see that each building set $B$ decomposes into $B_{max}$ and a set of smaller building sets, and that each nested complex $N$ can be decomposed into the union of $B_{max}$ and a series of $\Delta_i$-nested sets $N_i$ of $\B_i$ for each $i \in [k]$.
\end{proof}

As an immediate consequence of Proposition \ref{prop:product-of-simplices}, we obtain the following statement about nestohedra.

\begin{proposition}\label{prop:simplex-nestohedra}
A polyhedron is a classical nestohedron if and only if it is combinatorially isomorphic to the Cartesian product of a set of simplex-nestohedra.
\end{proposition}

We also note the following as a result of Proposition \ref{prop:nestohedron-of-product}.

\begin{proposition}
The nestohedron of a classical building set $B$ with $B_{max}=\{S_1,\ldots,S_k\}$ is isomorphic to the $P$-nestohedron of $\B$, where $\B=B\backslash B_{max}$, and $P=\prod_{i=1}^k D_i$, where $D_i$ is the simplex with facet set $S_i$.
\end{proposition}

%

\begin{remark}
It should also be noted that every classical nestohedron on $S$ with building set $B$ is combinatorially isomorphic to a face of a $P$-nestohedron with building set $B$, where $P$ is a simplicial cone isomorphic to the product of $|S|$ rays. This is the construction one achieves when one does not require the set $B_{max}$ to be contained as as subset of every $B$-nested set. Combinatorially, this structure is not much more interesting than the nestohedron, and we note that this cone-nestohedron is isomorphic to the product of the classical nestohedron $\mathcal{K}_B$ times the product of $|B_{max}|$ rays. 
\end{remark}

We can now focus on defining classical graph associahedra. Graph associahedra are mentioned as a special case of nestohedra in \cite{postnikov}, and are studied in great depth in \cite{devadoss2011}. It is important to note here that two slightly different definitions of graph associahedra exist in the literature. When $G$ is a connected graph, the two definitions are equivalent, but when $G$ is not connected, the two are different, although both are nestohedra. Many papers elide the issue by focusing on the case where a graph is connected. Throughout this paper, the relevant definition is Definition \ref{def:graphassociahedra}, as these classical graph associahedra are isomorphic to simplex-graph associahedra.

Once we have proven that these graph associahedra are in fact simplex-graph associahedra, we will use the term simplex-graph associahedron over graph associahedron to avoid confusion.

\danger{I realized I was trying to prove stuff while elaborating on the implications of a definition, so I decided I might as well make a definition and a proposition.}

The following definition is equivalent to the one defined in \cite{devadoss2011}

\begin{definition}\label{def:graphassociahedra}
A \emph{classical graph associahedron} of a graph $G$ on set $S$ is a nestohedron whose building set $B$ contains all proper vertex sets of $G$ which induce connected subgraphs, as well as the entire vertex set of $G$.
\end{definition}

\begin{proposition}
The classical graph associahedron of a graph $G$ is equal to the simplex-graph associahedron of $G$.
\end{proposition}

\begin{proof}
We have already established a bijection between simplex-nestohedra and nestohedra for which $B_{max}$ contains an entire base set. As a result, we see that a simplex-graph associahedron is a nestohedron whose building set $B$ contains all proper vertex sets of $G$ which induce connected subgraphs, as well as the entire vertex set of $G$.
\end{proof}


We can define an alternate nestohedron whose building set only contains graph tubes. This is the definition of graph associahedra presented in \cite{postnikov}.

\begin{definition}
The \emph{component-product classical graph associahedron} of a graph $G$ on vertex set $\mathcal{S}$ is the nestohedron whose building set $B$ contains all subsets of $\mathcal{S}$ which induce connected subgraphs of $G$.
\end{definition}

When $G$ is not connected, $G$ is the product of the simplex-graph associahedra of each component subgraph of $G$, and will be lower-dimensional than the simplex-graph associahedron of $G$. We should note that this graph associahedron is less relevant to the research in this paper than the classical graph associahedron; in general, when we present an unconnected $\p$-graph, we do not expect $\p$-graph associahedra in general to be isomorphic to a product of $\p$-graph associahedra of component graphs.

	\subsection{Forbidden Subset Diagrams}\label{sub:forbiddensubsetdiagrams}

%
%
%

	Simplicial complexes can be defined in multiple ways. Given a base set $\mathcal{S}$ and a collection of subsets $S_1, \ldots, S_k$ of $\mathcal{S}$. We say that the \emph{forbidden subsets} $S_1, \ldots, S_k$ define a simplicial complex $\Delta$ consisting of all subsets $J$ of $\mathcal{S}$ such that $S_i \not\subseteq J$ for each $1 \le i \le k$.

	Note that it is possible for two sets of forbidden subsets to define the same simplicial complex. However, for any simplicial complex $\Delta$, there is a unique minimal set of forbidden subsets required to define $\Delta$. A set of forbidden subsets is minimal if and only if no two forbidden subsets are incomparable. We can call the minimal set of forbidden subsets defining $\Delta$ the \emph{circuit set} of $\Delta$. Every set of forbidden subsets can be turned into a circuit set by removing any subsets contained in other forbidden subsets.




	A \emph{forbidden subset diagram} is a hypergraph on a vertex set $\mathcal{S}$, whose edges are forbidden subsets. It is drawn with dashed edges connecting faces of two elements, and dashed shapes drawn around forbidden subsets containing more than 2 vertices, as in Figure \ref{fig:circuitdiagram}.

	\begin{figure}[h]
	\centering
	\includegraphics[width=.75\textwidth]{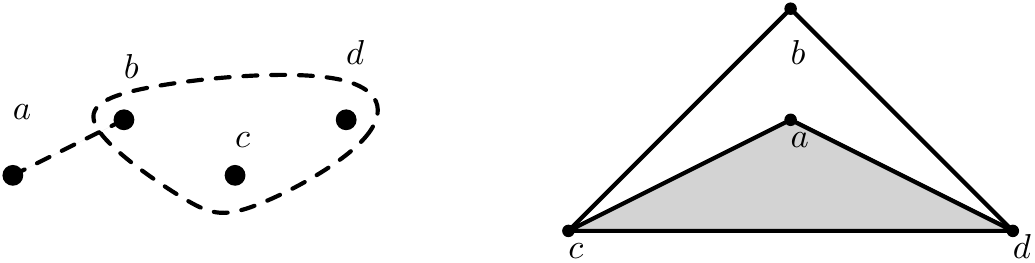}
	\caption{A forbidden subset diagram on the left, and the equivalent simplicial complex on the right.}\label{fig:circuitdiagram}
\end{figure}

	We refer to the forbidden subset diagram of the dual simplicial complex of a simple polyhedron as the forbidden subset diagram of that polyhedron. Figure \ref{fig:circuitdiagramprimitives} illustrates several low-dimensional examples. The forbidden subset diagram of a single ray is a graph containing one vertex and no edges. The forbidden subset diagram of a simplex on two vertices is a graph on two vertices with a single dashed edge, and the forbidden subset diagram of a simplex on $n$ vertices is a graph on $n$ vertices with a forbidden subset containing the entire graph.

	\begin{figure}[h]
	\centering
	\includegraphics[width=.75\textwidth]{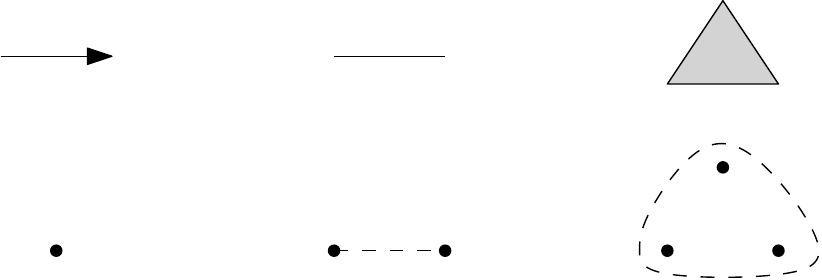}
	\caption{Simple polyhedra, and their forbidden subset diagrams below.}\label{fig:circuitdiagramprimitives}
\end{figure}

The following proposition is a trivial result from Proposition \ref{prop:polyhedral-product}:

	\begin{proposition}
		For two simple polyhedra $P_1, P_2$, the forbidden subset diagram of their Cartesian product, $P_1+P_2$, is equal to the disjoint union of the forbidden subset diagrams of $P_1$ and $P_2$.
	\end{proposition}

With this, we can very easily characterize the forbidden subset diagrams of Cartesian products of rays and simplices, especially 1-simplices whose forbidden subsets are dashed edges. Chapter \ref{chap:hypercubes} characterizes hypercube-graph associahedra using forbidden subset diagrams, and Chapter \ref{chap:res2} includes several cases, such as the Double Path Graph or Twisted Cycle Graph, whose facial enumeration requires enumeration of faces of $P$-graph associahedra where $P$ is the product of a hypercube and a set of rays. However, we note that a more fleshed-out theory of $\p$-graph associahedra is outside of the scope of this thesis.
%
%
%
%

	\section{Graphic matroid and edge-tubing complexes} \label{sub:subtrees}

One interesting application of $\Delta$-graph tubings is the theory of edge-tubings, or subtree tubings, presented here. Consider a graph $G=(V,E)$ consisting of vertex set $V$ and edge set $E$. Two edges $e_1, e_2$ are \emph{adjacent} if they share a vertex. The \emph{graphic matroid} is the simplicial complex $M_G$ with base set $E$ such that $E' \subseteq E$ is in $M_G$ if $(V,E')$ contains no cycles. These sets are called the \emph{independent sets} of the matroid $M_G$. The \emph{line graph} of a graph $G=(V,E)$ is the graph $L(G)=(E,E_{adj})$, where $\{e_1,e_2\}$ is an edge in $E_{adj}$ if edges $e_1,e_2 \in E$ share a vertex in $G$.

A \emph{subtree} of a graph $G$ is a collection of edges in $G$ which form a tree.

Given a subtree consisting of edges $e_1=\{u_1,v_1\}, \ldots, e_k=\{u_k,v_k\}$, the vertex set of that subtree is the set $e_1 \cup \cdots \cup e_k$, which is a subset of the vertex set $V$.

\begin{definition}
Two subtrees $t_1, t_2$ of a graph $G$ are compatible if and only if $t_1 \subseteq t_2, t_2 \subseteq t_1$, or the vertex sets of $t_1, t_2$ are disjoint.
\end{definition}

\begin{proposition}
Given a graph $G$, the simplicial complex of pairwise-compatible subtrees of $G$ is equal to the tubing complex $\mathcal{N}(L(G),M_G)$.
\end{proposition}

\begin{proof}
Consider a set $t$ of edges $e_1=\{u_1,v_1\}, \ldots, e_k=\{u_k,v_k\}$ of $G$. We note that $t$ induces a connected subgraph of $L(G)$ if and only if $t$ is connected in $G$. We also note that $t$ is cycle-free if and only if $t$ is an independent set in $M_G$. As a result, we find a set of edges is a subtree of $G$ if and only if it is a tube of the $M_G$-graph $L(G)$.

Now we consider weak compatibility rules. We need to show that $t_1, t_2$ are compatible subtree if and only if they are weakly compatible tubes in $L(G)$. This is true in the case that $t_1 \subseteq t_2$ or vice-versa. Now we note that two edges $e_1, e_2$ in $G$ have disjoint vertex sets if and only if there is no edge in $L(G)$ connecting the nodes $e_1, e_2$. From there, we see that two compatible disjoint subtrees have no shared vertices in their vertex sets if and only if $t_1$ and $t_2$ in $L(G)$ are not adjacent, and the subtrees $t_1, t_2$ are compatible as subtrees if and only if they are weakly compatible as tubes of $L(G)$.

Finally, we note that any union of subtrees with disjoint vertex sets cannot contain a cycle, and so any collection of weakly compatible tubes cannot have a union not in $L(G)$. As a result, a collection of subtrees is pairwise compatible in $G$ if and only if it is a tubing of $L(G)$ in $M_G$, proving our proposition.
\end{proof}

Call these tubings \emph{subtree tubings} of $G$. Figure \ref{fig:edge-tubings} shows a collection of subtree-tubings for a given graph, with subtrees drawn bold and in color. Note that the edges $abc$ do not form a subtree, and as a result do not form a tube in the $M_G$-graph $L(G)$.

\begin{figure}[h]
\centering
\includegraphics[width=.5\textwidth]{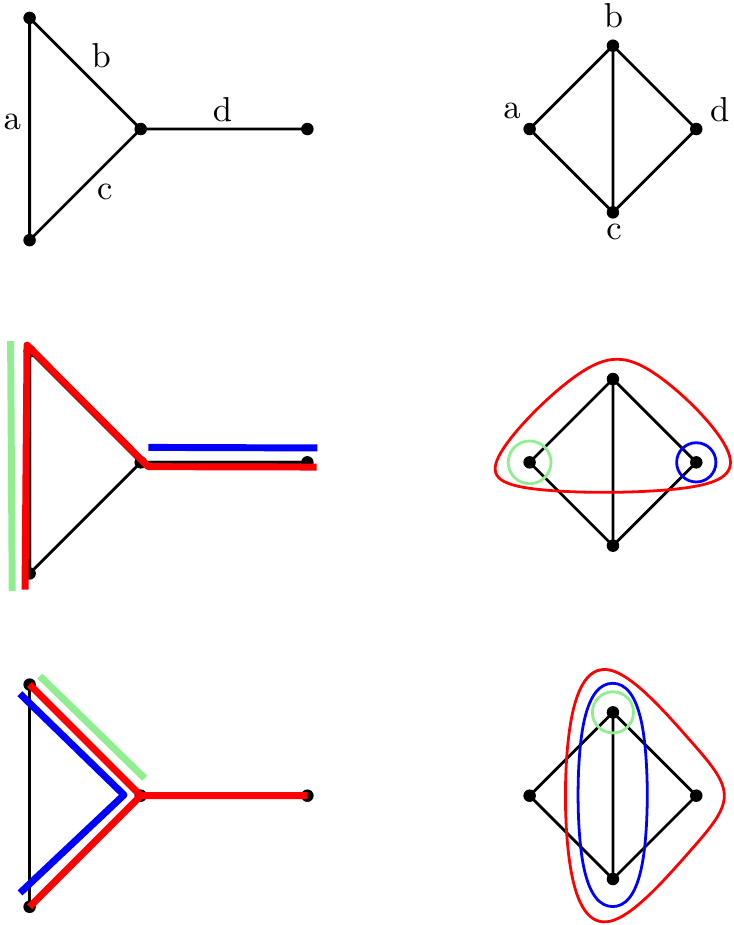}
\caption{An example graph (top left), its line graph (top right), and two subtree tubings with accompanying tubings in the line graph.}\label{fig:edge-tubings}
\end{figure}

We can draw subtree tubings similarly to standard tubings, by drawing shapes around edges instead of vertices. Figure \ref{fig:edge-tubings-cycle} shows a subtree tubing of a cycle on four vertices, and the corresponding graph tubing of its line-graph.

\begin{figure}[h]
\centering
\includegraphics[width=.5\textwidth]{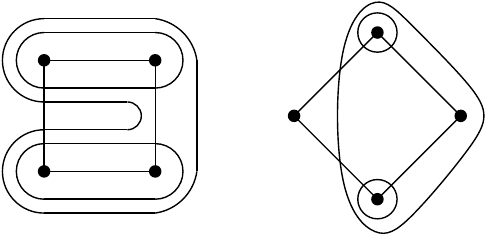}
\caption{subtree tubing of a cycle and its corresponding line-graph tubing.}\label{fig:edge-tubings-cycle}
\end{figure}

We are not familiar with any literature which describes these tubings. Clearly, when $G$ is a tree, $M_G$ is a Boolean lattice, and the set of subtree-tubings containing the entire graph $G$ is just equal to the classical graph tubing complex of $L(G)$; as a result, the graph associahedron of the line graph of a tree is connected to the subtree-tubing complex of that tree. Otherwise, we have not found these simplicial complexes in other contexts. 

%
%
%
%
\chapter{Normal fans of $\p$-nestohedra} \label{chap:normal-fans}

This chapter explores the normal fans of $\p$-nestohedra, in large part by restating results regarding preposets and the normal fans of classical nestohedra, and extending them to the $\Delta$-nested complex and $\p$-nestohedron cases. We start by describing preposets and braid cones in Sections \ref{sec:preposet-intro} and \ref{sec:classical-braid-cones}, describing existing results but providing an explicit realization of braid cones and focusing on preposets and braid cones defined by nested sets. We then introduce two new constructions, facial preposets and $\Delta$-braid cones, in Section \ref{sec:facial-preposets-etc}, which allow us to define fans from nested complexes. Section \ref{sec:fan-intersection} finds a way to define these nested complex fans as the coarsest common refinements of a set of simpler fans, and Section \ref{sec:barycentric-sub} characterizes nested complex fans as coarsenings of the barycentric subdivision of the fan of a simplicial complex. Finally, all of this comes together in Section \ref{sub:p-nestohedron-fans} to characterize the normal fans of $\p$-nestohedra.

\section{Preposets}\label{sec:preposet-intro}


A \emph{binary relation} $R$ on a set $S$ is a subset $R \subseteq S \times S$. A \emph{preposet} is a reflexive and transitive binary relation. This means that $(x,x) \in R$ for all $x \in S$, and if $(x,y),(y,z) \in R$, then $(x,z) \in R$. We will use the notation $x \preceq_R y$ to denote that $(x,y) \in R$.

An \emph{equivalence relation} $\equiv$ is a binary relation that is reflexive, symmetric, and transitive. Every preposet $Q$ defines an equivalence relation, defined such that $x \equiv_Q y$ if and only if $x \preceq_Q y$ and $y \preceq_Q x$. We can define a \emph{poset} as a preposet whose equivalence relation divides a set into singleton equivalence classes.
Every preposet $Q$ gives rise to a poset $Q / \equiv_Q$ on the equivalence classes of $S / \equiv_Q$.

Given two binary relations $R_1, R_2$ on a set $S$, let $R_1 \subseteq R_2$ denote containment as subsets of $S\times S$, and let $R_1 \cup R_2$ be the union of the two relations. If $R$ is a reflexive binary relation, define $\overline{R}$ to be the smallest preposet such that $R \subseteq \overline{R}$. We can call this the \emph{transitive closure} of $R$.

For preposets $P, Q$ on the same set $S$, note that the binary relation $P \cup Q$ is not necessarily a preposet. We note however that $\overline{P\cup Q}$ is.

A \emph{cover relation} in a preposet $Q$ is a special pair $x \lessdot_Q y$ such that $x \preceq_Q y$ and there exists no element $z$ such that $x \prec_Q z \prec_Q y$. The \emph{Hasse diagram} of a poset is an oriented graph on a poset whose edges are cover relations. The Hasse diagram of a preposet $Q$ is the Hasse diagram of the poset $Q/\equiv_Q$.

 Let $R^{op}$ be the \emph{opposite} of a binary relation $R$, such that $x \preceq_R y$ if and only if $y \preceq_{R^{op}} x$.

For two preposets $P$ and $Q$ on the same set, let us say that $Q$ is a \emph{contraction} of $P$ if there is a binary relation $R \subseteq P$ such that $Q = \overline{P\cup R^{op}}$. In other words, there is a way to obtain $Q$ from $P$ by merging certain equivalence classes in $P$ along relations in $P$, typically cover relations.

A preposet is a tree if its Hasse diagram is a tree. A preposet $Q$ is a \emph{rooted tree} if it is a tree, and for every equivalence class $I \in Q/\equiv_Q$, there is at most one set $J \in Q/\equiv_Q$ such that $I \preceq_{Q/\equiv_Q} J$. A forest is a preposet whose connected components are trees.

An \emph{order ideal} of a preposet $Q$ is a set $I$ such that, for $x,y \in Q$, if $x \preceq_Q y$ and $y \in I$, then $x \in I$.

%
%
%
%

\section{Preposets of classical nested sets and braid cones}\label{sec:classical-braid-cones}

Section 3 of \cite{postnikovfaces} outlines a bijection between preposets and braid cones. Braid cones are crucial to the study of generalized permutohedra, of which nestohedra are one example. In this section we will restate definitions and results for braid cones, before using them to define a generalization which we call \emph{$\Delta$-braid cones} or $\p$-braid cones.

\subsection{Preposets of classical nested sets}

\sloppy A \emph{principal order ideal} of a preposet $Q$ on vertex set $\mathcal{S}$ is an order ideal of the form ${I_y = \{x \in \mathcal{S}| x \preceq_Q y\}}$. The \emph{principal order ideal poset} of $Q$ is the collection of principal order ideals for all $y \in \mathcal{S}$ ordered by inclusion. For any preposet $Q$, the principal order ideal poset is isomorphic to the poset $Q/\equiv_Q$.

%

\danger{I have redefinted this preposet definition to something simpler to state, and slightly more general.}

\begin{definition} \label{def:preposet-from-nested-set-most-general}
For a set $N$ containing subsets of a base set $\mathcal{S}$, define a binary relation $P_N$ on the base set $\mathcal{S}$, such that $i \preceq j$ if and only if every set $I \in N$ that contains $j$ also contains $i$.
\end{definition}

We note that for a general set $N$ of subsets of a base set $\mathcal{S}$, this preposet is not unique and may not reflect all the structure of $N$. For example, if $N=\{\{1,2\},\{2,3\},\{1,3\}\}$ and $\mathcal{S}=\{1,2,3\}$, then $P_N$ is the empty preposet on $\mathcal{S}$. However, when $N$ is a nested set, we can define $P_N$ more concretely. Consider a classical building set $B$ on vertex set $\mathcal{S}$ and a $B$-nested set $N$. Remember that each classical nested set contains $B_{max}$ as a subset, so every vertex $i \in \mathcal{S}$ is contained in at least one set in $N$. For every element $i \in \mathcal{S}$, if $i \in I$ and $i \in J$ for $I, J \in N$, then the two sets cannot be disjoint or have nontrivial intersection, so either $I \subseteq J$ or $J \subseteq I$. As a result, for each element $i \in \mathcal{S}$, there is a unique smallest set $I_i \in N$ which contains $i$. The preposet $P_N$ can also be characterized in terms of these sets as follows.

\begin{proposition} \label{prop:smallest-containing-set-order-ideal}
For a classical $B$-nested set $N$ where $B$ has base set $\mathcal{S}$, the preposet $P_N$ is equal to the preposet on the base set $\mathcal{S}$ such that for any two elements $i, j \in \mathcal{S}$, $i \preceq j$ if and only if $I_i \subseteq I_j$.
\end{proposition}


We can note that every set $I_i$ in $N$ is a principal order ideal of $P_N$, and we can state the following.

\begin{proposition}\label{prop:nested-sets-are-ideals}
For any classical nested set $N$, $N$ is the set of principal order ideals of $P_N$.
\end{proposition}


Figure \ref{fig:graph-poset} shows a classical nested set $N$ on the set $\{1,2,3,4\}$, and the resulting preposet $P_N$. For any classical $B$-nested set $N$, the preposet $P_N$ is a forest of rooted trees. Each component of this preposet is a rooted tree on a subset of $B_{max}$.


\begin{figure}[h]
\centering
\includegraphics[width=.5\textwidth]{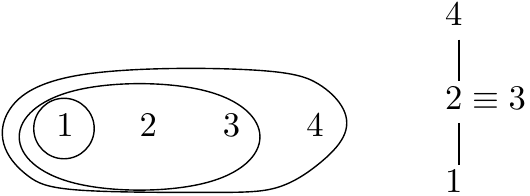}
\caption{A graph tubing and its associated preposet}\label{fig:graph-poset}
\end{figure}


\subsection{Defining braid cones}

The foundational \danger{I think that's a good choice of words} text on braid cones and faces of generalized permutohedra is \cite{postnikovfaces}. This subsection restates basic findings from this paper, and Section \ref{sec:facial-preposets-etc} of this thesis adapts their work to $\Delta$-braid cones.

For a finite set $\mathcal{S}$, define the vector space $\R^{\mathcal{S}}$ to be the $|\mathcal{S}|$-dimensional vector space with basis vectors $\{e_i|i \in \mathcal{S}\}$. A vector $x \in \R^{\mathcal{S}}$ contains components $x_i$ for $i \in \mathcal{S}$. Define the vector $(1,\ldots,1) = \sum_{i \in \mathcal{S}} e_i$. If $|\mathcal{S}| \ge 1$, we define the vector space $\R^{\mathcal{S}}/(1,\ldots,1)$ to be the set of equivalence classes of $\R^{\mathcal{S}}$ modulo $(1,\ldots,1)$. We write the equivalence class of a vector $x \in \R^{\mathcal{S}}$ as $[x]$, taking care not to confuse this notation with $[n]$ for integers. We note that comparison of components within one vector is well-defined on these equivalence classes; that is, if $[x] = [x']$, then $x_i \le x_j$ if and only if $x_i' \le x_j'$ for any $i,j$ in $[n+1]$. It also holds that $x_i-x_j=x_i'-x_j'$. We note that any $(|\mathcal{S}|-1)$ elements of $\mathcal{S}$ define a basis of $\R^{\mathcal{S}}$, and $\sum_{i \in \mathcal{S}} [e_i] = 0$.

A \emph{cone} in a vector space is any set that is closed under addition and multiplication by a nonnegative scalar. Polyhedral cones are cones that are polyhedra. Polyhedral cones can also be defined as any set in a vector space defined by a finite set of homogeneous linear inequalities of the form $ax \le 0$. A polyhedron is \emph{pointed} if it contains a vertex face, which for cones must be the origin of the vector space. Non-polyhedral cones exist, but are not covered in this thesis, and we will assume from here that any cone mentioned is a polyhedral cone. \danger{'assume for the scope of this thesis that we are referring to polyhedral cones' feels like a cop-out, and I never liked that language in textbooks, but I feel like it's probably the right thing to do huh.}


The \emph{braid arrangement} on a set $\mathcal{S}$ is the arrangement of hyperplanes in $\R^{\mathcal{S}}/(1,\ldots,1)$ defined by equalities of the form $x_i-x_j=0$ for $i \ne j$ in $\mathcal{S}$. When $\mathcal{S} = [n+1]$, this is the \emph{$n$-dimensional braid arrangement}, and these hyperplanes divide the space into cones of the form

\[
	C_\sigma = \{x_{\sigma(1)} \le x_{\sigma(2)} \le \cdots \le x_{\sigma(n+1)}\}
\]
where $\sigma$ is a permutation $\sigma \in \mathfrak{S}_{n+1}$. These cones are called \emph{Weyl chambers} of the type $A_n$ Coxeter group.

%
%
%
%

\begin{definition}
Given a preposet $Q$ on $\mathcal{S}$, define the \emph{braid cone} $K_Q$ of $Q$ as the cone in the space $\R^{\mathcal{S}}/(1,\ldots,1) \R$ defined by inequalities $x_i \le x_j$ for each relation $i \preceq_Q j$ in $Q$. \end{definition}


The following proposition is a rephrasing of \cite[Proposition 3.5]{postnikovfaces} in the notation used in this thesis.

\begin{proposition} \label{prop:preposet-list-1}
Given two preposets $Q, Q'$ on $\mathcal{S}$:
\begin{enumerate}
\item $K_{\overline{Q\cup Q'}}=K_Q \cap K_{Q'}$.
\item \label{statement:contractions} The preposet $Q$ is a contraction of $Q'$ if and only if $K_Q$ is a face of $K_{Q'}$.
\item \label{statement:poset} $Q$ is a poset if and only if $K_Q$ is full-dimensional, i.e., $(|\mathcal{S}|-1)$-dimensional.
\item The linear span of $K_Q$ is the cone $K_{\equiv_Q}$, where the equivalence relation $\equiv_Q$ is considered a preposet, and $K_{\equiv_Q}$ is the subspace defined by equations $x_i = x_j$ if $i \equiv_Q j$.
\item \label{statement:equiv-isomorphic-cone} The cone $K_{Q/\equiv_Q}$ is isomorphic to $K_Q$.
\item \label{statement:pointed-cones} $K_Q$ is pointed if and only if $Q$ is connected, i.e., the Hasse diagram of $Q$ is a connected graph.
\item If $Q$ is a poset, then the minimal set of inequalities describing $K_Q$ is $\{x_i \le x_j|i \lessdot_Q j\}$.
\item The Hasse diagram of $Q$ is a tree if and only if $\sigma$ is a full-dimensional simplicial cone.
\item When $\mathcal{S}=[n+1]$, for $\sigma \in \mathfrak{S}_{n+1}$, the cone $K_Q$ contains the Weyl chamber $C_\sigma$ if and only if $Q$ is a poset and $\sigma$ is its linear extension, that is $\sigma(1) \prec_Q \cdots \prec_Q \sigma(n+1)$.
\end{enumerate}
\end{proposition}

Recalling that $Q/\equiv_Q$ is a poset, we can state a corollary to statements \ref{statement:poset} and \ref{statement:equiv-isomorphic-cone}.

\begin{corollary}
If the preposet $Q$ has $k+1$ equivalence classes then the cone $K_Q$ is $k$-dimensional.
\end{corollary}

This comes from the fact that $K_{\equiv_Q}$ is a $k$-dimensional space, and $K_{Q \equiv_Q}$ is full-dimensional in that space, and isomorphic to $K_Q$.

For a classical nested set $N$, define the notation $K_N=K_{P_N}$. 

The following is a result from \cite{postnikovfaces}.

\begin{proposition}\label{prop:classical-BS}
For a classical building set $B$, the set of cones $K_N$ such that $N$ is $B$-nested forms a fan, and the face poset of this fan is isomorphic to the nested complex of $B$.
\end{proposition}

This fan is the normal fan of a classical $B$-nestohedron constructed by Minkowski sums of simplices as described in \cite{postnikov}. We note that for two $B$-nested sets, $K_N \cap K_{N'} = K_{N \cap N'}$. In addition, the three statements are equivalent: $K_N \subseteq K_{N'}$, $K_N$ is a face of $K_{N'}$, and $N \subseteq N'$.

\subsection{Expressing braid cones as conic hulls of their rays} \label{sub:waterfall}

\danger{I don't think we need any serious consideration of what this means for infinite sums, but this definition does allow us to define cones as conic hulls of either (1) a set of vectors contained in rays, or (2) the set of rays themselves.}

Define the \emph{conic hull} $\conichull V$ of a set of vectors $V=\{v_s|s \in S\}$ indexed by some set $S$ to be the set of finite sums of the form $\sum_{s \in S} a_s v_s$, where $a_s \ge 0$ for each $s \in S$. The conic hull of a single nonzero vector is a one-dimensional pointed cone, called a \emph{ray}, and the conic hull of a cone is itself. Polyhedral cones are exactly the cones that can be defined as conic hulls of a finite set of vectors. When a polyhedral cone is pointed, it is equal to the conic hull of a set of vectors $v_1, \ldots, v_k$, such that each vector $v_i$ is contained in the interior of a ray face of the cone.

Every polyhedral cone is equal to the conic hull of a set of vectors, and every polyhedral pointed cone is equal to the conic hull of a set of vectors containing one vector in each ray face of the cone.

\begin{definition}
A \emph{ray preposet} is a preposet $R$ containing two equivalence classes: a \emph{lower class}, and an \emph{upper class}, such that $i \prec j$ for all $i$ in the lower class and $j$ in the upper class of $R$.
\end{definition}

We know from Proposition \ref{prop:preposet-list-1} that the cone $K_Q$ is pointed if and only if $Q$ is connected, and $K_Q$ is one-dimensional if and only if it is a preposet with 2 equivalence classes, so we can see $K_Q$ is a ray if and only if $Q$ is a ray preposet.

A \emph{ray contraction} of a preposet $Q$ is a ray preposet which is a contraction of $Q$. Ray contractions are obtained by contracting edges one-by-one in the Hasse diagram of a preposet until only one edge remains. Note that in general, there is not a bijection between edges in the Hasse diagram and rays of a preposet cone; Figure \ref{fig:ray-contractions} illustrates two possible ray contractions which arise from leaving one edge alone and contracting the rest of the edges. The choice and order of edge contraction generally matters.

In general, we can characterize braid cones as follows.

\begin{proposition}
Given a connected preposet $Q$ on a set $S$, the cone $K_Q$ is equal to the conic hull of the rays of the form $K_R$, where $R$ is a ray contraction of $Q$.
\end{proposition}

\begin{proof}
From Proposition \ref{prop:preposet-list-1}, we know from statement \ref{statement:pointed-cones} that connected preposets have pointed cones, and pointed hulls are equal to the conic hulls of ray faces. From statement \ref{statement:contractions}, we know that all faces of a cone $K_Q$ are cones of the face $K_{Q'}$ where $Q'$ is a contraction of $Q$, and we know a cone $K_{Q'}$ is a ray if and only if $Q'$ is a ray preposet, so $K_Q$ is equal to the conic hull of all rays of the form $K_R$, where $R$ is a ray contraction of $Q$.
\end{proof}

	\begin{figure}[h]
	\centering
	\includegraphics[width=.5\textwidth]{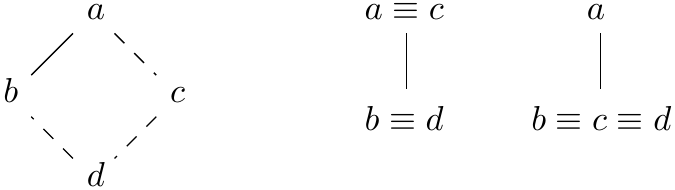}
	\caption{Contracting until no dashed edges remain leads to two possible ray contractions.} \label{fig:ray-contractions}
\end{figure}

When the Hasse diagram of a preposet $Q$ is a connected tree, edge contraction order does not matter, and ray contractions are in bijection with individual edges of the Hasse diagram. When $Q$ is also a rooted tree, then it is even simpler to characterize the ray contractions of $Q$.

\begin{proposition}\label{prop:rooted-ray-contractions}
If the preposet $Q$ on base set $\mathcal{S}$ is a rooted tree, then $R$ is a ray contraction of $Q$ if and only if the lower class of $R$ is the principal order ideal of a non-maximal element in $Q$.
\end{proposition}




\begin{lemma}
For any ray preposet $R$ with lower class $I$, $K_R$ is the conic hull of the vector $\sum_{i \in I} [-e_i]$.
\end{lemma}

\begin{proof}

Write $v = \sum_{i \in I} [-e_i]$. This is the equivalence class of the vector $x$ in $\R^{\mathcal{S}}$ with $x_i=0$ if $i \notin I$ and $x_i=-1$ if $i \in I$. We note then that $v_i = v_j$ if $i, j \in I$ or $i, j \notin I$, and $v_i < v_j$ if $i \in I, j \notin I$. These are then exactly the conditions regarding whether or not $i \equiv_R j$ or $i <_R j$ in $R$. As a result, $v_i \le v_j$ if and only if $i \le_R j$. This means that $v$ is in the interior of $K_R$, and as $K_R$ is a ray, this proves $K_R$ is equal to the conic hull of $v$.
\end{proof}

\begin{definition}
Given a preposet $Q$ on a set $S$, the \emph{principal order ideal vector} of an element $j \in S$ is the sum \[w_j^Q = \sum_{i \preceq_q j} [-e_i]
\]
\end{definition}


\begin{proposition} \label{prop:rooted-tree-conic-hull}
If $P/\equiv_P$ is a rooted tree on base set $\mathcal{S}$, then the cone $K_P$ is the conic hull of principal order ideal vectors $w_i^P$ for all $i \in \mathcal{S}$.
\end{proposition}

\begin{proof}
Recall that $K_P$ is equal to the conic hull of all cones of the form $K_R$, where $R$ is a ray contraction of $P$. Because $P$ is a rooted tree, the only ray contractions of $P$ are those which have lower class $I_i$, where $I_i$ is the principal order ideal of some non-maximal element $i \in \mathcal{S}$. As a result, if $i$ is a non-maximal element in $P$, and $R$ is a ray contraction with lower class $I_i$, then $K_R$ is equal to the conic hull of the vector $w_i^P$.

As a result, $K_P$ is equal to the conic hull of principal order ideal vectors $w_i^P$ for all non-maximal elements $i \in P$. When $m$ is a maximal element in $P$, we find that $w_m^P=\sum_{j \in S} [-e_j]$. As a result, $w_m^P \equiv 0$, and $K_P$ is equal to the conic hull of all vectors $w_i^P$ for $i \in \mathcal{S}$.
\end{proof}

In the special case where $N$ is a classical nested set of $B$ with $|B_{max}|=1$, Proposition \ref{prop:nested-sets-are-ideals} states that the sets of $N$ are exactly the principal order ideals of the preposet $P_N$. We note that $P_N$ is a rooted tree, so Proposition \ref{prop:rooted-tree-conic-hull} implies the following:

\begin{proposition}\label{prop:classical-nested-set-conic-hull}
For a nested set $N$ of a connected classical nested set $B$, the cone $K_N$ is equal to \allowbreak {${\conichull\{\sum_{i \in I} [-e_i]| I \in N\}}$}.
\end{proposition}


\section{Facial Preposets, $\Delta$-nested sets, and $\Delta$-braid cones}\label{sec:facial-preposets-etc}

This material was motivated by the desire to characterize the normal fans of $\p$-nestohedra in a manner analogous to the normal fans of classical nestohedra. Speaking informally, the normal fan of a $\p$-nestohedron looks like portions of the normal fans of classical nestohedra, stitched together. When two nested sets share the same support, the related cones in the normal fan behave very similarly to the braid cones of the related classical nested sets. However, we need to be able to manage the fact that two $\B$-nested sets for a $\Delta$-building set $\B$ will likely not have the same support. If we want a set of propositions similar to Proposition \ref{prop:preposet-list-1} for $\Delta$-braid cones, we will need to deal with the fact that preposets defined from $\B$-nested sets should have the same support so that contractions and unions of preposets are well-defined, which leads us to the decision to define facial preposets the way we do in the following subsection.

\subsection{Introducing facial preposets}

Consider a simplicial complex $\Delta$ with base set $S$. We introduce an element $\infty$, and if $Q$ is a preposet on $\mathcal{S} \cup \{\infty\}$, then define the \emph{finite elements} of $Q$ to be the elements $\{x \in \mathcal{S}| x \prec_Q \infty\}$, and the \emph{infinite elements} to be the elements $\{x \in \mathcal{S}| x \equiv_Q \infty\}$. Now define $Q_{finite}$ to be the preposet on the finite elements of $Q$, ordered by $\preceq_Q$.

\begin{definition}
Given a simplicial complex $\Delta$ on base set $\mathcal{S}$, a \emph{facial preposet} of $\Delta$ is a preposet on $\mathcal{S} \cup \{\infty\}$ if $x \preceq_Q \infty$ for all $x \in \mathcal{S}$, and the base set of $Q_{finite}$ is a face of $\Delta$.
\end{definition}

We note that every facial preposet has a connected Hasse diagram. We also note that every contraction of a facial preposet is a facial preposet.

Facial preposets will be used to define cones we will call \emph{$\Delta$-braid cones}. Subsection \ref{sub:delta-braid-cones-from-nested} describes a method for defining $\Delta$-braid cones from $\Delta$-nested sets, and defining cones corresponding to $\Delta$-preposets is our primary motivation for introducing facial preposets.

\subsection{$\Delta$-braid cones arising from facial preposets}

Classical braid cones are traditionally defined by linear inequalities, and we have been able to provide an alternate realization for some of them by finding sets of vectors whose conic hull is a braid cone. In this subsection we will define a new class of cones from facial preposets, called $\Delta$-braid cones, defined primarily as conic hulls of certain vectors. The definition of these $\Delta$-braid cones is defined specifically to allow for the definition of cones from $\Delta$-nested sets in subsection \ref{sub:delta-braid-cones-from-nested}.

For a simplicial complex $\Delta$ on set $S$ with indexed vector set $V=\{v_s|s \in V\}$, define $\mathcal{F}(\Delta,V)$ to be the set of cones $C^V_I=\conichull\{v_s|s \in I\}$ for each face $I \in \Delta$. This is not always a fan; later we will only focus on the case where $\Delta$ is the dual simplicial complex of a simple polyhedron, which will guarantee that we can make a fan to represent $\Delta$. For now, call it a cone complex.


\begin{definition}
A cone complex $\mathcal{F}(\Delta,V)$ is \emph{non-degenerate} if $C^V_I \cap C^V_{I'}=C^V_{I \cap I'}$ for all $I, I' \in \Delta$, and each cone $C^V_I$ is $|I|$-dimensional.
\end{definition}

We recall that a fan is a collection of cones such that any intersection of two cones in the fan is a cone in the fan. It is possible for a cone complex to fail to be a fan, and it is also possible for a cone complex to be a fan but be degenerate by failing dimensionality requirements. Figure \ref{fig:degenerate-fans} illustrates these possible cases in two dimensions. When $\mathcal{F}(\Delta,V)$ is non-degenerate, the poset of cones defined by inclusion is isomorphic to $\Delta$.

\begin{figure}[h]
\centering
\includegraphics[width=.8\textwidth]{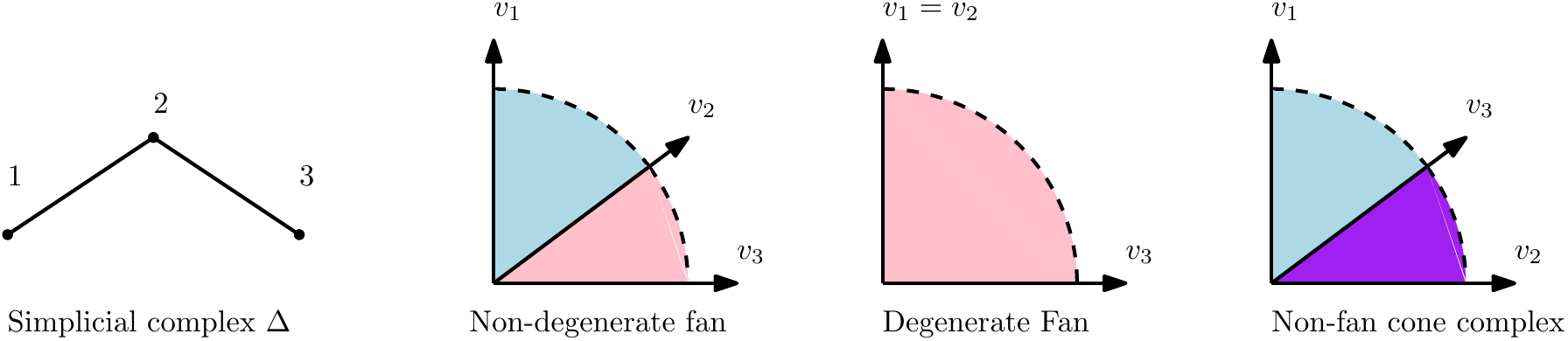}
\caption{A non-degenerate fan and two degenerate cone complexes.}\label{fig:degenerate-fans}
\end{figure}

Ray preposets are connected preposets with two equivalence classes, and so we know that only connected preposets have ray contractions. When $Q$ is a facial preposet, $Q$ is connected, and so there exist ray contractions of $Q$. In addition, these ray contractions are all facial preposets. In the same way that we can characterize braid cones of connected preposets as conic hulls of rays defined by ray contractions, we can define $\Delta$-braid cones as conic hulls of rays defined by ray contractions.


\begin{definition}
When a facial preposet $R$ is a ray preposet, define the \emph{$\Delta$-braid cone} $K_R^V$ to be the ray of the vector $\sum_{i \in R_{finite}} v_i$. When a facial preposet $Q$ is not a ray preposet, define $K_Q^V$ to be the conic hull of all rays $K_R^V$ where $R$ is a ray contraction of $Q$.
\end{definition}

\subsection{Properties of $\Delta$-braid cones}

In this subsection, we wish to prove a $\Delta$-braid cone analogue to Proposition \ref{prop:classical-BS}.
First, we will have to prove an analogue to Proposition \ref{prop:preposet-list-1}. We do this for certain cones by defining linear isomorphisms which map $\Delta$-braid cones onto classical braid cones and which allow us to apply Proposition \ref{prop:preposet-list-1} directly. For example, in order to characterize the intersection of two $\Delta$-braid cones, we will find a linear isomorphism mapping both cones onto classical braid cones, and then characterize the intersection of their images.

A map taking $\Delta$-braid cones to classical braid cones implies a map from facial preposets to preposets. There are two obvious ways to turn facial preposets into preposets. The first method, removing the element $\infty$ from a facial preposet, is very useful for characterizing the normal fan of a simplex-nestohedron as both a classical nestohedron and as a $\p$-nestohedron, and we devote Subsection \ref{sub:simplex-case} to studying this correspondence. The second method for defining a preposet from a facial preposet $Q$ is to forget that $Q$ is a facial preposet, and treat the element $\infty$ in $Q$ like any other element. We note that if $Q$ is a facial preposet on the set $S$, then $Q$ is a preposet on the set $S \cup \{\infty\}$. When we treat $Q$ as a preposet, we will refer to it as the \emph{finitized} preposet of $Q$.

\begin{lemma} \label{lemma:finitized}
For a facial preposet $Q$ on base set $\mathcal{S}$, define the vector set $V$ in the vector space $\R^\mathcal{S}/(1,\ldots,1)$ such that $v_{i} = [-e_i]$ for each $i \in \mathcal{S}$. We find that the cone $K_Q$ of the finitized preposet $Q$ is equal to the cone $K^V_Q$.
\end{lemma}

\begin{proof}
We first prove this for the case that $Q$ is a ray preposet as well as being a facial preposet. In this case, if $Q$ has lower class $I$, we find that $Q_{finite}$ is the preposet on $I$ such that all elements are equivalent, and we find $K^V_Q$ is equal to the conic hull of the single vector $\sum_{i \in I} v_i=\sum_{i \in I} [-e_i]$, which is equal to $K_Q^V$.

Now consider the case where $Q$ is not a ray preposet. We note that every contraction of a facial preposet is a facial preposet, so if $R$ is a ray contraction of $Q$ then it is also a facial preposet. As a result, the conic hull of $K_Q$ is equal to the conic hull of all rays $K_R$ where $R$ is a ray contraction of $Q$, which is equal to the conic hull of $K_R^V$ for any ray contraction $R$ of $Q$, which we have defined as $K_Q^V$. 
\end{proof}

\begin{proposition}  \label{prop:preposet-list-2}
Consider two facial preposets $Q, Q'$ of a simplicial complex $\Delta$, such that $Q/\equiv_Q, Q'/\equiv_{Q'}$ are rooted trees. Assume that $V$ is a set of vectors such that $\mathcal{F}(\Delta,V)$ is non-degenerate.
\begin{enumerate}
\item If $I$ is a face of $\Delta$, then the cone $C^V_I$ is equal to the cone $K_{Q''}$, where $Q''$ is the facial preposet defined as follows: $i \prec j$ for all $i \in I, j \in \mathcal{S}\backslash I$, and $j \equiv \infty$ for all $j \in \mathcal{S}\backslash I$.

\item If $I$ is a face of $\Delta$, then  the base set of $Q_{finite}$ is a subset of $I$ if and only if the cone $K_Q^V$ is contained in the cone $C^V_I$.

\item The preposet $Q$ is a contraction of $Q'$ if and only if $K_Q^V$ is a face of $K_{Q'}^V$.

\item \label{statement:cone-intersection} The cone $K_{\overline{Q \cup Q'}}$ is equal to the cone $K_Q^V \cap K_{Q'}^V$.

\item $Q \subseteq Q'$ if and only if $K_{Q'}^V \subseteq K_Q^V$.

\item The dimension of $K_Q^V$ is the number of equivalence classes in $Q_{finite}$.
%
\end{enumerate}
\end{proposition}
\begin{proof}
\begin{enumerate}
\item The facial preposet $Q''$ has a finite set whose equivalence classes are the singleton subsets of $I$, and is such that all elements of $I$ are incomparable in the preposet. This facial preposet is a rooted tree, with maximal equivalence class $(\mathcal{S} \cup \{\infty\}) \backslash I$, and edges in the Hasse diagram between the equivalence class and the elements in $I$. As a result, the ray contractions of this preposet are the ray preposets whose lower class only contains one element in $I$. The conic hull of these rays is the conic hull of the vectors $\{v_i|i \in I\}$, which is exactly the cone $C^V_I$.
%

\item 

Consider the case where the base set of $Q_{finite}$ is a subset of $I$. Every facial ray contraction $R$ of $Q$ has a finite set that is a subset of $I$, and so is a sum of a subset of vectors of the form $v_i$ for $i \in I$. As a result, $K_Q^V$ is a conic hull of a collection of vectors in the cone $C^V_I$, and so $K_Q^V \subseteq C^V_I$.

Now consider the case that $Q_{finite}$ is not a subset of $I$. This means there exists at least one element $j \in Q_{finite}$ such that $j \notin I$. Define $J$ to be the principal ideal of $j$ in $Q_{finite}$. The vector $v_J= \sum_{i \in J} v_i$ is in $K_Q^V$ but not in $C^V_I$, and is contained in the interior of $C^V_J$. Because $\mathcal{F}(\Delta,V)$ is a fan, we note that the interior of $C^V_J$ intersects with $C^V_I$ if and only if $C^V_J \subseteq C^V_I$. Because $\mathcal{F}(\Delta,V)$ is non-degenerate, we also note that $C^V_J \subseteq C^V_I$ if and only if $J \subseteq I$, which is not the case. As a result, $v_J \in K_Q^V$ is not contained in $C^V_I$, and $K_Q^V \not\subseteq C^V_I$.

\item \label{item:prop-with-isomorphism} First, note that if $Q$ is a contraction of $Q'$, then the base set of the finite set of $Q$, which we are writing as $I$, is a subset of the base set of the finite set $I'$ of $Q'$, and $K_Q^V, K_{Q'}^V \subseteq C^V_{I'}$. Secondly, if $K_{Q}$ is a face of $K_{Q'}$, then $K_Q \subseteq C^V_{I'}$ and $I \subseteq I'$. As a result, it is sufficient to prove that this is true in the case where $I \subseteq I'$.

 Define a linear map on the linear span of $C^V_{I'}$, into the space $\R^{I'\cup\{\infty\}}/(1,\ldots,1)$, defined by $T(v_i) = [-e_i]$ for $i \in I'$. We note that $\{v_i|i \in I'\}$ are linearly independent, and the set of vectors $[-e_i]$ for $i \in I'$ are all linearly independent, with $\sum_{i \in I'} [-e_i] = [e_\infty]$. This map is therefore a linear isomorphism. We can see from Lemma \ref{lemma:finitized} and the definition of $\Delta$-braid cones that $T$ is an isomorphism such that $T(K_{Q'}^V)=K_{Q'}$ and $T(K_Q^V)=K_Q$. In addition, from Proposition \ref{prop:preposet-list-1}, $K_Q$ is a face of $K_{Q'}$ if and only if $Q$ is a contraction of $Q'$, and the linear isomorphism proves that $K_Q^V$ is a face of $K_{Q'}^V$ if and only if $Q$ is a contraction of $Q'$.

\item We wish to prove that $K_{\overline{Q \cup Q'}}^V = K^V_Q \cap K^V_{Q'}$. Note that the linear isomorphism $T$ defined in part (3) cannot be defined on the linear span of the vectors $\{v_i|i \in I \cup I'\}$, as the vectors might not be linearly independent. However, we know the cone $C^V_{I \cap I'}$ is simplicial, and the vectors $\{v_i|i \in I \cap I'\}$ must be independent. As a result, it is our goal in this proof to find new preposets $P$ and $P'$, such that $K_P^V=K_Q^V \cap C^V_{I'}$ and $K_{P'}^V=K_{Q'}^V \cap C^V_I$, and then prove that $\overline{P \cup P'}=\overline{Q \cup Q'}$, and $K_P^V \cap K_{P'}^V=K_{\overline{P\cup P'}}^V$.

We define $P$ as the preposet defined by taking $Q$ and defining $x \equiv_P \infty$ for all $x \in I\backslash I'$. This removes elements from $Q_{finite}$ that are not contained in $Q'_{finite}$. The finite set of $P$ is a subset of $I \cap I'$, but may possibly not be equal to $I \cap I'$. We find that $P$ is a contraction of $Q$, and therefore $K_P^V$ is a face of $K_Q^V$ and is contained in $C^V_{I \cap I'}$. Note that $C^V_{I \cap I'}$ is a face of $C^V_{I}$, which contains $K_Q^V$, so $C^V_{I \cap I'} \cap K_Q^V$ is a face of $K_Q^V$. From this, we find that $K_P^V$ is a face of $K_Q^V \cap C^V_{I \cap I'}$. However, $P$ is the minimal contraction of $Q$ with finite set contained in $I \cap I'$. As a result, $K_P^V=K_Q^V \cap C^V_{I \cap I'}$.

We can define $P'$ in an analogous way, being the minimal contraction of $Q'$ with finite set contained in $I \cap I'$, and find $K_{P'}^V = K_{Q'}^V \cap C^V_{I \cap I'}$. We then see that $K_P^V \cap K_{P'}^V = K_Q^V \cap K_{Q'}^V$.

Now consider $\overline{P \cup P'}$. It is trivial to see that $\overline{P \cup P'}=\overline{Q \cup Q'}$. All we need to do now is prove that $K_{\overline{P\cup P'}}^V=K_P^V \cap K_{P'}^V$. Now, we need only define the map $T$, defined before for part (3), on the set of vectors $\{v_i|i \in I \cap I'\}$. With this linear isomorphism, we can find a mapping taking $K_P^V$ to $K_P$, $K_{P'}^V$ to $K_{P'}$, and $K_{\overline{P\cup P'}}^V$ to $K_{\overline{P\cup P'}}$. Using Proposition \ref{prop:preposet-list-1}, we find that $K_{\overline{P\cup P'}}=K_P \cap K_{P'}$, and because of the linear isomorphism $T$, we have proven our statement.

\item Note that $Q \subseteq Q'$ if and only if $Q'=\overline{Q \cup Q'}$, which is true if and only if $K_{Q'}^V = K_Q^V \cap K_{Q'}^V$, which is true if and only if $K_{Q'}^V \subseteq K_Q^V$.

\item The cone $K_Q^V$ is isomorphic to the cone $K_Q$, which is the full-dimensional cone of $Q/\equiv_Q$ inside the linear span of $K_Q$. The preposet $Q/\equiv_Q$ has $k+1$ elements, where $k$ is the number of equivalence classes of $Q_{finite}$, and so $K_Q$ is $k$-dimensional, and $K_Q^V$ is $k$-dimensional.

\end{enumerate}
\end{proof}

Note that several statements in Proposition \ref{prop:preposet-list-1} do not have parallels in Proposition \ref{prop:preposet-list-2}. For instance, every facial preposet is connected by the equivalence class containing infinity, and so every $\Delta$-braid cone is pointed, which is not the case for general preposets and braid cones.

\subsection{$\Delta$-facial preposets arising from $\Delta$-nested sets}

Definition \ref{def:preposet-from-nested-set-most-general} defines a preposet $P_N$ from a set $N$ of subsets of a base set $\mathcal{S}$. We define a similar construction to $P_N$, but this time defining a preposet on a base set including the element $\infty$.


\begin{definition}
Consider a simplicial complex $\Delta$ on base set $\mathcal{S}$ and a $\Delta$-building set $\B$. For a $\B$-nested set $N$, define $Q(N)$ to be the preposet on $\mathcal{S} \cup \{\infty\}$ such that $i \preceq j$ if and only if every set in $N$ which contains $j$ also contains $i$.
\end{definition}

We note that $\infty$ is contained in no set in $N$, so $i \preceq \infty$ for all $i \in \mathcal{S}$, and $i \equiv \infty$ if and only if $i \notin \bigcup N$, meaning the finite set of $Q(N)$ is $\bigcup{N}$, which is a face of $\Delta$. As a result, $Q(N)$ is a facial preposet.

We note that $Q(N)$ defined for a nested set over a building set with base set $\mathcal{S}$ is equal to $P_N$ defined on the base set $\mathcal{S}\cup\{\infty\}$. For a $\B$-nested set $N$, we find that for each element $i$ in the set $\bigcup N=\bigcup_{I \in N} I$, there is a unique minimal set $I_i \in N$ such that $i \in I_i$. The following proposition is directly analogous to Proposition \ref{prop:smallest-containing-set-order-ideal}, but applies to facial preposets of $\Delta$-nested sets instead of preposets of classical nested sets.

\begin{proposition}
The facial preposet $Q(N)$ defined by a $\B$-nested set $N$ of a simplicial complex $\Delta$ is equal to the facial preposet of $\Delta$ such that if $i, j \in \bigcup N$, then $i \preceq j$ if and only if $I_i \subseteq I_j$, and $i \equiv \infty$ for all $i \in S \backslash (\bigcup N)$.
\end{proposition}


We note that $Q(N)$ is a rooted tree, and $Q(N)_{finite}$ is a forest of rooted trees. We will use the notation $K_{N}^V = K_{Q(N)}^V$ to denote the $\Delta$-braid cone defined by a $\Delta$-nested set.


\begin{proposition} \label{prop:nested-subset-preposet}
For two $\B$-nested sets $N, N'$ of a simplicial complex $\Delta$, the following statements are equivalent:
\begin{enumerate}
\item $N \subseteq N'$
\item $Q(N)$ is a contraction of $Q({N'})$
\item $Q({N'}) \subseteq Q(N)$
\item If $V$ is a vector set such that $\mathcal{F}(\Delta,V)$ is non-degenerate, then $K_{N}^V$ is a face of $K^V_{N'}$.
\item If $V$ is a vector set such that $\mathcal{F}(\Delta,V)$ is non-degenerate, then $K_{N}^V \subseteq K_{N'}^V$.
\end{enumerate}
\end{proposition}

\begin{proof}

First we should note some analogous statements for classical $B$-nested sets. If $M, M'$ are $B$-nested for a classical building set $B$, then the following four statements are equivalent: $M \subseteq M'$, $P_{M'} \subseteq P_{M}$, $P_M$ is a contraction of $P_{M'}$, and $K_M$ is a face of $K_{M'}$.

Consider the case where $N'=N \cup \{I\}$. We wish to prove that $Q(N)$ can be obtained by contracting a single cover relation in $Q(N')$. If $I$ is maximal in $N'$, define the set $J=I \cap Q(N)_{finite}$. We leave it as trivial that $Q(N)$ is found by contracting the equivalence classes $J$ and $Q(N)_{infinite}$ in $Q(N')$. As a result, $Q(N')$ is a contraction of $Q(N)$ in this case. Alternatively, if $I$ is not maximal in $N'$, then the base sets of $Q(N)_{finite}$ and $Q(N')_{finite}$ are equal. If we restrict $\B$ to the set $\bigcup N'$, then we find $N, N$ are both nested sets of a classical building set $B=\{S \in \B| S \subseteq \bigcup N\}$. Because $N, N'$ are $B$-nested and $N \subseteq N'$, we know that $P_{N'}$ is a contraction of $P_N$. Because $P_N=Q(N)_{finite}, P_{N'}=Q(N)_{finite}$, this means $Q(N')$ is a contraction of $Q(N)$. This argument has applied to the case that $|N'|=|N|+1$, but applied inductively, we find (1) implies (2).

Proposition \ref{prop:preposet-list-2} means that (2) implies (3).

We wish to prove that (3) implies (1). Assume that $Q(N') \subseteq Q(N)$. For an element $i$ in the preposet $Q(N)$, define principal ideals $I_i=\{x| x \preceq_{Q(N)} i\}$, and $I'_i=\{x|x \preceq_{Q(N')} i\}$. We find that for each $i$ in the base set $\mathcal{S}$ of $\Delta$ that either $I_i \in N$ or $I_i=\mathcal{S} \cup \{\infty\}$, and similarly for $Q(N')$.

We also can see that each set $I_i$ is equal to the union $\bigcup_{j \in J} I'_j$ for some set $J \subseteq \mathcal{S}$. We wish to prove that $I_i \in N'$. If not, then there exists a minimal subset $J \subseteq \mathcal{S}$ such that $\bigcup_{j \in J} I'_j=I_i$. Because this set is minimal, we find $I'_j, I'_{j'}$ are incomparable for any $j, j' \in J$. As a result, there exists a subset of $N'$ of order $\ge 2$ whose union is $I_i$. We know that $I_i \in \B$, which is a contradiction. Proposition \ref{complex-nestedset} states that no union of a subset of order $\ge 2$ of incomparable elements of a $\B$-nested set is equal to a set in $\B$. As a result, we find that every set $I \in N$ is contained in $N'$, and $N \subseteq N'$.

We know (2) and (4) are equivalent according to Proposition \ref{prop:preposet-list-2}. Similarly, we know (3) and (5) are equivalent according to the same proposition.
\end{proof}

\begin{proposition} \label{prop:facial-preposet-union-intersection}
For two $\B$-nested sets $N, N'$, we find $\overline{Q(N') \cup Q(N)}=Q(N \cap N')$.
\end{proposition}

\begin{proof}

Note that $Q(N \cap N')$ is the smallest common contraction of $Q(N)$ and $Q(N')$, and as a result $\overline{Q(N) \cup Q(N')} \subseteq Q(N\cap N')$. However, proving that the two preposets are actually equal is not trivial. We will instead rely upon the fact that when $B$ is a classical building set, and $M, M'$ are $B$-nested, then the set of cones of $B$-nested sets is a fan, and specifically, $K_M \cap K_{M'}=K_{M \cap M'}$ and $\overline{P(M) \cup P(M')}=P(M \cap M')$. We will use a linear isomorphism similar to the one used in Proposition \ref{prop:preposet-list-2} part (4).



For any face $I \in \Delta$, define the facial preposet $Q_I$ as the unique facial preposet with finite set equal to $I$, and $i, j$ incomparable for all distinct $i, j \in I$. We wish to prove a lemma that $\overline{Q_I \cup Q(N)}=Q(M)$, where $M=\{S \in N|S \subseteq I\}$.

We can see that any elements not in $I$ are contracted to $\infty$ in this preposet, and so we note that the base set of $(\overline{Q_I \cup Q(N)})_{finite}$ is the set of all elements $i \in I$ such that $j \not \preceq i$ for any $j \notin I$. In addition, $(\overline{Q_I \cup Q(N)})_{finite}$ is a subpreposet of $Q(N)_{finite}$, with $i \preceq j$ in $(\overline{Q_I \cup Q(N)})_{finite}$ if and only if $i, j \in (\overline{Q_I \cup Q(N)})_{finite}$ and $i \preceq_{Q(N} j$.

For any element $i \in Q(N)$, the principal order ideal $I_i$ of $i$ in $Q(N)$ is the smallest set in $N$ containing $i$. We note that $j \in I_i$ if and only if $j \preceq_{Q(N)} i$. We also note that $I_i \subseteq I$ if and only if $j \in I$ for all $j \preceq_{Q(N)}$. Not only is $i \in (\overline{Q_I \cup Q(N)})_{finite}$, but the entire set $I_i$ is a subset, and $I_i$ is the principal order ideal of $i$ in $(\overline{Q_I \cup Q(N)})_{finite}$.

As a result, we have found that for each $i \in \bigcup{N}$, $i \in (\overline{Q_I \cup Q(N)})_{finite}$ if and only if $I_i \subseteq I$, and the principal order ideal of $i$ in $(\overline{Q_I \cup Q(N)})_{finite}$ is $I_i$. This means that the set of principal order ideals of $(\overline{Q_I \cup Q(N)})_{finite}$ is equal to the set of all sets $I_i \in N$ such that $I_i \subseteq I$. This is exactly the set $M$, and because $M \subseteq N$, $M$ is a $\B$-nested set. This means that because the set of principal order ideals of $(\overline{Q_I \cup Q(N)})_{finite}$ is equal to $M$, we find $(\overline{Q_I \cup Q(N)})_{finite} = Q(M)_{finite}$, and therefore $\overline{Q_I \cup Q(N)} = Q(M)$, proving this lemma.

Now we wish to apply this lemma. Define $I=\bigcup{N'}$ and $I'=\bigcup{N}$. We can define nested sets $M=\{S \in N|S \subseteq I\}$ and $M'=\{S \in N'|S \subseteq I\}$. As we have just proven, $Q(M) = \overline{Q_I \cup Q(N)}$ and $Q(M') = \overline{Q_{I'} \cup Q(N')}$. In addition, statement (1) of Proposition \ref{prop:preposet-list-2} shows that $K_{Q_I}^V = C^V_I$, and $K_{Q_{I'}}=C^V_{I'}$. For any $\Delta$ we can find a vector set $V$ such that $\mathcal{F}(\Delta,V)$ is non-degenerate, and from statement \ref{statement:cone-intersection} of Proposition \ref{prop:preposet-list-2}, we see that $K_M^V = C^V_{I} \cap K_N^V$ and $K_{M'}^V = C^V_{I'} \cap K_{N'}^V$. Finally, we notice that $K_M^V \cap K_{M'}^V = K_N^V \cap K_{N'}^V$, which is a subset of the cone $C^V_{I \cap I'}$. This is true if and only if $\overline{Q(N)\cup Q(N')} = \overline{Q(M)\cup Q(M')}$.

We recall that $\overline{P(M) \cup P(M')}=P(M \cap M')$ for two classical nested sets. Using the same linear isomorphism as in the proof of statement \ref{item:prop-with-isomorphism} of Proposition \ref{prop:preposet-list-2}, we find an isomorphism mapping $K_M^V$ to $K_M$, and similarly for $M'$ and $M \cap M'$. Now, because $K_M \cap K_{M'} = K_{M \cap M'}$, we find that $K_M^V \cap K_{M'}^V = K_{M \cap M'}^V$. 

Finally, because $K_M^V \cap K_{M'}^V = K_{\overline{Q(M) \cup Q(M')}}$, and $\overline{Q(M)\cup Q(M')}=\overline{Q(N)\cup Q(N')}$, this implies $K_{\overline{Q(N)\cup Q(N')}}^V = K_{Q(M\cap M')}^V$, which means $\overline{Q(N)\cup Q(N')}=Q(M \cap M')$. Note now that $M \cap M'=N \cap N'$, proving that $\overline{Q(N) \cup Q(N')}=Q(N\cap N')$.
\end{proof}

\subsection{$\Delta$-braid cones arising from $\Delta$-nested sets}\label{sub:delta-braid-cones-from-nested}

\begin{proposition}
When $\Delta$ is a simplicial complex with vector set $V$ such that $\mathcal{F}(\Delta,V)$ is non-degenerate, and $N$ is a nested set for a building set $\B$, we find $K_N^V=\conichull\{\sum_{i \in I} v_i|I \in N\}$.
\end{proposition}

\begin{proof}
We use the same map $T(v_i)=[-e_i]$ from the proof of part \ref{item:prop-with-isomorphism} of Proposition \ref{prop:preposet-list-2} to establish an isomorphism mapping the cone $K_N^V$ to the cone $K_{N}$. Proposition \ref{prop:classical-nested-set-conic-hull} states that $K_N$ is equal to the conic hull of vectors $\{\sum_{i \in I}[-e_i]|I \in N\}$, and so by the inverse of $T$, $K_N^V$ is the conic hull of vectors of the form $\sum_{i \in I} v_i$ for $I \in N$.
\end{proof}

\begin{proposition}\label{prop:rank-preserving}
For a $\B$-nested set $N$, the number of equivalence classes of $Q(N)_{finite}$ is equal to $|N|$, and for a vector set $V$ such that $\mathcal{F}(\Delta,V)$ is non-degenerate, the cone $K_N^V$ is $|N|$-dimensional.
\end{proposition}

\begin{proof}
If $N$ contains $|N|$ sets, then $Q(N)_{finite}$ is a forest preposet with $|N|$ principal ideal sets, and $K_N^V$ is a simplicial cone with $|N|$ extremal rays, meaning it is $|N|$-dimensional. In addition, according to Proposition \ref{prop:preposet-list-2}, $Q(N)_{finite}$ has $|N|$ equivalence classes.
\end{proof}

\begin{definition} \label{def:Delta-V-B}
For a $\Delta$-building set $\B$ and cone complex $\mathcal{F}(\Delta,V)$, the collection of cones $K^V_N$ where $N$ is a $\B$-nested set is written as $\mathcal{F}(\Delta,V,\B)$.
\end{definition}

We call $\mathcal{F}(\Delta,V,\B)$ a \emph{nested set fan complex} of $\B$. When $\B$ is the graphical building set of a $\Delta$-graph $G$, we will write this as $\mathcal{F}(\Delta,V,G)$.

We will find as a corollary to Theorem \ref{thm:common-refinement} that when $\mathcal{F}(\Delta,V)$ is a non-degenerate fan, then $\mathcal{F}(\Delta,V,\B)$ is a fan which refines $\mathcal{F}(\Delta,V)$.


\begin{proposition}\label{prop:Delta-complex-fan}
When $\mathcal{F}(\Delta,V)$ is non-degenerate and $\B$ is a $\Delta$-building set, then the face poset of the fan $\mathcal{F}(\Delta,V,\B)$ is isomorphic to the nested complex of $\B$.
\end{proposition}

\begin{proof}
We have a map taking the $\B$-nested set $N$ to the cone $K_N^V$. We note that each cone $K_N^V$ is $|N|$-dimensional. Now if $N \ne N'$ for two $\B$-nested sets, then $K_{N}^V \cap K_{N'}^V = K_{N \cap N'}^V$, which is lower dimensional, and so $K_N^V \ne K_{N'}^V$. As a result, this map is injective, and must be a bijection. We also note that $N \subseteq N'$ if and only if $K_N^V \subseteq K_N^{V'}$, and so this is an isomorphism.
\end{proof}

\subsection{Simplex Case} \label{sub:simplex-case}

%

%


When a classical building set $B$ is connected, we can define a simplex-building set $\B$ by removing $B_{max}$ from $B$, and if $N$ is a $B$-nested set, then $N \backslash B_{max}$ is a $\B$-nested set. In this section we describe a similar isomorphism, this time between the preposets obtained from classical nested sets of connected building sets, and preposets of simplex-nested sets.


Consider the preposet obtained by taking a facial preposet $Q$, and removing the element $\infty$. Call this preposet $Q\backslash \infty$. Note we are removing only the element $\infty$, and not the equivalence class containing $\infty$.

\begin{proposition} \label{prop:preposet-correspondence}
If $N$ is a classical $B$-nested set on base set $\mathcal{S}$ and $B$ is connected, then write $N'=N \backslash B_{max}$, and $\B=B\backslash B_{max}$. 
We find that the preposet $P_N$ is equal to $Q(N')\backslash \infty$.
\end{proposition}

\begin{proof}

The preposet $P_N$ over base set $\mathcal{S}$ is defined in Definition \ref{def:preposet-from-nested-set-most-general} as the preposet on $\mathcal{S}$ such that $i \preceq j$ if and only if every set containing $j$ also contains $i$. We should note that because $B$ is connected, $B_{max}=\{\mathcal{S}\}$. We then see that $P_N = P_{N'}$ defined over $\mathcal{S}$.

The preposet $Q(N')$ is defined as the preposet on $\mathcal{S}\cup\{\infty\}$ such that $i \preceq j$ if and only if every set in $N'$ containing $j$ also contains $i$. This is equal to the preposet $P_{N'}$, but with an added element $\infty$. Removing the infinite element of $\mathcal{S}\cup\{\infty\}$ then yields $Q(N')\backslash \infty = P_N$.
\end{proof}

We see this illustrated in Figure \ref{fig:graph-poset-2}, where the classical nested set $N$ from Figure \ref{fig:graph-poset} has $B_{max}=\{\mathcal{S}\}$ removed, and $Q(N')$ is equivalent to $P_N$ with an added infinite element.

\begin{remark}
We note that for every classical $B$-nested set $N$, the infinite set of $Q(N')$ will contain at least one element in $\mathcal{S}$. As such, the reader may consider an alternate definition of facial preposets, equivalent to removing $\infty$. However, we can find a pair of $\B$-nested sets $N, N'$ for $\Delta$-building set $\B$ such that $N\ne N'$ but $Q(N) \backslash \infty = Q(N') \backslash \infty$. As a minimal example: if $\Delta$ is a simplicial complex with a single element $\{1\}$, and $\B=\{\{1\}\}$, then define $N=\{1\}$ and $N'=\emptyset$. We find $Q(N)$ is the preposet with relation $1 \prec \infty$ and $Q(N')$ is the preposet such that $1 \equiv \infty$, but $Q(N)\backslash\infty = Q(N')\backslash \infty$. As a result, we see that $Q(N)$ encodes information about $N$ that $Q(N)\backslash \infty$ does not.
\end{remark}

\begin{figure}
\centering
\includegraphics[width=.5\textwidth]{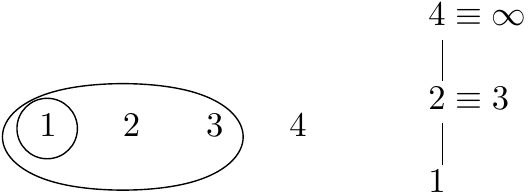}
\caption{A simplex-graph tubing and its associated preposet}\label{fig:graph-poset-2}
\end{figure}



\begin{proposition} \label{prop:minus-infinity}
If $\Delta$ is the simplicial complex consisting of proper subsets of $\mathcal{S}$ and $v_i=[-e_i]$ in $\R^\mathcal{S}/(1,\ldots,1)\R$ for each $i \in \mathcal{S}$, then for each facial preposet $Q$, the cone $K^V_Q$ is equal to the cone $K_{Q\backslash \infty}$.
\end{proposition}

\begin{proof}

We know that the base set of $Q_{finite}$ is a proper subset of $\mathcal{S}$. As a result, when $Q$ is a ray preposet, then $Q\backslash \infty$ is a ray preposet. In this case, the cone $K_{Q \backslash \infty}$ is equal to the ray of $\sum_{i \in Q_{finite}} -e_i$, which is exactly equal to $K^V_Q$.

Consider the case when $Q$ is not a ray preposet. We note that the base set of $Q_{finite}$ must be a face of $\Delta$, which in this case is the set of proper subsets of $\mathcal{S}$. As a result, there exists at least one element $i \in \mathcal{S} \backslash Q_{finite}$. This means that $Q\backslash \infty$ is connected, and so $K_{Q\backslash \infty}$ is equal to the conic hull of the rays of its ray contractions.

Note that $R$ is a facial preposet and a ray contraction of $Q$ if and only if $R\backslash \infty$ is a ray contraction of $Q \backslash \infty$. As a result, we see that $K_{Q\backslash \infty}$ is equal to the conic hull of the cones of ray contractions of $Q\backslash \infty$, which is equal to the conic hull defined as $K_Q^V$.
\end{proof}

If we write $N'=N\backslash B_{max}$ for any classical $B$-nested set $N$ and connected classical building set $B$, we then find that $P_N=Q(N')\backslash \infty$ implies $K_N=K_{N'}^V$. This implies the following.

\begin{proposition}
Say $\Delta$ is the simplicial complex consisting of proper subsets of $\mathcal{S}$ and $v_i=[-e_i]$ in $\R^\mathcal{S}/(1,\ldots,1)\R$ for each $i \in \mathcal{S}$. Say that $B$ is a connected classical building set on $\mathcal{S}$, and $\mathcal{B}=B\backslash B_{max}$. Then the nested fan of $B$, containing all cones of the form $K_N$ where $N$ is $B$-nested, is equal to the fan $\mathcal{F}(\Delta,V,\B)$.
\end{proposition}


\begin{remark}
We note that if $N$ is a classical $B$-nested set for connected classical building set $B$ on base set $\mathcal{S}$, and $v_i=[-e_i]$ for all $v_i \in V$, then we have just proven that we can find $\Delta$-nested set $N'$ such that $P_N = Q(N')\backslash \infty$ and $K_N = K_{N'}^V$. However, if $Q$ is a preposet on $\mathcal{S}$, then we cannot always find a facial preposet $Q'$ such that $K_Q=K_{Q'}^V$. We know this because braid cones are not always contained in the maximal cones of the braid arrangement; for instance, the preposet $Q_0$ such that $i, j$ are incomparable on $\mathcal{S}$ has cone $K_{Q_0}$ equal to the entire vector space $\R^{\mathcal{S}}/(1,\ldots,1)$. Meanwhile, for each facial preposet $Q'$, we know $K_{Q'}^V \subseteq C_I^V$ for some set $I$, which must be a proper subset of the entire vector space. As a result, $K_Q \ne K_{Q'}^V$ for any facial preposet $Q'$. As a result, we must keep in mind that $\Delta$-braid cones \emph{do not generalize braid cones}. Instead, they generalize a certain subset of braid cones, those whose preposets have a unique maximal equivalence class.
\end{remark}

%
%
%
%

\section{Fan intersection theorem} \label{sec:fan-intersection}


Recall the definition of stellar subdivision for simplicial complexes in Definition \ref{def:stellar-subdivision}. We can define stellar subdivision for simplicial fans.

\begin{definition}
Given a non-degenerate fan $\mathcal{F}(\Delta,V)$, a face $I \in \Delta$, and a vector $v_h$ in the relative interior of $C^V_I$, the \emph{stellar subdivision} of $\mathcal{F}(\Delta,V)$ is the fan $\mathcal{F}(\textrm{St}_I(\Delta),V \cup \{v_h\})$, where $\textrm{St}_I(\Delta)$ is the stellar subdivision of $\Delta$ replacing the set $I$ with the element $h$.
\end{definition}

The stellar subdivision of a simplicial fan corresponds to the geometric construction of deleting an $n$-dimensional simplicial cone, and adding in $n$ new cones which fill the hole left by the previous cone. As a result, the union of all cones in a fan, called the \emph{support} of the fan, is equal to the union of all cones in one of its stellar subdivisions. We note that for non-degenerate $\mathcal{F}(\Delta,V)$, Theorem \ref{thm:feichtnermainresult} implies that $\mathcal{F}(\Delta,V,\B)$ is the result of repeated stellar subdivision of $\mathcal{F}(\Delta,V)$.

Consider a simplicial complex $\Delta$ and vector set $V$ such that $\mathcal{F}(\Delta,V)$ is non-degenerate. If $I$ is a face of $\Delta$, define $\B_I$ to be the building set containing $I$ and all singleton subsets of the base set $\mathcal{S}$ of $\Delta$. This fan is a stellar subdivision of $\mathcal{F}(\Delta,V)$, replacing the cone $C^V_I$ with the ray equal to the conic hull of $\sum_{i \in I} v_i$.

\begin{definition}
The \emph{coarsest common refinement} of $n$ fans $\mathcal{F}_1, \ldots, \mathcal{F}_n$ is the set of cones $\{C_1 \cap \cdots \cap C_n | C_1 \in \mathcal{F}_1, \ldots, C_n \in \mathcal{F}_n\}.$
\end{definition}



\begin{lemma}\label{lemma:common-refinement-lemma}
In order to prove that a cone $C$ is in the coarsest common refinement of a finite set of fans $\mathcal{F}_1, \ldots, \mathcal{F}_k$, it is sufficient to prove two conditions. The first is that $C$ is equal to the intersection of some subset of cones from the fan set, but not necessarily one cone from every fan. The second is that $C$ is a subset of the intersection of some set of cones $C_1 \in \mathcal{F}_1, \ldots, C_k \in \mathcal{F}_k$.
\end{lemma}

\begin{proof}
Say that $C \subseteq \bigcap_{i=1}^k C_i$ for some set of cones $C_1 \in \mathcal{F}_1, \ldots, C_k \in \mathcal{F}_k$. Say also that $C = \bigcap_{D \in K} D$, where $K$ is a set of cones such that for each $D \in K$, $D \in \mathcal{F}_i$ for some $i$.

For each $i \in [k]$, define $C_i'$ as the intersection of cone $C_i$, and every cone $D \in K$ such that $D \in \mathcal{F}_i$. Because fans are closed under intersection, we find $C_i'\in \mathcal{F}_i$, and now find $C = \bigcap_{i =1}^k C_i'$, proving $C$ is in the coarsest common refinement of these fans.
\end{proof}

\begin{theorem} \label{thm:common-refinement}
For a $\Delta$-building set $\B$, the fan $\mathcal{F}(\Delta,V,\B)$ is equal to the coarsest common refinement of fans $\mathcal{F}(\Delta,V,\B_I)$ for each $I \in \B$.
\end{theorem}

\begin{proof}

The coarsest common refinement of the fans of the form $\mathcal{F}(\Delta,V,\B_I)$, and the fan $\mathcal{F}(\Delta,V,\B)$, both cover the same space. As a result, one fan cannot be a proper subset of the other, and $\mathcal{F}(\Delta,V,\B)$ and the coarsest common refinement are equal if and only if every cone in $\mathcal{F}(\Delta,V,\B)$ is in the coarsest common refinement.

We use Lemma \ref{lemma:common-refinement-lemma} to prove that every cone $K_N^V$ in $\mathcal{F}(\Delta,V,\B)$ is in the coarsest common refinement of the fans $\mathcal{F}(\Delta,V,\B_I)$. To do this, we first identify a set of cones in the union of these fans whose intersection is $K_N^V$, and then for each $I \in \B$ we must prove that $K_N^V$ is a subset of some cone in $\mathcal{F}(\Delta,V,\B_I)$.

One important note for this theorem: Propositions \ref{prop:preposet-list-2} and \ref{prop:facial-preposet-union-intersection} only apply when two nested sets share the same $\Delta$-building set. If two sets $N, N'$ are nested but under different building sets, then the propositions do not apply.

Consider a $\B$-nested set $N$. For each pair $i,I_i$ such that $I_i \in N$ and $I_i$ is the principal order ideal of $i$ in $Q(N)$, define the set $M_{i}$ such that $M_{i}$ contains the set $I_i$ and all singleton subsets of $Q(N)_{finite}$ except for $\{i\}$. This set is $\B_{I_i}$-nested. Note that the set $Q(M_{i})_{finite}$ contains only the relations $j \preceq i$ for $j \in I_i$. This means that for all $j$ in base set $\mathcal{S}$, $j \preceq_{Q(M_i)} i$ if and only if $j \preceq_{Q(N)} i$. We now note:

\[
	Q(N) = \bigcup_{i \in \bigcup N} Q(M_i).
\]

Using Proposition \ref{prop:preposet-list-2}, we note the following:

 $$K_{Q(N)}^V = \bigcap_{i \in \bigcup N} K_{Q(M_i)}.$$

Because each cone $K_{Q(M_i)}$ is contained in some fan $\mathcal{F}(\Delta,V,\B_{I_i})$, we have satisfied our first condition.

For our second condition, we must now prove that for every set $I \in \B$, there exists a $\B_I$-nested set $M_I$ such that $Q(M_I) \subseteq Q(N)$, and therefore $K_{N}^V \subseteq K_{M_I}^V$. For each $I \in \B$, we consider two cases.

If $I$ is not a subset of the base set of $\bigcup{N}$, then choose $M_I=\{\{x\}|x \in \bigcup{N}\}$. We note that $M_I$ only contains singleton sets and $I$, so $M_I$ must be $\B_I$-nested. Because the base set of $Q(N)_{finite}$ is $\bigcup{N}=Q(M_I)$, this means that $Q(M_I)$ has the same finite base set as $Q(N)$, but such that $Q(M_I)_{finite}$ contains only incomparable elements. As a result, $Q(M_I) \subseteq Q(N)$.

In the second case, consider the case where $I$ is a subset of the base set of $\bigcup{N}$. We first wish to prove that there exists a maximal element $i \in I$ under the partial relation $Q(N)$. If $I \in N$, then we know $I$ is the principal ideal of some element $i \in I$. Otherwise, assume there is no such element $i$, and that there exist elements $i_1,\ldots, i_k$ belonging to $k \ge 2$ equivalence classes, and generating principal ideals $I_1, \ldots, I_k$. Note that because $I$ and $I_j$ intersect for each $1 \le j \le k$, by building set rules, the set $I \cup I_1 \cup \cdots \cup I_k \in \B$. In addition, we should be able to see that these sets cover $I$ and are disjoint. As a result, the disjoint sets $I_1, \ldots, I_k \in N$ have a union $I_1 \cup \cdots \cup I_k \in \B$, which is a contradiction as it means $N$ is not nested. As a result, there exists an element $i \in I$ such that $j \preceq_{Q(N)} i$ for all $j \in I$. Because of this, we return to the case where $I \subseteq \bigcup N$. We pick a maximal element $i \in I$, and in this case define $M_I$ to be equal to the $\B_I$-nested set containing the sets $\{j\}$ for each element $j \in (\bigcup N )\backslash \{i\}$, and containing the set $I$. We note that $Q(M_I)$ in this case is the facial preposet with finite base set $\bigcup N$ and relations $j \preceq i$ for all $j \in I$. We note that $j \preceq i$ for all $j \in I$ in $Q(N)$, meaning that $Q(M_I) \subseteq Q(N)$.

As a result, for every set $I \in \B$, there exists a $\B_I$-nested set such that $Q(M_I) \subseteq Q(N)$, and from \ref{prop:preposet-list-2}, we note that $K_{Q(N)} \subseteq K_{Q(M_I)}$.

We have thus proven with our lemma that every cone $K_N^V$ is contained in the coarsest common refinement of these fans, and so $\mathcal{F}(\Delta,V,\B)$ is a subset of the coarsest common refinement. We should note that the two fans have the same support: recall that $\mathcal{F}(\Delta,V,\B)$ is the result of repeated stellar subdivision of $\mathcal{F}(\Delta,V)$, and so its support would be equal to the support of $\mathcal{F}(\Delta,V)$. Note as well that each fan $\mathcal{F}(\Delta,V,\B_I)$ has support equal to that of $\mathcal{F}(\Delta,V)$, and so their common refinement has support equal to that of $\mathcal{F}(\Delta,V)$. As a result, $\mathcal{F}(\Delta,V,\B)$ is a subset of the coarsest common refinement of all fans $\mathcal{F}(\Delta,V,\B_I)$, but because the two fans have equal support, they must be equal.
\end{proof}

\begin{corollary} \label{cor:union}
For any set of $\Delta$-building sets $\B_1 \cup \cdots \cup \B_m=\B$ and vector set $V$ such that $\mathcal{F}(\Delta,V)$ is non-degenerate, the fan $\mathcal{F}(\Delta,V,\B)$ is equal to the coarsest common refinement of the fans $\mathcal{F}(\Delta,V,\B_1), \cdots, \mathcal{F}(\Delta,V,\B_m)$.
\end{corollary}

\section{Barycentric subdivision} \label{sec:barycentric-sub}

A \emph{chain} of a simplicial complex $\Delta$ is an ordered list $S=(I_1,\ldots,I_k)$ of faces of $\Delta$, such that $I_1 \subsetneq \cdots \subsetneq I_k$.

For every simplicial complex $\Delta$, the \emph{maximal building set} is the building set containing every face in $\Delta$, which is $\Delta$ itself. The nested sets of this building set are the sets of the form $N=\{S_1,\ldots,S_k\}$, where $S_1 \subseteq \cdots \subseteq S_k$ and $S_k \in \Delta$. As a result, we find the following.

\begin{proposition}
The nested sets of the maximal building set of a simplicial complex $\Delta$ are exactly the chains of $\Delta$.
\end{proposition}

As a result, we can define the facial preposet $Q(S)$ of a chain $S$ of a simplicial complex $\Delta$, and we will write $C_S^V$ to denote the cone $K^V_{Q(S)}$. Note that $\mathcal{F}(\Delta,V,\B)$ denotes the nested complex fan of a $\Delta$-building set $\B$. For clarity, we point out that if $\B$ is the maximal building set of $\Delta$, then the resulting fan is written $\mathcal{F}(\Delta,V,\Delta)$.




\begin{definition}
The \emph{barycentric subdivision} of a non-degenerate simplicial fan $\mathcal{F}(\Delta,V)$ with vector set $V$ is the nested complex fan of the maximal building set of $\Delta$, or $\mathcal{F}(\Delta,V,\Delta)$.
\end{definition}

This fan is the set of all cones $C^V_S$ where $S$ is a chain of $\Delta$. We should note that barycentric subdivisions of simplicial complexes are typically defined by finding the barycenter of each simplex face and are therefore unique. The barycenter of a simplicial cone, however, does not have a standard definition, and our definition of barycentric subdivision of a simplicial fan depends on the choice of vectors in $V$, and rescaling of vectors in $V$ will result in different barycentric subdivisions.

%
%
%

\begin{proposition}\label{prop:barycentric}
For every $\Delta$-building set $\B$ of a simplicial complex where $\mathcal{F}(\Delta,V)$ is non-degenerate, the fan $\mathcal{F}(\Delta,V,\B)$ is a coarsening of the barycentric subdivision of $\mathcal{F}(\Delta,V)$.
\end{proposition}

\begin{proof}
Define $\B_I$ to be the building set consisting of a set $I \in \Delta$ and all singleton faces of $\Delta$. We know $\B=\bigcup_{I \in \B} \B_I$ and $\Delta=\bigcup_{I\in \Delta} \B_I$. From Corollary \ref{cor:union}, the fan $\mathcal{F}(\Delta,V,\B)$ is the coarsest common refinement of fans $\mathcal{F}(\Delta,V,\B_I)$ for all $I \in \B$. Now we see the fan $\mathcal{F}(\Delta,V,\Delta)$ is the coarsest common refinement of fans $\mathcal{F}(\Delta,V,\B_I)$ for all $I \in \Delta$, so $\mathcal{F}(\Delta,V,\Delta)$ is a refinement of $\mathcal{F}(\Delta,V,\B)$.
\end{proof}

As a result, every $\Delta$-braid cone is equal to a convex union of cones of the form $C_S^V$. This is similar to the braid arrangement, where every maximal-dimension braid cone is equal to a union of Weyl chambers. Note however that while every convex union of Weyl chambers is a braid cone, not every convex union of maximal-dimension cones $C_S^V$ is a $\Delta$-braid cone; as an example, the entire vector space containing $\mathcal{F}(\Delta,V)$ if it is a complete fan. This is one way in which our $\Delta$-braid cone construction is not a perfect generalization of braid cones.

We define the \emph{maximal graph} of a simplicial complex $\Delta$ to be the $\Delta$-graph containing the edge $\{i,j\}$ whenever $\{i,j\}$ is a face of $\Delta$. 

\begin{proposition}
The graphical building set of the maximal graph of a simplicial complex is equal to the maximal building set of that simplicial complex.
\end{proposition}

The maximal graph of a simple polyhedron is the maximal graph of its dual simplicial complex, and is equal to the facet adjacency graph of that polyhedron.

%
%

\section{$\p$-nestohedron case} \label{sub:p-nestohedron-fans}

The fan $\mathcal{F}(\Delta,V)$ is a geometric representation of a simplicial complex. This general framework can be used for the specific case where $\mathcal{F}(\Delta,V)$ is the normal fan of a simple polyhedron, and can be used for $\p$-nestohedra. 

Given a polyhedron $P$ with facet set $S$, define a set $V$ of normal vectors $\{v_i|i \in S\}$. The normal fan of $P$ is equal to the fan $\mathcal{F}(\Delta(P),V)$, with cones $C^V_I=\conichull\{v_i|i \in I\}$ for every face $I$ of the dual simplicial complex $\Delta(P)$. 

\subsection{Nestohedron fan}

\danger{Normal fans are actually mentioned elsewhere in the paper, but this is the first explicit definition. I think the other mentions are enough in passing that this is fine.}

For a polyhedron $P$ in a vector space $V$ and vector $v \in V$, define $F_v(P)$ as the face of $P$ which maximizes the function $x \cdot v$ for all $x \in P$, defined for all vectors $v$ such that a maximum of the function $x \cdot v$ exists. \danger{This feels odd, but necessary; we want to have normal cones be dual to nonempty faces in our definition, and so we can't have the entire-vector-space-V here.} Define the \emph{normal cone} $NC(F,P)$ of a face $F$ of $P$ to be the set of vectors $v \in V$ such that $F \subseteq F_v(P)$. We note that for each $v$ where $F_v(P)$ is defined, $v$ is in the interior of $NC(F_v(P),P)$, and we can say $v$ is dual to $F_v(P)$.

\begin{definition}
The \emph{normal fan} of a polyhedron $P$ is the collection of cones $NC(F,P)$ for all nonempty faces $F$ of $P$.
\end{definition}

We note that if $F_1 \subseteq F_2$, then $NC(F_2,P) \subseteq NC(F_1,P)$, and we find that if $P$ is simple, then the poset of faces of the normal fan of $P$ is isomorphic to the dual simplicial complex of $P$. The rays of the normal fan of $P$ are dual to the facets of $P$, and if a vector $v$ is dual to a facet $F$, then that vector is a \emph{normal vector} of $F$.

\begin{proposition}
The fan $\mathcal{F}(P,V,\B)$ is the normal fan of a nestohedron $\nest{P}{\B}$, obtained by listing all sets $\{I_1,\ldots, I_k\} \in \B$ in descending order of size, and truncating faces $F_I$ of $P$ with hyperplanes normal to the vectors $v_I=\sum_{i \in I} v_i$.
\end{proposition}

\begin{proof}
Proposition \ref{prop:Delta-complex-fan} proves that the fan $\mathcal{F}(P,V,\B)$ is isomorphic to the the nested complex of $\B$, and as a result, should be combinatorially isomorphic to the normal fan of $\nest{P}{\B}$. Secondly, the rays of the normal fan of a nestohedron $\nest{P}{\B}$ truncated by the method specified will have rays which are the rays of vectors $v_I=\sum_{i \in I} v_i$ for all $I \in \B$. Because the two fans are combinatorially isomorphic by a map taking all rays to themselves, the two fans are equal.
\end{proof}

Note that this fan $\mathcal{F}(P,V,\B)$ is a coarsening of the barycentric subdivision of the normal fan $\mathcal{F}(P,V)$. For some polyhedra and choices of $V$, the barycentric subdivision of the normal fan has combinatorial significance. For instance, when $P$ is a simplex or a hypercube and $V$ is the set of normal unit vectors, the barycentric subdivision is the type $A$ or type $B$ braid arrangement, respectively.

\subsection{Minkowski sums}

The \emph{Minkowski sum} of two polyhedra $P, P'$ is the set of points $$P+P'=\{x+x'|x\in P,x' \in P'\}.$$

For any face $F$ of $P$, $F'$ of $P'$, the set $F+F'$ is a subset of $P+P'$, but is not necessarily a face of $P+P'$. The combinatorics of the resultant Minkowski sum depend on the normal fans of the summands. The following proposition is well-known, but bears repeating.


\begin{proposition}
Given two polyhedra $P, P'$ with normal fans $\mathcal{F}_1, \mathcal{F}_2$, the normal fan of $P+P'$ is equal to the coarsest common refinement of $\mathcal{F}_1$ and $ \mathcal{F}_2$.
\end{proposition}

\begin{proof}
	
For any vector $v$, the faces $F_v(P), F_v(P')$ exist if and only if the sets of vectors which maximize the function $x \cdot v$ on $P$ and $P'$ are $F_v(P)$ and $F_v(P')$, respectively. It is then trivial to find that if $x''\in P+P'$ is equal to the sum of a pair of points $x \in P, x' \in P'$, then $x'' \in F_v(P+P')$ if and only if it can be written as a sum $x+x'$ where $x \in F_v(P), x' \in F_v(P')$. As a result, $F_v(P+P')=F_v(P)+F_v(P')$.

This shows us that every face of $P+P'$ is a Minkowski sum of faces of $P$ and $P'$. We now consider the cone $NC(F_v(P+P'),P+P')$, and a vector $v'$. We say that $v' \in NC(F_v(P+P'),P+P')$ if and only if the function $x'' \cdot v'$ is maximized on the face $F_v(P+P')$ of $P+P'$, which we know to be the case if and only if the functions $x \cdot v'$ and $x' \cdot v'$ are maximized on the faces $F_v(P)$ of $P$ and $F_v(P')$ of $P'$, respectively. As a result, $v'$ is contained in $NC(F_v(P+P'),P+P')$ if and only if $v' \in NC(F_v(P),P)$ and $v' \in NC(F_v(P'),P')$.
\end{proof}

For a simple polyhedron $P$ with $P$-building set $\B$ and normal vector set $V$, define $P_I$ to be a truncation of $P$ at the face $F_I$ by a hyperplane normal to the vector $v_I=\sum_{i \in I} v_i$. This polyhedron is a nestohedron of the building set $\B_I$ containing only singleton sets and the set $I$, and the normal fan of $P_I$ is equal to the fan $\mathcal{F}(P,V,\B_I)$. The following proposition is a direct corollary of Theorem \ref{thm:common-refinement}.

\begin{proposition}
Given a simple polyhedron $P$ and $P$-building set $\B$, with normal vector set $V$ indexed by facets of $P$ defining truncations $P_I$ for each $I \in \B$, the Minkowski sum of all polyhedra $P_I$ for $I \in \B$ is a $P$-nestohedron of $\B$.
\end{proposition}

\danger{New proposition!}

Consider a $P$-graph $G$. We define a connected component of $G$ \emph{as a $P$-graph} to be the $P$-graph containing all edges in that connected component. \danger{awkward language ahead of me!}

\begin{proposition} \label{prop:connected-components-P-graph}
	For a $P$-graph $G$ with connected components, we can realize the $P$-graph associahedron as the Minkowski sum of the $P$-graph associahedra of each connected component of $G$.
\end{proposition}

\begin{proof}
If $G$ has connected components $C_1, \ldots, C_k$, then the set of tubes of $G$ is equal to the union of all possible tubes of its connected components. As a result, the graphical building set of $G$ is equal to the union of all the graphical building sets of $C_1, \ldots, C_k$, and so we can realize a $P$-graph associahedron of $G$ as the Minkowski sum of $P$-graph associahedra for each graph $C_1, \ldots, C_k$ as a $P$-graph.
\end{proof}

One would hope for a generalization of the results in \cite{postnikov}, where nestohedra are constructed as Minkowski sums of lower-dimensional simplices. When $P$ is a simplex, we note that each $P_I$ is isomorphic to the sum of $P$ and a lower-dimensional simplex. When $P$ is a more complicated polyhedron, such as a hypercube, the summands are more complex, as in Figure \ref{fig:minkowski-truncation}.

\begin{figure}
\centering
\includegraphics[width=.5\textwidth]{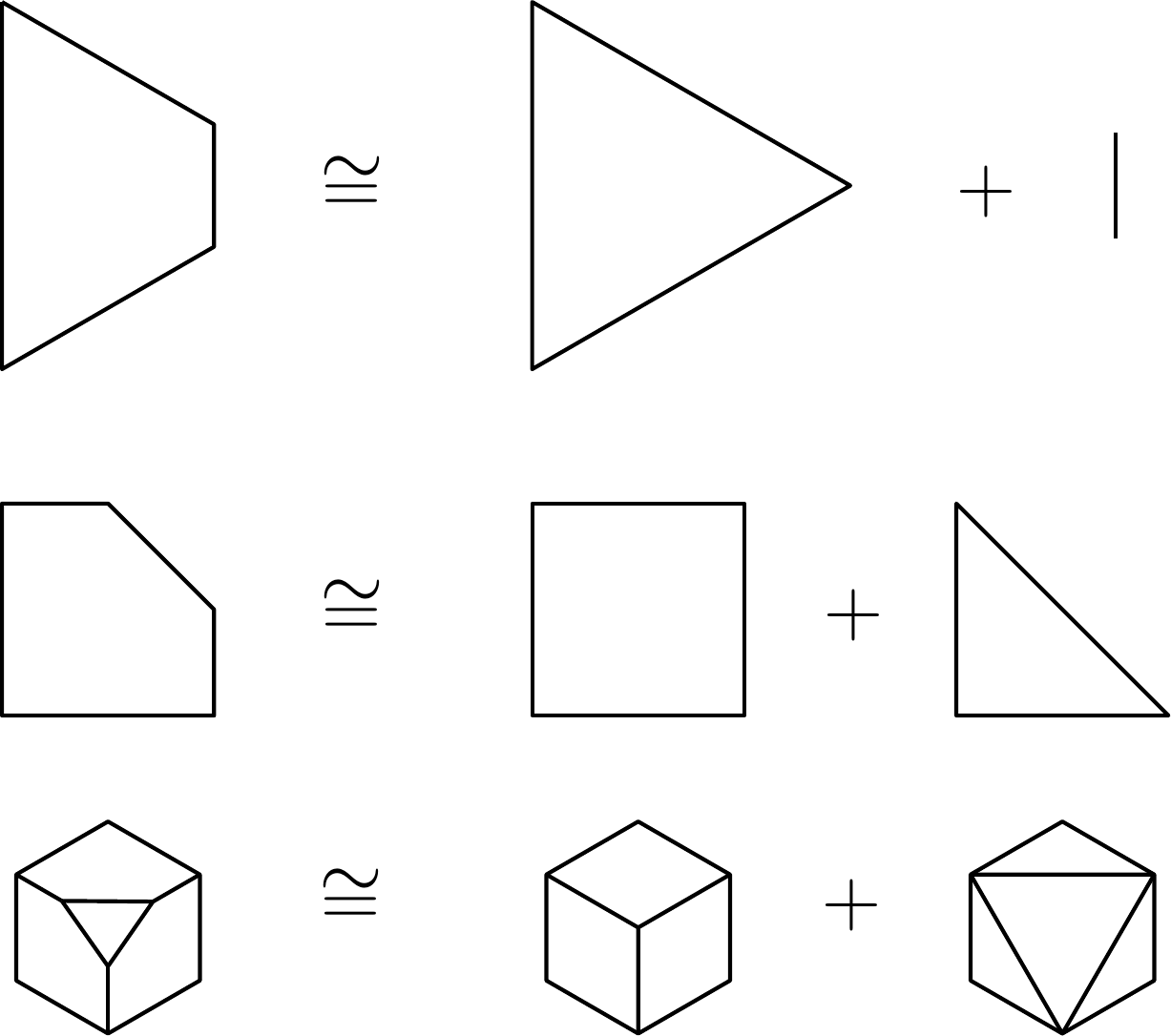}
\caption{Truncations rewritten as Minkowski summands for various polyhedra.}\label{fig:minkowski-truncation}
\end{figure}

\chapter{Combinatorics of Hypercube Graph Associahedra}\label{chap:hypercubes}

This chapter aims to summarize all major properties of hypercube graph associahedra.

\section{Hypercube Graphs}

Define the conventions $\pm[n]=\{-n,\ldots,-1,1,\ldots,n\}$, and in the vector space $\R^n$, for a vector $(x_1,\ldots,x_n) \in \R^n$ define $x_{-i}=-x_i$ for all $i \in [n]$. We define a standard $n$-dimensional hypercube to be the set of points defined by inequalities $\{\mathbf{x} \in \R^n| x_i \le 1 \, \forall i \in \pm[n]\}$, with the facets indexed by elements of $\pm[n]$ with $F_i$ defined by the inequality $x_i \le 1$. 

The $n$-dimensional hypercube is the Cartesian product of $n$ $1$-simplices. In Subsection \ref{sub:forbiddensubsetdiagrams}, we saw that the forbidden subset diagram of a $1$-simplex consists of a dashed edge on two vertices, so the forbidden subsets of an $n$-dimensional hypercube are $\{\{1,-1\},\ldots,\{n,-n\}\}$. The forbidden subset diagram contains $n$ pairs of vertices connected by dashed edges.

The tubes and tubings of hypercube graphs are easily characterized. If we consider a hypercube graph $G$ as a graph consisting of solid and dashed edges, then we find that $t$ is a tube if and only if the subgraph induced by $t$ is connected and contains no dashed edges. If $T$ is a collection of tubes, then $T$ is a tubing if and only if it satisfies the typical conditions of pairwise compatibility and if the set $\bigcup T=\bigcup_{t\in T} t$ induces a graph with no dashed edges. In this case, two tubes $t_1, t_2$ are weakly compatible if one of the following is true: $t_1 \subset t_2, t_2 \subset t_1$, or $t_1 \cap t_2 = \emptyset$ and $t_1, t_2$ are not adjacent (ie, there exists no solid edge between vertices in $t_1$ and $t_2$). As a result, a collection of hypercube graph tubes is a tubing if and only if all pairs of tubes are weakly compatible, and there exist no dashed edges between any tubes.

\begin{definition}
A simplicial complex $\Delta$ is a \emph{flag complex} if and only if, for every set $S$ such that $\{i,j\}$ is a face of $\Delta$ for all $i, j \in S$, the set $S$ is a face of $\Delta$.
\end{definition}

We note that a set of hypercube graph tubes is a tubing if and only if it contains no pairs of weakly incompatible tubes, and it contains no forbidden subsets of the hypercube. However, we note that a set of hypercube tubes contains a forbidden subset if and only if it contains a pair of tubes $t_1, t_2$ such that $t_1 \cap -t_2 \ne \emptyset$. Because of this, we find the following:

\begin{proposition} \label{prop:hypercube-pairwise-compatibility}
The tubing complex of a hypercube graph is a flag complex.
\end{proposition}

As an aside, we note that when $B$ is a connected classical building set, a subset $N$ of $B$ is nested if and only if every pair of sets $S_1, S_2 \in N$ is compatible; as a result, classical graph tubings are typically characterized only by pairwise strongly compatibility for connected graphs, as these are the flag tubing complexes.

\begin{definition}
We say that a bijection $\phi$ between flag tubing complexes \emph{preserves compatibility conditions} if and only if, for every pair of tubes $t, t'$ in the domain of $\phi$, we find $t, t'$ are strongly compatible if and only if $\phi(t), \phi(t')$ are strongly compatible.
\end{definition}

As a result, we see that two flag tubing complexes are isomorphic if and only if there is a map between them which preserves compatibility conditions.

\begin{example}

We introduce a running example of a three-dimensional hypercube graph in Figure \ref{fig:ygraph}. This hypercube graph consists of three edges $\{1,2\},\{2,3\}, \{2,-3\}$, and three forbidden subsets $\{1,-1\}, \{2,-2\}, \{3,-3\}$. A figure on the right illustrates the corresponding facets in the cube.


\begin{figure} [h!]
\centering
\includegraphics[width=.75\textwidth]{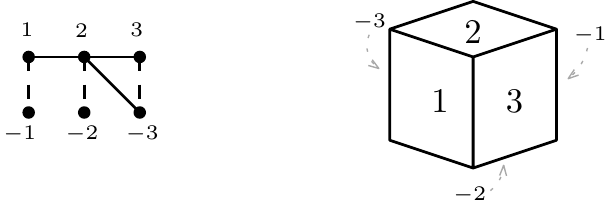}
\caption{Example hypercube graph and corresponding $(n=3)$ hypercube.} \label{fig:ygraph}
\end{figure}

Tubes of this graph are collections of vertices which induce connected subgraphs and avoid containing any vertex pairs of the form $\{i,-i\}$. Figure \ref{fig:ygraph2} illustrates all possible tubes of this graph. Note that tubes can be singleton vertices, such as $\{-1\}$ or $\{-2\}$, and they can mix positive and negative vertices, such as $\{1,2,-3\}$. However, they cannot induce a graph with any dashed edges, such as the set $\{2,3,-3\}$.

\begin{figure} [h!]
\centering
\includegraphics[width=.8\textwidth]{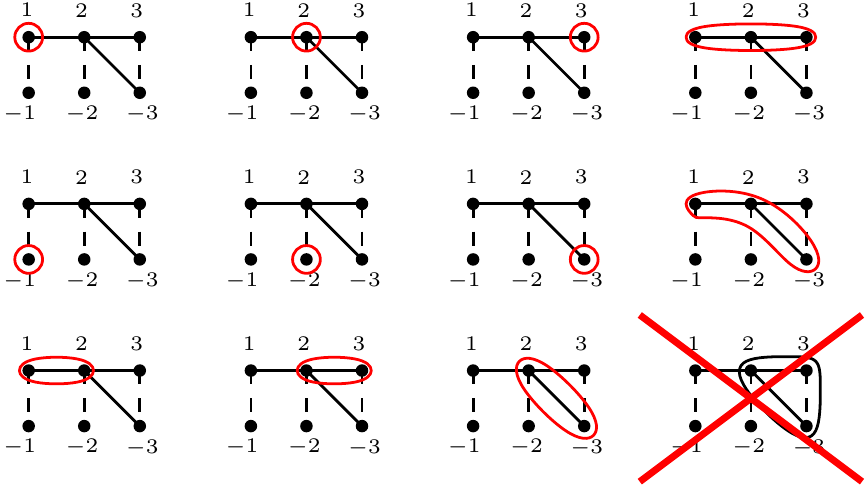}
\caption{All tubes of a hypercube graph, and one indicated non-tube.} \label{fig:ygraph2}
\end{figure}

 Figure \ref{fig:ygraph3} shows six valid hypercube tubings on the left side, and three pairs of hypercube graph tubes which fail the tubing criteria on the right hand side. On the right side, the first example shows a pair of tubes adjacent by a solid edge, the second example shows a pair of tubes adjacent by a dashed edge, and the third example shows a pair of non-trivially intersecting hypercube-graph tubes.

\begin{figure}[h!]
\centering
\includegraphics[width=.65\textwidth]{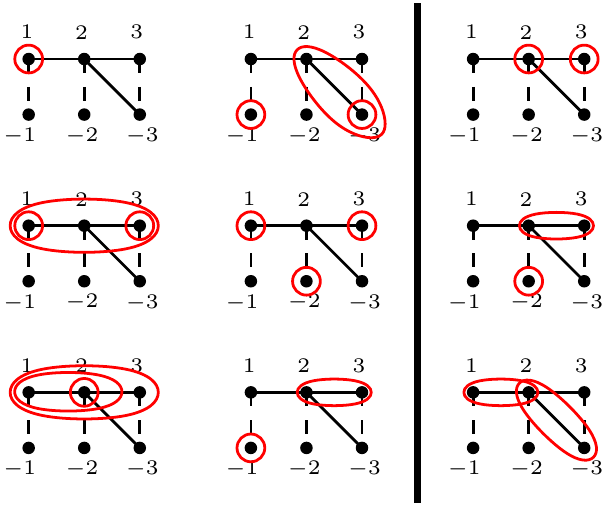}
\caption{Examples of valid and invalid tubings.}\label{fig:ygraph3}
\end{figure}

\end{example}

%
%
%
%
%
%
%
%
%
%

\section{Reconnected Complements}

Definition \ref{def:reconnected-complement-graph} defines the reconnected complement $G/t$ of a $\Delta$-graph $G$ with tube $t$. The vertex $x$ is in $G/t$ if and only if $x \notin t$ but $\{x\} \cup t \in \Delta$, and the edge $\{x,y\}$ is in $G/t$ if and only if either $\{x,y\}$ is an edge in $G$, or $x,y$ are both adjacent to $t$ and $\{x,y\}$ is a face of $\Delta/t$. The following proposition comes immediately from the definition of $G/t$.

\begin{proposition}
When $G$ is a hypercube-graph, $G/t$ is defined by removing both $t$ and $-t$, and adding edges to $G \backslash (\pm t)$ between former neighbors of $t$ wherever a dashed edge does not already exist.
\end{proposition}

Figure \ref{fig:reconnected-complement-hypercube} illustrates this process.

\begin{figure}[h!]
\centering
\includegraphics[width=.75\textwidth]{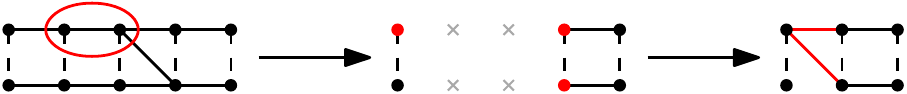}
\caption{An example of the reconnected complement process.} \label{fig:reconnected-complement-hypercube}
\end{figure}

\section{Standard cut hypercube graph associahdedra}

The standard basis vectors for $\R^n$ are $e_i$ for $i \in [n]$. We can also define vectors $e_{-i}=-e_i$ for each $i \in [n]$, and the vectors $e_i$ for $i \in \pm[n]$ are normal to the facets of the standard hypercube. Recall the definition of $\mathcal{F}(\Delta,V,\B)$ and $\mathcal{F}(P,V,G)$ in Definition \ref{def:Delta-V-B}. We use this definition to characterize hypercube-graph associahedra whose normal fans are well-behaved.

\begin{definition}\label{def:standard-cut}
Define a \emph{standard cut hypercube graph associahedron} of a hypercube graph $G$ to be any hypercube graph associahedron whose normal fan is equal to the fan $\mathcal{F}(Q_n,V,G)$, where $V$ is the set of facet normal vectors $v_i = e_i$ for $i \in \pm[n]$ and $Q_n$ is the standard $n$-dimensional hypercube.
\end{definition}

We note that the normal fan of the standard cut hypercube-graph associahedron of the maximal hypercube-graph is the barycentric subdivision of the normal fan of the hypercube. We note that for dimension $n$, each maximal cone of this fan is a cone of the form $\conichull\{e_{\sigma(1)}, e_{\sigma(1)}+e_{\sigma(2)}, \ldots, e_{\sigma(1)}+\cdots + e_{\sigma(n)}\}$, where $\sigma$ is a signed permutation on $[n]$. These are the maximal cones of the type $B_n$ Coxeter fan, which is defined by hyperplanes of the form $x_i = x_j, x_i=-x_j$, and $x_i=0$ for any $i,j \in [n]$. As a result, we see clearly that the standard cut hypercube-graph associahedron of the maximal hypercube-graph is normal to the type $B_n$ Coxeter fan. The following then comes from Proposition \ref{prop:barycentric}:

\begin{proposition}
Every standard-cut hypercube-graph associahedron is normal to a fan that coarsens the type $B_n$ Coxeter fan.
\end{proposition}

Such a polytope is called a \emph{deformation of the type $B_n$ Coxeter permutahedron} \cite{coxeterpermutahedra}. The relationship between hypercube-graph associahedra and type $B_n$ combinatorics is a rich area for future research, as discussed in Chapter \ref{chap:res3}.

\chapter{Enumeration Methods}\label{chap:methods}

\newcommand{\Acom}{\mathbb{A}}
\newcommand{\Bcom}{\mathbb{B}}
\newcommand{\Ag}{\mathbb{A}}
\newcommand{\Bg}{\mathbb{B}}
\newcommand{\Cg}{\mathbb{C}}

In this chapter, we aim to establish general methods for enumerating faces of $\p$-associahedra and tubings of $\Delta$-graphs. We define two separate methods for doing this. When applied to a family of $\p$-graph associahedra, the method outlined in Section \ref{sec:method-one} will calculate the $f$-polynomials of $\p$-graph associahedra by summing the $f$-polynomials of each facet. These facet $f$-polynomials are found by finding the induced subgraph and reconnected complement of every tube of every $\p$-graph in a family. This method is best applied when this operation of taking reconnected complements is, in some sense, closed; by analogy, the reconnected complement of any tube in a cycle graph is always a cycle, and every tube is a path, whereas the reconnected complement of a tree graph is often not a tree. In the $\p$-graphs studied in this thesis, this method results in a partial differential equation in two variables, and solving this differential equation results in an enumeration of the tubing complexes of these graphs.

The method outlined in Section \ref{sec:method-two} is less straightforward. When applied to a $\p$-graph $G$, it specifies a subset of vertices $X$ of a $\p$-graph, and sorts the tubings of $G$ by cases, according to which tube intersecting with the vertex set $X$ is maximal. This case-by-case method is powerful and is often capable of calculating $f$-polynomials without any differential equations involved. However, it can also require the calculation of $f$-polynomials of several auxiliary $\Delta$-graph associahedra.

In both of these methods, we will establish recursive systems of dependencies. For example, every facet of an associahedron is a product of associahedra, and the first method can be used to define the $f$-polynomials of associahedra. Using the same method, every facet of a cyclohedron is a product of an associahedron and a cyclohedron, and we can define a differential equation relating the $f$-polynomials of cyclohedra and associahedra. As a result, we can calculate $f$-polynomials of $\p$-graph associahedra by relating them to graphs whose $f$-polynomials are known or easier to calculate.

\section{Bivariate $f$-polynomials of families of simplicial complexes}

Define a \emph{power series} to be a symbolic sum $a(s)=\sum_{k\ge 0} a_k s^k$. We say that $a(s)$ expressed as a function for a real number $s$ is a \emph{generating function} of the series $\{a_0,a_1,\ldots\}$. Power series can be defined for multiple variables, such as $a(s,t)=\sum_{k\ge 0} a_{k,n} s^k t^n$. If $a_n(s)=\sum_{k \ge 0} a_{k,n}s^k$, then we can also write $a(s,t)$ as $\sum_{n \ge 0} a_n(s) t^n$.

For a simplicial complex $\Delta$, define the \emph{f-vector} of $\Delta$ to be the vector $(f^\Delta_0,f^\Delta_1, \ldots, f^\Delta_n)$, where $f^\Delta_k$ is the number of faces of $\Delta$ containing $k$ elements. The \emph{f-polynomial} of $\Delta$ is the polynomial $f^\Delta(s) = \sum_{k\ge 0} f^\Delta_k s^k$.

Consider a family of simplicial complexes $\Delta_0, \Delta_1, \ldots$ such that each complex $\Delta_n$ is rank $n$, where the rank of a simplicial complex is the largest sized face. We define the \emph{bivariate f-polynomial} of this family to be the bi-variate power series $$f^\Delta(s,t) = \sum_{n,k \ge 0} f^{\Delta_n}_k s^k t^n.$$

For a polyhedron $P$, define the \emph{f-vector} of $P$ to be the vector $(f^P_0,\ldots, f^P_n)$, where $f^P_k$ is the number of $k$-dimensional faces of $P$. The \emph{f-polynomial} of $P$ is the polynomial $f^P(x)=\sum_{k\ge 0} f_k^P x^k$.

When $\Delta$ is the dual simplicial complex of an $n$-dimensional simple polyhedron $P$, the f-vectors of $\Delta$ and $P$ are mirror images of each other, with $f^\Delta_k=f^P_{n-k}$ for each $0 \le k \le n$. We then find that $f^P(x)= x^n f^\Delta(x^{-1})$. When $P$ is a simplicial complex with lineality space of dimension $j$, then the dual simplicial complex of $P$ is isomorphic to the dual simplicial complex of the $(n-j)$-dimensional pointed polyhedron, $P / lineal(P)$. We find that $f^P(x) = x^j f^{P/lineal(P)}(x)$, and the first $j$ entries of the $f$-vector of $P$ are zeros. For this reason, we will often find it simpler to work with pointed polyhedra, so that $\Delta(P)$ always has rank equal to the dimension of $P$. \danger{pointed polyhedra mention}

\begin{remark}
Our definition of $f$-vector and $f$-polynomial of a polyhedron $P$ is constructed so as to be dual to the $f$-vector and $f$-polynomial of the dual simplicial complex of $P$. As a result, these $f$-vectors do not have an entry for the $-1$-dimensional face $\emptyset$. This definition varies from definitions for $f$-vectors and $f$-polynomials used in some other publications.
\end{remark}

Consider a family of pointed polyhedra $\{P_0, P_1, \ldots\}$, such that each polyhedron $P_n$ is $n$-dimensional. We define the \emph{bivariate f-polynomial} of this family to be the bi-variate polynomial $$f^P(x,y) = \sum_{n,k \ge 0} f^{P_n}_k x^k y^n.$$

\begin{proposition} \label{prop:change-of-basis}

For a family of simple pointed polyhedra $\{P_0,P_1,\ldots\}$ with dual simplicial complexes $\{\Delta_0,\Delta_1,\ldots\}$, the change of variables $x=1/s$ and $y=st$ gives the equality $$f^P(x,y)=f^\Delta(s,t).$$

\end{proposition}

\begin{proof}
If each polyhedron $P_n$ is dual to a simplicial complex $\Delta_n$, then we note $f_k^{\Delta_n} = f_{n-k}^{P_n}$. Also see that the change of variables reversed gives $s=1/x$ and $t=xy$. This means that $f_k^{\Delta_n} s^k t^n = f_{n-k}^{P_n} (1/x)^k (xy)^n = f_{n-k}^{P_n} x^{n-k} y^n$. Taking a sum over all $n,k$ proves the proposition.
\end{proof}

This allows for equations involving bivariate generating functions of families of simplicial complexes to be applied to families of simple polyhedra.

\begin{proposition}
	For two simplicial complexes $\Delta_1, \Delta_2$, $f^{\Delta_1 \times \Delta_2}(s)=f^{\Delta_1}(s) f^{\Delta_2}(s)$.
\end{proposition}

\begin{proof}
	We note that $\Delta_1 \times \Delta_2$ consists of the union of all faces of $\Delta_1$ and all faces of $\Delta_2$. As a result, every face containing $k$ elements in $\Delta_1 \times \Delta_2$ is the union of a face containing $i$ elements in $\Delta_1$ and $k-i$ elements in $\Delta_2$. This means $f^{\Delta_1 \times \Delta_2}_k = \sum_{i=0}^k f^{\Delta_1}_i f^{\Delta_2}_{k-i}$. As a result of this convolution, $f^{\Delta_1 \times \Delta_2}(s)=f^{\Delta_1}(s) f^{\Delta_2}(s)$.
\end{proof}

\begin{proposition} \label{prop:product-of-polyhedra-f-polynomials}
If $P_1$ and $P_2$ are simple polyhedra, then $f^{P_1 \times P_2}(x)=f^{P_1}(x) f^{P_2}(x)$.
\end{proposition}

\begin{proof}
If $\Delta_1, \Delta_2$ are the dual simplicial complexes of $P_1, P_2$ respectively, with $P_1$ $n$-dimensional and $P_2$ $m$-dimensional, then $$f^{P_1 \times P_2}(x)=x^{n+m}f^{\Delta_1 \times \Delta_2}(1/x)=x^n f^{\Delta_1}(1/x) x^m f^{\Delta_2}(1/x) = f^{P_1}(x) f^{P_2}(x).\qedhere$$
\end{proof}

For a family of $\p$-graphs $G=\{G_0, G_1, \ldots\}$, in this chapter we will write the bivariate $f$-polynomial of the family of tubing complexes of $G$ as $f^{\Delta(G)}(s,t)$. Meanwhile, we will write the bivariate $f$-polynomial of the family of pointed $\p$-nestohedra of $G$ as $f^{P(G)}(x,y)$.

\section{Atomic link sum polynomial method for calculating $f$-polynomials}\label{sec:method-one}

In introducing this section, we first provide a simplified version of the algorithm defined in this section. In a simple $n$-dimensional polyhedron, every vertex is contained in $n$ facets. As a result, we can count the number of vertices in an $n$-dimensional polyhedron by counting the number of vertices in each facet, adding the vertex counts for each facet together, and then dividing by $n$. If we are counting the number of vertices in a $\p$-nestohedron using this method, it helps to recall that according to Theorem \ref{thm:reconnectednestohedra1}, every facet of a $\p$-nestohedron is isomorphic to the product of two lower-dimensional $\p$-nestohedra. Also, according to Proposition \ref{prop:product-of-polyhedra-f-polynomials}, the number of vertices in the product of two polyhedra is equal to the product of the number of vertices in each polyhedron. As a result, we can calculate the number of vertices in each facet $\Phi_I$ of a $P$-nestohedron $\nest{P}{\B}$ by rewriting it as the product of two polyhedra, the simplex-nestohedron of $\B$ restricted to $I$, and the $F_I$-nestohedron of the building set $\B/I$. After finding the number of vertices of each nestohedron, we take the product and find the number of vertices in the facet $\Phi_I$. After finding the number of vertices in each facet $\Phi_I$ of the nestohedron $\nest{P}{\B}$, we sum the values up and divide by $n$.

The enumeration method outlined in this section introduces two complications to this method. First of all, every simple polyhedron is dual to a simplicial complex, and so we can get more general results by proving everything for enumeration of maximal faces in simplicial complexes. Second of all, we wish to enumerate not only the vertices of simple polyhedra, but instead enumerate all $k$-dimensional faces. We can do this by taking products of bivariate $f$-polynomials.

\subsection{Atomic link sum polynomial method for calculating $f$-polynomials of simplicial complexes}

An atomic element of a poset with minimal element is an element which covers the minimal element. The atomic elements of a simplicial complex are the singleton faces of that simplicial complex, and so we call the link of a singleton face of a simplicial complex an \emph{atomic link}.

Define the \emph{atomic link sum polynomial} of a simplicial complex $\Delta$ to be the sum
\[
	R^\Delta(s) = \sum_{x \in \mathcal{S}} f^{\Delta / \{x\}} (s)
\]
of $f$-polynomials of the links of every singleton subset of the base set $\mathcal{S}$ of $\Delta$. For every face $F$ of $\Delta$ with $|F| \ge 1$, there is a face $F \backslash \{x\}$ in $\Delta/\{x\}$ for each $x \in F$. This means that each face of size $k$ is counted $k$ times in $\sum_{x \in \mathcal{S}} f^{\Delta/\{x\}}_k$, and $\sum_{x \in \mathcal{S}} f^{\Delta/\{x\}}_k = k f_{k+1}^\Delta$ for all $k \ge 1$. As a result, we write $R^\Delta(s) = f_1^\Delta+2f_2 ^\Delta s + 3 f_3^\Delta s^2 + \cdots$. This can be reduced to a derivative:
\[
	R^\Delta (s) = D_s f^\Delta(s),
\]
where $D_s$ is the differential operator representing the derivative with respect to $s$. When $\Delta$ is a family $\{\Delta_0,\Delta_1,\ldots\}$, we define $R^\Delta(s,t)=\sum_{n\ge 0} R^{\Delta_n}(s) t^n$.

We consider a set of families of simplicial complexes indexed by $\sigma$ in a set $S$, such that for each element $\sigma \in S$, we find $\Acom^\sigma = \{\Acom^\sigma_0,\Acom^\sigma_1, \ldots\}$, such that each simplicial complex $\Acom^\sigma_n$ is rank $n$. We call each element $\sigma \in S$ a \emph{shape}. We can consider another set $\Bcom$ of families, a set $\Bcom^\tau$ for each element $\tau$ in a set $T$, such that each complex $\Bcom^\tau_n$ is rank $n$. For a family $\Delta$ of simplicial complexes and pair $\Acom,\Bcom$ of sets of families, we say that $\Delta$ is \emph{atomically closed under $\Acom$ and $\Bcom$} if for each complex $\Delta_n$, for every singleton face  $F$ of $\Delta_n$, there exists shapes $\sigma \in S$ and $\tau \in T$ and integer $i \in \{0,\ldots,n-1\}$ such that the atomic link of $F$ is isomorphic to the product $\Acom^\sigma_i \times \Bcom^\tau{n-1-i}$.

\bigskip

Consider a family of simplicial complexes $\Delta=\{\Delta_0,\Delta_1,\ldots\}$ which is atomically closed under $\Acom$ and $\Bcom$. Consider an atomic link of some complex $\Delta_n$; it is possible that there is more than one possible way to write this link as the product of a complex in $\Acom$ and a complex in $\Bcom$. In this case, for each atomic link in $\Delta_n$, we will choose a unique triple $(i,\sigma,\tau)$ of elements such that the atomic link is isomorphic to $\Acom_i^\sigma \times \Bcom^\tau_{n-1-i}$, called the \emph{canonical decomposition} of that atomic link. Now, for each $n$, we define $r_{\sigma,\tau}(n,i)$ to be the number of atomic links of $\Delta_n$ whose canonical decomposition is $(i,\sigma,\tau)$.

As a result, for each $n$ we have counted the combinatorial types of each atomic link of $\Delta_n$, such that for each $i, \sigma, \tau$, we find the product $\Acom_i^\sigma \times \Bcom_{n-1-i}^\tau$ appears $r_{\sigma,\tau}(n,i)$ times, without any overcounting. We then note that the $f$-polynomial of $\Acom_i^\sigma \times \Bcom_{n-1-i}^\tau$ is the product $f^{\Acom_i^\sigma}(s) \times f^{\Bcom_{n-1-i}^\tau}(s)$. As a result, we find

\[
	R^{\Delta_n} (s) = \sum_{\sigma \in S, \tau \in T} \sum_{i=0}^{n-1} r_{\sigma,\tau}(n,i) f^{\Acom_i^\sigma}(s) \times f^{\Bcom_{n-1-i}^\tau}(s) .
\]

Multiplying by $t^n$ for each $n$ and summing over all $n\ge 0$ gives the following.

\[
	R^\Delta(s,t) =  \sum_{n \ge 0} \sum_{\sigma \in S, \tau \in T} t \sum_{i=0}^{n-1} r_{\sigma,\tau}(n,i) f^{\Acom_i^\sigma}(s)t^i \times f^{\Bcom_{n-1-i}^\tau}(s) t^{n-1-i}.
\]

Consider the case that $r_{\sigma, \tau}(n,i)$ can be written as a sum of functions separable in $i,n-1-i$, of the form \[r_{\sigma,\tau}(n,i) = \sum_{j=1}^k r_{\sigma,\tau}^{\Acom,j}(i) r_{\sigma,\tau}^{\Bcom,j}(n-1-i).\]

In this event, we can express the sum $R^\Delta(s,t)$ as

\[
	R^\Delta(s,t) = \sum_{n \ge 0} \sum_{ \sigma \in S, \tau \in T}\sum_{j \ge 1}^k t \sum_{i=0}^{n-1} \left[r^{\Acom,j}_{\sigma,\tau}(i) f^{\Acom_i^{\sigma}}(s)t^i\right] \times \left[r^{\Bcom,j}_{\sigma,\tau} (n-1-i)f^{\Bcom_{n-1-i}^\tau}(s) t^{n-1-i}\right]
\]

	This is a convolution of the power series $r_{\sigma,\tau}^{\Acom,j}(n) f^{\Acom^\sigma_n}(s)t^n$ and $r_{\sigma,\tau}^{\Bcom,j}(n) f^{\Bcom^\tau_n}(s)t^n$. As a result, we can rewrite

\[
	R^\Delta(s,t) = \sum_{j =1}^k \sum_{\sigma \in S, \tau \in T} t
\left( \sum_{n\ge 0} r_{\sigma,\tau}^{\Acom,j}(n) f^{\Acom^\sigma_n} (s) t^n\right)
\left( \sum_{n\ge 0} r_{\sigma,\tau}^{\Bcom,j}(n) f^{\Bcom^\tau_n} (s) t^n\right) .
\]

Given a power series $a(t)=\sum_{n \ge 0} a_n t^n$, define $O_t^r$ to be the linear operator on power series defined by $O_t^r\left[a(t)\right]\sum_{n\ge 0} r(n) a_n t^n$.  To simplify notation and avoid nested scripts, we will define $O_t^{\Acom,\sigma,\tau,j}=O_t^{r^{\Acom,j}_{\sigma,\tau}}$. Noting that $R^\Delta(s,t)=D_s f^{\Delta}(s,t)$, we can now arrive at a final equation.

\begin{proposition} \label{prop:diff-eq}
For a family $\Delta$ of simplicial complexes atomically closed under $\Acom$ and $\Bcom$, with counting function $r_{\sigma,\tau}(n,i)$ separable into functions of the form $r_{\sigma,\tau}^{\Acom,j}(i)$ and $r_{\sigma,\tau}^{\Bcom,j}(n-1-i)$ and defining linear operators $O_t^{\Acom,\sigma,\tau,j}$ and $O_t^{\Bcom,\sigma,\tau,j}$ on power series, the following partial differential equation holds:
\[
	D_s f^{\Delta} (s,t) = t \sum_{\sigma \in S, \tau \in T}\sum_{j=1}^k O^{\Acom,\sigma,\tau,j}_t\left[f^{\Acom^\sigma}(s,t)\right] O_{t}^{\Bcom,\sigma,\tau,j} \left[f^{\Bcom^\tau}(s,t)\right]
	.
\]

\end{proposition}

We should note that for a family of simplicial complexes, we assume $f^{\Delta_n}_0=1$ for each $n \ge 0$. This means that $f^{\Delta}(0,t)=1+t+t^2 \cdots = \frac{1}{1-t}$. This then creates a boundary condition $f^\Delta(0,t)=\frac{1}{1-t}$. This boundary condition is sufficient for the cases and operators presented in this thesis.

\begin{remark}
For many functions $r(n)$, the linear operator $O_t^r$ is fairly simple. For instance, when $r(n)=\delta(n)$, the function such that $\delta(n)=1$ if and only if $n=0$ and $0$ otherwise, we find $O_t^r[a(t)]=a(0)$. When $r(n)=n$, we find $O_t^r[a(t)]=t D_t [a(t)]$. Finally, $O_t^{rs}=O_t^r O_t^{s}$ and $O_t^{r+s}=O_t^r + O_t^s$, so we can calculate $O_t^r$ easily whenever $r$ is a polynomial or the $\delta$ function. We also note that $O_t^{\delta(n-i)}[a(t)]=a^{(i)}(0)/i!$.
\end{remark}

Now we can restate this for simple polyhedra. In the event that each simplicial complex in $\Delta, \Acom$, and $\Bcom$ of rank $n$ is isomorphic to the dual simplicial complex of an $n$-dimensional pointed simple polyhedron, we can perform a change of variables $x=1/s, y=st$ and note that $f^{\Delta}(s,t)=f^{P(\Delta)} (x,y)$. Performing this change of variables allows us to reformulate an earlier expression:

\[
	D_s f^\Delta(s,t) = \sum_{\sigma \in S, \tau \in T} \sum_{j=1}^k xy \left(r^{\Acom,j}_{\sigma,\tau}(n) f^{P(\Acom^\sigma_n)}(x)y^n\right)
	\left(r^{\Bcom,j}_{\sigma,\tau}(n) f^{P(\Bcom^\tau_n)}(x)y^n\right).
\]

We can rewrite this using the same linear operator on power series $O_y^{\Acom,\sigma,j}$, this time using the variable $y$ instead of $t$. It should be emphasized that one does not need to re-calculate these linear operators beyond replacing every use of $t$ with $y$.

We can also rewrite $D_s f^\Delta(s,t)=xy D_y f^{P(\Delta)}(x,y)-x^2 D_x f^{P(\Delta)}(x,y)$. After dividing by $x$, we can finally write the equation as follows.

\begin{proposition} \label{prop:diff-eq-polyhedra}
When each simplicial complex in $\Delta, \Acom,$ and $\Bcom$ from Proposition \ref{prop:diff-eq} is dual to a pointed simple polyhedron, the following differential equation holds:
\[
	(yD_y -x D_x) f^{P(\Delta)}(x,y) = \sum_{\sigma \in S, \tau \in T} \sum_{j =1}^k y O^{\Acom,\sigma,\tau,j}_y \left[f^{P(\Acom^\sigma)}(x,y)\right]O^{\Bcom,\sigma,\tau,j}_y\left[f^{P(\Bcom^\tau)}(x,y)\right].
\]
\end{proposition}

\subsection{$\p$-graph case}

Consider a family $P=\{P_0,P_1,\ldots\}$ of simple pointed polyhedra and a family of graphs $G=\{G_0,G_1, \ldots\}$ such that each graph $G_n$ is a $P_n$-graph. We wish to be able to calculate the $f$-vectors of the simplicial complex $\mathcal{N}(G_n,P_n)$ and the polyhedron $\nest{P_n}{G_n}$. We will use this method to calculate bivariate $f$-polynomial generating functions $f^{\Delta(G)}(s,t)$ and $f^{P(G)}(x,y)$.

The definition of atomically-closed families of simplicial complexes is inspired by self-similarity of associahedra and graph associahedra. Recall that every facet of a $\p$-graph associahedron is combinatorially isomorphic to the product of the associahedron of a $\p$-graph $G/t$ and a simplex-graph associahedron $G|_t$. Say that a family of $\p$-graphs $G=\{G_0,G_1,\ldots\}$ is \emph{facet-closed} under a pair of sets of $\p$-graph families $\Ag$ and $\Bg$ indexed by $n$, $\sigma$ and $\tau$ if, for every tube $t \in G_n$, the simplex graph $G|_t$ is isomorphic to a simplex-graph $\Ag_i^\sigma$ in $\Ag$ and the reconnected complement of $t$ is isomorphic to a $\p$-graph $\Bg_{n-1-i}^\tau$ in $\Bg$. We also say that the families of their $\p$-graph associahedra are facet-closed. Note that we are making a shift in convention here: while in the previous subsection, $\Ag$ and $\Bg$ were sets of families of simplicial complexes, we note that here $\Ag$ and $\Bg$ are $\p$-graphs, with each graph $\Ag_n^\sigma$ a $\p$-graph on an $n$-dimensional simplex, and each graph $\Bg_n^\tau$ a $\p$-graph of an $n$-dimensional simple polyhedron.

We take note of some well-known facet-closed families of graph associahedra. The associahedron is a simplex-graph associahedron of a path graph. Every tube of a path graph induces a path graph, and the reconnected complement is a path graph, so we say that path simplex-graphs are facet-closed under path simplex-graphs and path simplex-graphs, and associahedra are facet-closed under associahedra. Alternatively, every facet of a cyclohedron is isomorphic to the product of a cyclohedron and an associahedron, so it is facet-closed under the families of associahedra and cyclohedra. For a much larger family, consider the set of all $\p$-graphs. Every facet of a $\p$-graph associahedron is the product of lower-dimensional $\p$-graph associahedra, and so this set of graphs is facet-closed if we choose a shape set $S$ with infinite possible shapes.

There are several features in families of graphs which make them good candidates for atomic link sum enumeration. First, we wish for a small number of possible tube shapes, and a small number of possible reconnected complement shapes. Secondly, we consider how the functions $r_{\sigma,\tau}^{\Ag,j}$ and $r_{\sigma,\tau}^{\Bg,j}$ are written. If these consist of only polynomials and $\delta$ functions, then $O_t^{\Ag,\sigma,\tau,j}$ and $O_t^{\Bg,\sigma,\tau,j}$ are likely to be easy to calculate and express in terms of differential operators. Third, we consider self-similarity. If one of the families $\Ag^\sigma$ or $\Bg^\tau$ is equal to the family $\Delta$, then the final equation will involve both $D_s f^\Delta(s,t)$ on the left hand side, and will involve a linear operator $O_t$ operating on $f^\Delta(s,t)$ on the right hand side. If all values of $f^{\Ag^\sigma}, f^{\Bg^\sigma}$ are known, then $D_s f^\Delta$ can be calculated directly and we can integrate to find $f^\Delta$. However, self-similarity can generate a more difficult partial differential equation to solve. In this thesis, we have proven results for several cases which do involve self-similarity by creating a partial differential equation, and using it to verify a hypothesis.

\begin{proposition}
Consider a family of simplicial complexes $\Delta=\{\Delta_0,\Delta_1,\ldots\}$ and a collection of $\Delta$-graphs $G=\{G_0,G_1,\ldots\}$ such that each graph $G_n$ is a $\Delta_n$-graph. Say that a family of graphs $G$ decomposes into families of graphs $\Ag$ and $\Bg$, such that for every tube shape $\sigma$ and reconnected complement shape $\tau$ in $G_n$, there are $r_{\sigma,\tau}(n,i)$ tubes of $G_n$ containing $i+1$ elements whose induced simplex graph is isomorphic to $\Ag_i^\sigma$ and whose reconnected complement graph is isomorphic to $\Bg_{n-1-i}^\tau$. Assume that each $r_{\sigma,\tau}(n,i)$ is separable into a finite sum of products of functions of the form $r^{\Ag,j}_{\sigma,\tau}$ and $r^{\Bg,j}_{\sigma,\tau}$. If $f^G, f^{\Ag^\sigma}, f^{\Bg^\tau}$ are taken as the bivariate $f$-polynomials of the nested complexes of the respective families of graphs, we find the following:
\[
	D_s f^{\Delta(G)} (s,t) = t \sum_{\sigma \in S, \tau \in T}\sum_{j=1}^k O^{\Ag,\sigma,\tau,j}_t\left[f^{\Delta(\Ag^\sigma)}(s,t)\right] O_{t}^{\Bg,\sigma,\tau,j} \left[f^{\Delta(\Bg^\tau)}(s,t)\right].
\]
\end{proposition}

The following is a restatement of this proposition for the $\p$-graph case.

\begin{proposition} \label{prop:p-graph-theorem}
Consider a family of simple pointed polyhedra $P=\{P_0,P_1,\ldots\}$ such that each polyhedron $P_n$ is $n$-dimensional, and a family of $P$-graphs $G=\{G_0,G_1,\ldots\}$ such that each graph $G_n$ is a $P_n$-graph. Say that a family of graphs $G$ decomposes into a pair of sets of families of $\p$-graphs $\Ag$ and $\Bg$ indexed by variables $n\ge 0$ and shapes $\sigma \in S, \tau \in T$, such that for every tube $t$ in graph $G_n$ with $|t|=i+1$ with shape $\sigma$ and reconnected complement shape $\tau$, the simplex-graph $G_n|_t$ is isomorphic to the simplex-graph $\Ag_i^\sigma$, and the reconnected complement of $t$ in $G_n$ is a $F_t$-graph isomorphic to the $\p$-graph $\Bg_{n-1-i}^\tau$. Assume that each $r_{\sigma,\tau}(n,i)$ is separable into a finite sum of products of functions of the form $r^{\Ag,j}_{\sigma,\tau}$ and $r^{\Bg,j}_{\sigma,\tau}$. If $f^G, f^{\Ag^\sigma}, f^{\Bg^\tau}$ are taken as the bivariate $f$-polynomials of the $\p$-graph associahedra of their respective families of graphs, we find the following:
\[
(yD_y -x D_x) f^{P(G)}(x,y) = \sum_{\sigma \in S, \tau \in T} \sum_{j =1}^k y O^{\Ag,\sigma,\tau,j}_y \left[f^{P(\Ag^\sigma)}(x,y)\right]O^{\Bg,\sigma,\tau,j}_y\left[f^{P(\Bg^\tau)}(x,y)\right].
\]
\end{proposition}

\section{Maximal tube sub-complexes for calculating $f$-polynomials} \label{sec:method-two}

The previous section relies upon a method of overcounting, where faces of simplicial complexes or simple polyhedra are counted more than once and our count is adjusted accordingly. This section presents a method where tubings are divided by cases, and each tubing is counted exactly once. The broad concept is as follows: for any tube $t$, we wish to characterize all tubings $T$ that contain the tube $t$, such that $t$ is maximal in $T$. Next, we wish to pick a set of tubes $t_1, \ldots, t_k$ such that only one tube out of the set can ever be maximal in a tubing at the same time. Finally, we characterize the set of tubings that do not contain any tube in the set $t_1, \ldots, t_k$.

	\subsection{Maximal tube sub-complexes}

\newcommand{\nestmax}[3]{\mathcal{N}_{/#3 \max}(#1, #2)} 

		\begin{definition}
		The maximal tube sub-complex $\nestmax{G}{\Delta}{t}$ is the set of all tubings such that $t \notin T$, $\{t\} \cup T$ is a valid tubing and $t$ is maximal in $\{t\} \cup T$.
	\end{definition}

	This is equivalent to taking the set of tubings $T$ such that $t \in T$ and $t$ is maximal in $T$, and removing $t$ from each $T$. It is also equivalent to taking the complex $\mathcal{N}(G,\Delta)$, removing any tubes which contain $t$ as a proper subset, and then taking the link of $t$ in this new complex.

	Define $n(t)$ to be the \emph{neighborhood} of a tube $t$ in $G$; that is, if $v \in t$ or $v$ is adjacent to a vertex in $t$, then $v \in n(t)$. Define $xn(t) = n(t) \backslash t$ to be the \emph{exclusive neighborhood} of $t$. For a graph $G$ and vertex subset $S$, we define the graph $G \backslash S$ to be just the graph induced on the complement of $S$ in the vertex set of $G$. For a simplicial complex $\Delta$ and subset $S$ of the base set of $\Delta$, we define $\Delta \backslash S$ to be the simplicial complex obtained by deleting any faces containing any elements of $S$. We note that a tube $t$ of $G$ is a tube of $G \backslash S$ if and only if $t \cap S = \emptyset$. From there, the following proposition is trivial:

	\begin{proposition} \label{prop:deleted-vertices-in-graph-stuff}
		The tubing complex $\mathcal{N}(G \backslash S,\Delta \backslash S)$ for any subset $S$ of the base set of $\Delta$ is equal to the set of tubings in the complex $\mathcal{N}(G,\Delta)$ such that $t \cap S = \emptyset$ for each tube $t$ in the tubing.
	\end{proposition}

	Recall that $G/t$ is the reconnected complement of the $\Delta$-graph $G$, and it is a graph of the link $\Delta/t$ obtained by connecting neighbors of $t$. We now define the removal of vertices $xn(t)$ from $G/t$.

	\begin{definition}
		The \emph{neighborless complement} of a tube $t$ in a $\Delta$-graph $G$ is the $(\Delta/t)\backslash xn(t)$-graph $(G/t) \backslash xn(t)$.
	\end{definition}

	We can now prove a statement:

	\begin{proposition}
		For a $\Delta$-graph $G$ with tube $t$, the tubing complex, $\mathcal{N}((G/t)\backslash xn(t),(\Delta/t)\backslash xn(t))$, of the neighborless complement is equal to the set of tubings $T$ in $\nestmax{G}{\Delta}{t}$ such that, for every tube $t' \in T$, $t' \cap t = \emptyset$.
	\end{proposition}

	\begin{proof}
	We wish to characterize the tubes in $\nestmax{G}{\Delta}{t}$. This complex contains all tubes $t' \subset t$, which are tubes of the simplex-graph $G|_t$, where $G|_t$ is the simplex-graph induced by $t$. It also contains all tubes $t'$ such that $t, t'$ are compatible and $t' \not\subseteq t$ and $t \not\subseteq t'$. We note that if $t, t'$ are compatible, then $t \subset t'$ if and only if $t' \cap xn(t) \ne \emptyset$. Recall also that $t'$ is a tube compatible with but not contained in $t$ if and only if $t' \backslash t$ is a tube of $G/t$. As a result, we note that for two compatible tubes $t, t'$, we find $t \not\subseteq t', t' \not\subseteq t$ if and only if $t'$ is a tube of $G/t$ not containing any vertex in $xn(t)$. 
	
	Now we note that $T$ is a tubing in $\nestmax{G}{\Delta}{t}$ if and only if it is a tubing in $G/t$ containing no vertex in $xn(t)$ in any of its tubes. Proposition \ref{prop:deleted-vertices-in-graph-stuff} shows that this is the case if and only if $T$ is a tubing of $(G/t)\backslash xn(t)$.
\end{proof}

	Figure \ref{fig:neighborless-complement} illustrates the finding of the neighborless complement of a tube in a hypercube-graph in three steps. First, we take the reconnected complement; this removes the vertices in $t$ from $G$, as well as any vertices not in $\Delta/t$, which in this case is the vertices of $-t$. Finally, we remove all neighbors of $t$.

\begin{figure}
\centering
\includegraphics[width=.5\textwidth]{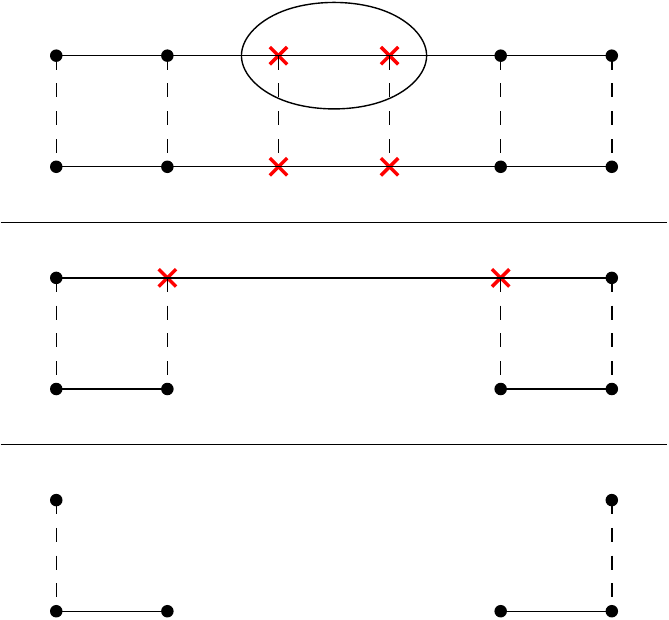}
\caption{Reconnected complement and neighborless complement of a tube in a hypercube graph.}\label{fig:neighborless-complement}
\end{figure}

\danger{The nested complex notation says $\mathcal{N}(\B,\Delta)$, so $\mathcal{N}(G, \Delta|t\max )$ might be the correct version. Maybe even $\mathcal{N}(G,\Delta,t=\max)$ ?? One problem is that I have $G|_t$. So maybe $\mathcal{N}_{t\max}(G,\Delta)$. I am converting this to a newcommand to make things easier on me.}


	For a graph $G$, refer to $\mathcal{N}(G)$ as the simplex-graph tubing complex of $G$. We can now characterize maximal tubing complexes entirely.

	\begin{proposition}
		The complex $\nestmax{G}{\Delta}{t}$ is equal to the Cartesian product of tubing complexes $\mathcal{N}(G|_t)$ and $\mathcal{N}((G/t)\backslash xn(t),(\Delta/t)\backslash xn(t))$.
	\end{proposition}

	\begin{proof}

	We know that the set of tubes $t'$ in the base set of $\nestmax{G}{\Delta}{t}$ such that $t' \cap t=\emptyset$ are just tubes of the neighborless complement of $t$. The tubes $t'$ in $\nestmax{G}{\Delta}{t}$ such that $t' \cap t \ne \emptyset$ are exactly the tubes which are subsets of $t$, and all such proper subsets are tubes of the simplex-graph $G|_t$. We then note that any tube $t'$ in the neighborless complement and any tube $t''$ in $G|_t$ are disjoint and non-adjacent, and so it is trivial to see that every tubing $T$ is equal to the disjoint union of a tubing in the neighborless complement of $t$ and $G|_t$, and the maximal tube complex of $t$ is equal to the product of the two complexes.
	\end{proof}

	Now that we have defined the maximal tube complex of a tube $t$, we need to find an easy method to characterize a set of tubes such that no two tubes in the set can ever both be maximal. We also wish to be able to easily enumerate the set of tubings which contain no tubes in this set.

	Consider a $\Delta$-graph $G$. For a set $X$ of vertices of $G$, define $intset_X$ to be the set of tubes of $G$ that have nonempty intersection with $X$. We define a \emph{kingmaker set} to be a set $X$ of vertices of a $\Delta$-graph $G$ such that, for any two compatible tubes $t_1, t_2 \in intset_X$, we find $t_1 \subset t_2$ or $t_2 \subset t_1$. There are several ways this condition may hold; for instance, every clique in a graph is a kingmaker set, even including every 1-clique.

	We note that for every tubing $T$ of a graph with kingmaker set $X$, either there exists a unique tube $t \in intset_{X}$ such that $t \in T$ and $t$ is maximal in $T$, or no tube in $T$ intersects with $X$. In the first case, $T \backslash \{t\}$ is in $\nestmax{G}{\Delta}{t}$. In the second case, $T \in \mathcal{N}(G\backslash X, \Delta\backslash X)$. We recall that $f^{\mathcal{N}(G,\Delta)}(s)=\sum_{k\ge 0} f_k s^k$ counts the number of tubings of size $k$ in $G$. As a result, the $f$-vector of the set of tubings such that $t$ is maximal is equal to $s$ times the $f$-vector of $\nestmax{G}{\Delta}{t}$, giving us the following equation.

	\begin{proposition} \label{prop:kingmaker-partition}
		For a $\Delta$-graph $G$ with kingmaker set $X$,
		\[
			f^{\mathcal{N}(G,\Delta)} (s) = f^{\mathcal{N}(G\backslash X, \Delta \backslash X)} + s \sum_{t \in intset_X} f^{\nestmax{G}{\Delta}{t}}(s)
			.
		\]
	\end{proposition}

	Finally, we recall that we can write $\nestmax{G}{\Delta}{t}$ as a Cartesian product of tubing complexes. If $\mathcal{N}(G|_t)$ is the simplex-graph tubing complex induced by $G|_t$, we get the following equation.

	\begin{proposition} \label{prop:kingmaker-partition-2}
		For a $\Delta$-graph $G$ with kingmaker set $X$,
		\[
			f^{\mathcal{N}(G,\Delta)} (s) = f^{\mathcal{N}(G\backslash X, \Delta \backslash X)} + s \sum_{t \in intset_X} 
	f^{\mathcal{N}(G|_t)}(s)
	f^{\mathcal{N}((G/t)\backslash xn(t),(\Delta/t)\backslash xn(t))}(s).
		\]
	\end{proposition}

	\begin{remark}
		What is the rank of $\mathcal{N}(G\backslash X, \Delta\backslash X)$ or $\nestmax{G}{\Delta}{t}$? We recall that if $\Delta$ is a pure simplicial complex of rank $n$ with face $S$, then $\Delta/S$ is rank $n-|S|$, allowing us to very nicely index our terms and create a function of bivariate generating polynomials. In contrast, if $X$ is a set in the base set of $\Delta$, then $\Delta \backslash X$ could have any rank between $n-|X|$ and $n$. We know the rank of $\mathcal{N}(G|_t)$ is $(|t|-1)$, and if $\Delta$ is a pure rank $n$ simplicial complex, then $(\Delta/t)$ is rank $(n-|t|)$. The rank of $(\Delta/t)\backslash xn(t)$ is then not fixed, and depends on the vertices in $xn(t)$.
	\end{remark}

	\subsection{Calculating $f$-polynomial through maximal tube enumeration}

	In this subsection, we will apply Proposition \ref{prop:kingmaker-partition-2} in order to calculate the bivariate $f$-polynomial of a family of $\Delta$-graph nested complexes. Consider a family of simplicial complexes $\Delta=\{\Delta_0,\Delta_1,\ldots\}$, where each simplicial complex $\Delta_n$ has rank $n$. Consider a family of $\Delta$-graphs $\{G_0,G_1,\ldots\}$ such that each graph $G_n$ is a $\Delta_n$-graph.

	We first wish to calculate the $f$-polynomials of each nested complex $\mathcal{N}(\Delta_n,G_n)$. We will pick a kingmaker $X_n$ for each graph $G_n$. We also define a pair $\Ag, \Bg$ of sets of families of simplicial complex-graphs, such that $\Ag_n^\sigma$ is a simplicial complex-graph of a rank $n$ complex for each shape $\sigma \in S$, and $\Bg_n^\tau$ is a simplicial complex-graph of a rank $n$ complex for each shape $\tau \in T$, and a constant $\gamma_{\sigma,\tau}$ for each pair $\sigma, \tau$, as follows.
For each tube $t \in intset_{X_n}$ with $i+1$ vertices, the complex $\nestmax{G}{\Delta}{t}$ is isomorphic to the product of complexes $\mathcal{N}(G|_t)$ and $\mathcal{N}((G_n/t)\backslash xn(t),(\Delta_n/t)\backslash xn(t)))$. We know that $\mathcal{N}(G|_t)$ must be rank $i$. We also know that the reconnected complement nested complex has rank $n-1-i$, but the neighborless complement nested complex may have a rank that is of lesser value. We make an assumption that for a pair $\sigma, \tau$, the neighborless complement nested complex of a tube of shape $\sigma$ and neighborless complement of shape $\tau$ has rank $n-1-i-\gamma_{\sigma,\tau}$ for some constant $\gamma_{\sigma,\tau}$. With these assumptions, we define graph families $\Ag^\sigma$ and $\Bg^\tau$ such that for any tube $t$ containing $i+1$ vertices of shape $\sigma$ and has neighborless complement of shape $\tau$, we define $\Ag_i^\sigma$ which is isomorphic to the simplex-graph $\mathcal{N}(G|_t)$, and $\Bg_{n-1-i-\gamma_{\sigma,\tau}}^\tau$ isomorphic to the neighborless complement simplicial-complex graph $\mathcal{N}((G_n/t)\backslash xn(t),(\Delta_n/t)\backslash xn(t)))$. We may say the tube $t$ \emph{splits} the graph $G_n$ into graphs $\Ag_i^\sigma$ and $\Bg_{n-1-i-\gamma_{\sigma,\tau}}^\tau$.



Finally, assume that there exists a constant nonnegative number $\rho$ such that the simplicial complex $\Delta_n \backslash X$ is rank $n-\rho$ for each $n \ge \rho$. We define a family $\{\Cg_0,\Cg_1,\ldots\}$ of simplicial complex-graphs such that each $\Cg_{n-\rho}$ is the complement graph, or the simplicial complex graph $G_n \backslash X_n$ on complex $\Delta_n \backslash X_n$. As a result, if $r_{\sigma,\tau}(n,i)$ is the number of tubes of shape $\sigma$ containing $i+1$ vertices with neighborless complement of shape $\tau$, we find the following result:
	\[
		f^{\Delta(G_n)} (s) = f^{\Delta(\Cg_{n-\rho})}(s) + s\sum_{\sigma \in S, \tau \in T}
		\sum_{i=0}^{n-1-\gamma_{\sigma, \tau}}
		r_{\sigma, \tau}(n,i)
		f^{\Delta(\Ag^\sigma_i)} (s)
		f^{\Delta(\Bg^\tau{n-1-\gamma_{\sigma, \tau}-i})}(s).
	\]
	We can then create a bivariate generating function by multiplying by $t^n$ and summing for each $n \ge \rho$.
		\begin{align*}
		\sum_{n \ge \rho} f^{\Delta(G_n)} (s)t^n =& 
		\sum_{n \ge \rho} f^{\Delta(\Cg_{n-\rho})}(s)t^\rho +\\
		&\sum_{n \ge \rho} s\sum_{\sigma \in S, \tau \in T}t^{1+\gamma_{\sigma, \tau}}
		\sum_{i=0}^{n-1-\gamma_{\sigma, \tau}}
		r_{\sigma, \tau}(n,i)
		f^{\Delta(\Ag^\sigma_i)} (s)
		f^{\Delta(\Bg^\tau_{n-1-\gamma_{\sigma, \tau}-i})}(s)
		t^{n-1-\gamma_{\sigma, \tau}}
	\end{align*}
We want to rewrite this equation by summing for all $n$, not just $n \ge \rho$. We will introduce a pair of error terms. Define  \[\epsilon_L=\sum_{n=0}^{\rho-1} f^{\Delta(G_n)} (s)t^n\] for the left hand side, and
		\[
		\epsilon_R =
		\sum_{n =0}^{\rho-1} f^{\Delta(\Cg_{n-\rho})}(s)t^\rho +
		\sum_{n =0}^{\rho-1} s\sum_{\sigma \in S, \tau \in T}t^{1+\gamma_{\sigma, \tau}}
		\sum_{i=0}^{n-1-\gamma_{\sigma, \tau}}
		r_{\sigma, \tau}(n,i)
		f^{\Delta(\Ag^\sigma_i)} (s)
		f^{\Delta(\Bg^\tau{n-1-\gamma_{\sigma, \tau}-i})}(s)
		t^{n-1-\gamma_{\sigma, \tau}}.
	\]
We are now able to rewrite an earlier equation as
		\begin{align*}
		&\left(\sum_{n \ge 0} f^{\Delta(G_n)} (s)t^n \right)-\epsilon_L+\epsilon_R=\\
		&\sum_{n \ge 0}\left( f^{\Delta(\Cg_{n-\rho})}(s)t^\rho +
		s\sum_{\sigma \in S, \tau \in T}t^{1+\gamma_{\sigma, \tau}}
		\sum_{i=0}^{n-1-\gamma_{\sigma, \tau}}
		r_{\sigma, \tau}(n,i)
		f^{\Delta(\Ag^\sigma_i)} (s)
		f^{\Delta(\Bg^\tau_{n-1-\gamma_{\sigma, \tau}-i})}(s)
		t^{n-1-\gamma_{\sigma, \tau}}\right)
	\end{align*}

	We again assume that we are able to rewrite $r_{\sigma, \tau}(n,i)$ as a sum of separable functions, $r_{\sigma, \tau}(n,i)=\sum_{j=1}^k r_{\sigma, \tau}^{\Ag,j}(i) r_{\sigma, \tau}^{\Bg,j}(n-1-i-\gamma_{\sigma, \tau})$ for each pair $\sigma, \tau$. With these two substitutions, we can write

	\[
		f^{\Delta(G)} (s) -\epsilon_L+\epsilon_R= f^{\Delta(\Cg)}(s)t^\rho +\]

		\[
		\sum_{n\ge 0}
		\sum_{\sigma \in S, \tau \in T}
		\sum_{j=1}^k
		st^{1+\gamma_{\sigma, \tau}}
		\sum_{i=0}^{n-1-\gamma_{\sigma, \tau}}
		r_{\sigma, \tau}^{\Ag,j}(i)
		f^{\Delta(\Ag^\sigma_i)} (s)
		t^i
		r_{\sigma, \tau}^{\Bg,j}(n-1-i-\gamma_{\sigma, \tau})
		f^{\Delta(\Bg^\tau_{n-1-\gamma_{\sigma, \tau}-i})}(s)
		t^{n-1-i-\gamma_{\sigma, \tau}}.
	\]

	We can then rewrite this convolution as a product of power series. Using the same definition of the linear operator $\sum_{n\ge0} r_{\sigma, \tau}^{\Ag,j}(n) a_nt^n=O_t^{\Ag,\sigma, \tau,j}[a(t)]$, we get a sum of products of linear operations on $f^{\Delta(\Ag^\sigma)}$ and $f^{\Delta(\Bg^\tau)}$, which we can rewrite as the equation in the following proposition.

	\begin{proposition} \label{prop:maximal-tube-enumeration-big-result}
	Consider the case where $\Delta=\{\Delta_0,\Delta_1,\ldots\}$ is a family of simplicial complexes with $\Delta_n$ rank $n$, and graph family $G=\{G_0,G_1,\ldots\}$ is such that $G_n$ is a $\Delta_n$-graph, and $X_n$ is a kingmaker set of $G_n$. Consider the case that each tube $t$ of shape $\sigma$ containing $i+1$ vertices and having neighborless complement of shape $\tau$ splits $G_n$ into an induced graph isomorphic to simplex-graph $\Ag_i^\sigma$, and a neighborless complement graph isomorphic to $\Bg_{n-1-\gamma_{\sigma,\tau}-i}$. We also assume there is a constant value $\rho$ such that each complex $\Delta_n \backslash X_n$ is rank $n-\rho$ for $n \ge \rho$, and define $C_{n-\rho}$ equal to the $\Delta_n\backslash X_n$-graph $G_n\backslash X_n$. In this case, we find the following equation:
	\[
		f^{\Delta(G)}(s,t)-\epsilon_L+\epsilon_R
		=
		f^{\Delta(\Cg)}(s,t)t^\rho +
		st\sum_{\sigma \in S, \tau \in T} t^{\gamma_{\sigma, \tau}} \sum_{j=1}^k
		O_t^{\Ag,\sigma, \tau,j}\left[f^{\Delta(\Ag^\sigma)}(s,t)\right]
		O_t^{\Bg,\sigma, \tau,j}\left[f^{\Delta(\Bg^\tau)}(s,t)\right].
	\]
	\end{proposition}

	We now finish by taking note of $\epsilon_L$ and $\epsilon_R$. When $\rho=0$, these are both empty sums, and $\epsilon_L=\epsilon_R=0$. When $\rho=1$, $\epsilon_L=f^{\mathcal{N}(\Delta_0,G_0)}(s)$, and because $\Delta_0$ is is a rank $0$ simplicial complex, its tubing complex must be the empty set, and $\epsilon_L=1$. In calculating $\epsilon_R$, we calculate the terms for $n=0$, but indices $n-\rho$ and $n-1-\gamma_\sigma-i$ are all less than zero, and so $\epsilon_R = 0$. The calculation of $\epsilon_L, \epsilon_R$ is more complicated when $\rho \ge 2$ and requires manual calculation, but we do not need this case for this thesis and so we omit this case.

		\subsection{Half-Open Polyhedra}\label{sub:halfopen}

	Using this method to enumerate tubings of $\Delta$-graphs requires the calculation of the graphs $(G/t)\backslash xn(t)$ and $G\backslash X_n$, and the simplicial complexes $(\Delta_n/t)\backslash xn(t)$ and $\Delta_n\backslash X_n$ When $\Delta$ is the dual simplicial complex of a polyhedron, we know that $\Delta_n/t$ is also dual to a polyhedron. However, we have no guarantee that $(\Delta_n/t)\backslash xn(t)$ or $\Delta_n\backslash X_n$ is dual to a polyhedron. In general, if $P$ is a a simple polyhedron and $S$ is a set of facets of $P$, it is possible for $\Delta(P)\backslash S$ to not be dual to any polyhedron. We will define a new class of shapes related to polyhedra, called half-open polyhedra, in order to understand this.

	\begin{definition} \label{def:shaving}
	A \emph{half-open polyhedron} is any collection of points in a vector space $\R^n$ defined by a finite set of inequalities $\{cx \le a\}$ and a finite set of strict inequalities $\{dx < b\}$.
	\end{definition}

	Every half-open polyhedron is convex, and half-open polyhedra are a class of shapes which include polyhedra. Define the \emph{shaving} of a polyhedron $P$ by a set of proper faces $F_1,\ldots, F_k$ to be the set $P \backslash (F_1 \cup F_2 \cup \cdots \cup F_k)$.

	\begin{proposition}
	For any polyhedron $P$ and set of proper faces $S=\{F_1,\ldots, F_k\}$, the shaving of $P$ by $S$ is a half-open polyhedron.
	\end{proposition}

	\begin{proof}
		For each face $F_i$, there exists a vector $c_i$ in the normal cone of $F_i$ such that $c_i \cdot x$ is maximized on $F_i$, with $c_i \cdot x = a_i$ for all $x \in F_i$. The shaving of $P$ by $S$ is the half-open polyhedron defined by all non-strict inequalities defining $P$, and adding strict inequalities $c_i \cdot x < a_i$ for all $1 \le i \le k$.
	\end{proof}

	We define faces of half-open polyhedra in the same way we define faces of polyhedra.

		\begin{definition}
	A \emph{face} of a half-open polyhedron $Q$ is any set of points $F$ such that there exists a linear inequality $cx \le b$ for all $x \in Q$ and $x \in F$ if and only if $cx=b$ and $x\in Q$.
	\end{definition}

	\begin{proposition}
		Every half-open polyhedron is the shaving of a polyhedron by some set of faces.
	\end{proposition}
\begin{proof}
Consider a set of inequalities defining a half-open polyhedron $Q$ of the form $\{cx \le a\}$ and $\{dx < b\}$. Now define $Q'$ as the polyhedron defined by replacing every strict inequality $\{dx < b\}$ with the inequality $dx \le b$. We see that $Q'$ is a shaving of $Q$.
\end{proof}


	To further understand half-open polyhedra, we note that if $P'$ is a nonempty shaving of a polyhedron $P$, then we note that $relint(P) \subseteq P' \subseteq P$, where $relint(P)$ is the relative interior of $P$. This then means that the topological closure of $P'$ is equal to $P$.

	The \emph{face lattice} of a half-open polyhedron $P$ is the set of faces of $P$ ordered under inclusion. We define combinatorial isomorphism the same for half-open polyhedra as for polyhedra, but note that it behaves very differently. 
	
	\begin{definition}
	Two half-open polyhedra $P, P'$ are \emph{combinatorially isomorphic} if their face posets are related by a dimension-preserving isomorphism.
	\end{definition}

	Whereas we can typically imagine that every combinatorial isomorphism between polyhedra can be realized by a homeomorphism, this is not the case for half-open polyhedra. Figure \ref{fig:unintuitive-isomorphism} shows two half-open polyhedra, each a hexagon with four vertices shaved off. The two half-open polyhedra are combinatorially isomorphic, even if any combinatorial isomorphism would have to permute the order of the faces on the boundary. Figure \ref{fig:half-open-isomorphism}, on the other hand, shows a more intuitive example of combinatorially isomorphic half-open polyhedra; one case where we have shaved two sides off a square, one case where we have shaved one side from a triangle, and one cone.

	\begin{figure}
	\centering
	\includegraphics[width=.75\textwidth]{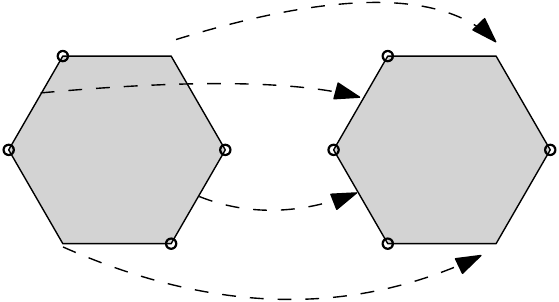}
	\caption{A surprising combinatorial isomorphism.}\label{fig:unintuitive-isomorphism}
\end{figure}

	\begin{figure}
	\centering
	\includegraphics[width=.75\textwidth]{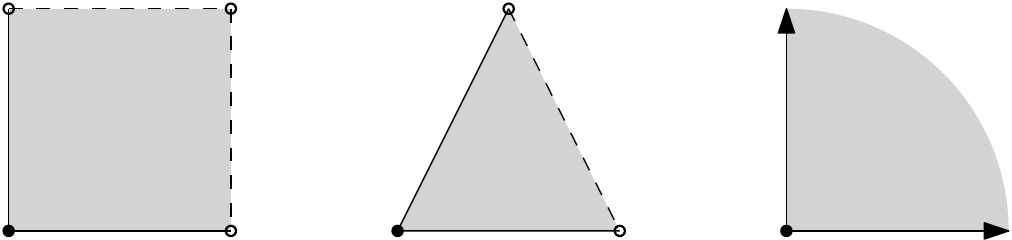}
	\caption{Three combinatorially isomorphic half-open polyhedra}\label{fig:half-open-isomorphism}
\end{figure}

	A \emph{simple} half-open polyhedron is a half-open polyhedron such that every codimension-$k$ face is contained in $k$ facets. We note that every shaving of a simple polyhedron is a simple half-open polyhedron. The \emph{dual simplicial complex} of a half-open polyhedron $P$ is a simplicial complex such that if $F_s$ for $s \in \mathcal{S}$ is the facet set of $P$, then $I$ is in the dual simplicial complex of $P$ if and only if $\bigcap_{s \in I} F_s \ne \emptyset$. These definitions are identical to those for polyhedra.

	We include the next definition to clarify Definition \ref{def:shaving} by way of contrast.

	\begin{definition}
		For a simple polyhedron $P$ defined by facet-defining inequalities $c_i x \le a_i$ for each facet $F_1,\ldots,F_k$ of $P$, the polyhedron defined by the \emph{removal} of a facet $F_j$ from $P$ is the polyhedron defined by inequalities $c_ i x \le a_i$ where $i \ne j$.
	\end{definition}

	There are some cases where shaving a polyhedron by a facet gives a half-open polyhedron isomorphic to the removal of that facet, but they are not isomorphic in general. Figure \ref{fig:facet-removal} demonstrates the difference between shaving by a facet, and removal of a facet. We note that the combinatorial type of the removal of a facet depends on the angle of adjacent facets, and not just the combinatorial type of the facet. For polytopes, we can define a projective transformation for each facet such that the shaving of that facet is isomorphic to the removal of that facet. However, this is not always possible for unbounded polyhedra or shaving multiple facets.

Figure \ref{fig:half-open-polyhedra} shows a case where the shaving of several facets gives a half-open polyhedron is not combinatorially isomorphic to any polyhedron, as the resulting polyhedron would have to contain three line faces and no other nonempty proper faces.
	
	\begin{figure}[h]
	\centering
	\includegraphics[width=.75\textwidth]{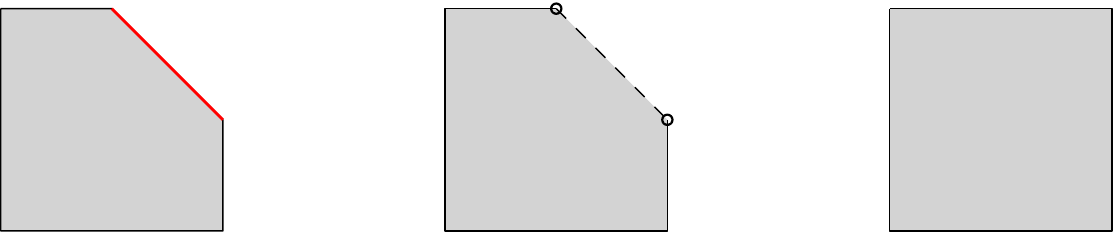}
	\caption{A pentagon, the shaving of the pentagon by a facet, and the removal of that facet.} \label{fig:facet-removal}
\end{figure}

	\begin{figure}[h]
	\centering
	\includegraphics[width=.5\textwidth]{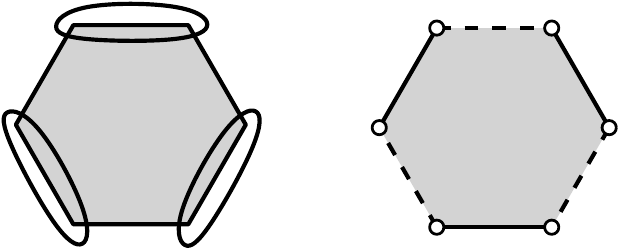}
	\caption{Shaving three facets from a hexagon to obtain a half-open polyhedron.} \label{fig:half-open-polyhedra}
\end{figure}

	\begin{proposition} \label{prop:shaved-simplex}
		When $P$ is an $n$-dimensional simplex, the shaving of $P$ by a facet is combinatorially isomorphic to an $n$-dimensional simplicial cone.
	\end{proposition}

	This operation can be visualized on the forbidden subset diagram of a simplex. Figure \ref{fig:shaving-forbidden-subsets} shows that the shaving of a facet $F$ removes the facet $F$ from the forbidden subset diagram, as well as the forbidden subset containing $F$, but not the rest of the vertices in the forbidden subset.

\begin{figure}[h]
	\centering
	\includegraphics[width=.75\textwidth]{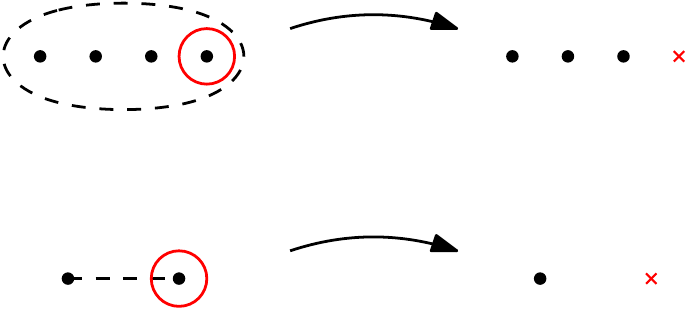}
	\caption{Forbidden subset diagrams of $3$- and $1$-dimensional simplices, and their resulting forbidden subset diagrams after shaving a facet.} \label{fig:shaving-forbidden-subsets}
\end{figure}

	The motivation behind all of this discussion of half-open polyhedra comes in the following proposition.

	\begin{proposition}
		When $P$ is a simple polyhedron with facet index set $\mathcal{S}$, and $X \subseteq \mathcal{S}$, the shaving of $P$ by all facets $F_s$ for $s \in X$ has a dual simplicial complex equal to $\Delta(P)\backslash X$.
\end{proposition}

	\begin{proof}
		Shaving by a facet $F_s$ removes all faces contained in that facet, which is analogous to removing any face $S$ of $\Delta(P)$ which contains $s$. We repeat this for every element $s \in X$.
	\end{proof}

	We can then refer to the half-open polyhedron obtained by shaving $P$ by a facet set $X$ as $P \backslash X$. We note that $P \backslash X$ is not always combinatorially isomorphic to a polyhedron. However, Proposition \ref{prop:shaved-simplex} shows that the shaving of a single facet from a simplex is isomorphic to a polyhedron, and it is trivial to show that the shaving of a set of facets from a product of simplices is isomorphic to a simple polyhedron that is the product of rays and simplices. Most importantly for our application, the shaving of any hypercube is isomorphic to the product of a set of $1$-simplices and rays.

	\subsection{$\p$-graph case}

	If $G$ is a family of $\p$-graphs, with each polyhedron $P_n$ being pointed and $n$-dimensional, then we recall Proposition \ref{prop:change-of-basis}, which will relate the generating functions $f^{P(G)}(x,y)=f^{\Delta(G)}(s,t)$ by the change of variables $x=1/s, y=st$. We then apply this change of variables to Proposition \ref{prop:maximal-tube-enumeration-big-result}.

%


	\begin{proposition} \label{prop:enumeration-maximal-tube-polyhedra}
When the sets of families of simplicial-complex graphs $G, \Ag, \Bg,$ and $\Cg$ in Proposition \ref{prop:maximal-tube-enumeration-big-result} are all $\p$-graphs of families of simple pointed polyhedra, and $\rho \le 1$, we find
	\[
		f^{P(G)}(x,y)-\epsilon_L= f^{P(\Cg)}(x,y) (xy)^\rho +
		y\sum_{\sigma \in S, \tau \in T} (xy)^{\gamma_{\sigma,\tau}} \sum_{j=1}^k
		O_y^{\Ag,\sigma,\tau,j}\left[f^{P(\Ag^\sigma)}(x,y)\right]
		O_y^{\Bg,\sigma,\tau,j}\left[f^{P(\Bg^\tau)}(x,y)\right]
	\]
where $\epsilon_L=1$ if $\rho=1$ and $\epsilon_L=0$ if $\rho=0$.
\end{proposition}

	\begin{remark}

	We should note what happens when $\rho =1$. If $G_n$ is a $P_n$-graph where $P_n$ is $n$-dimensional, then $P_n \backslash X_n$ is an $n$-dimensional half-open polyhedron, but its dual simplicial complex is a rank $n-1$ complex. This is because the half-open polyhedron $P_n\backslash X_n$ has no vertices. As a result, we will find the complex $\mathcal{N}(G_n\backslash X_n,P_n\backslash X_n)$ is rank $n-1$ as well. We are making the assumption in our proposition that $P_n\backslash X_n$ is isomorphic to some simple polyhedron, which we know is not possible for all half-open polyhedra. In this case, $P_n\backslash X_n$ is isomorphic to an $n$-dimensional unpointed polyhedron, which is isomorphic to the product of a line and a pointed $n-1$-dimensional polyhedron. We then define $C_{n-1}=G_{n}\backslash X_n$ as a $\p$-graph on this $n-1$-dimensional polyhedron. We see this happen often in the case of hypercube graph associahedra, and we will consider an example when $X_n=\{1,-1\}$. See Figure \ref{fig:shaving-dimension} which illustrates the case where $P_2$ is a square, and $X_2=\{1,-1\}$, a pair of opposing facets, and $G_2$ is a $P_2$-graph with no edges. The resulting half-open polyhedron $P_2\backslash X_2$ is unpointed, and its dual simplicial complex is isomorphic to that of a 1-simplex. This can be seen directly on the level of forbidden subset diagrams, as we see that the forbidden subset diagram of $P_2\backslash X_2$ is just the forbidden subset diagram of a $1$-simplex.
\end{remark}

\begin{figure}[h]
	\centering
	\includegraphics[width=.5\textwidth]{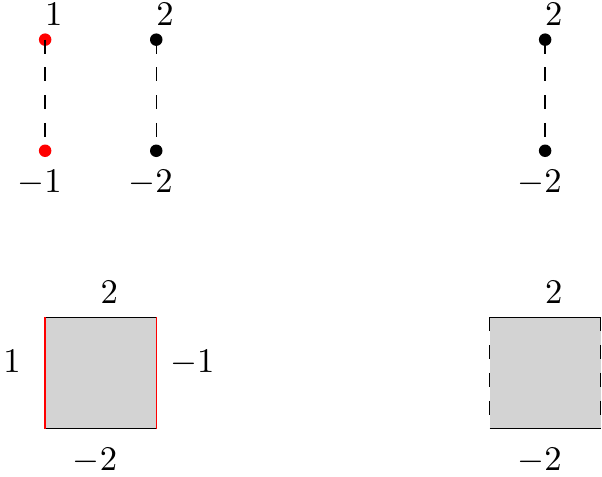}
	\caption{The forbidden subset diagram of a square, and the forbidden subset diagram of a square shaved by two facets.} \label{fig:shaving-dimension}
\end{figure}

\chapter{Examples and Enumeration}
\label{chap:res2}


	For a family of $\p$-graphs $G=\{G_0,G_1,\ldots\}$, we can calculate the generating functions of the bivariate $f$-polynomials of either the $\p$-graph associahedra of $G$, or the $\p$-graph tubing complexes of $G$. In Chapter \ref{chap:methods}, we wrote these functions as $f^{P(G)}(x,y)$, the generating function enumerating the $f$-polynomials of $\p$-graph associahedra of graphs in $G$, and $f^{\Delta(G)}(s,t)$, the generating function enumerating the $f$-polynomials of $\p$-graph tubing complexes of graphs in $G$.

	In this chapter, we will write $f^{G}(x,y)$ to denote the $f$-polynomial $f^{P(G)}(x,y)$, and we will often omit the $(x,y)$ part of this expression and write $f^{G}$. This is done in the spirit of simplifying several complicated expressions. Almost all calculations are performed using the $f$-polynomials of $\p$-graph associhahedra, and when the $f$-polynomial of a tubing complex is calculated, it will be specified as such, as $f^{\Delta(G)}(s,t)$. We also omit several indices during computations wherever it is unambiguous. For instance, when it is understood that path tubes with reconnected complements isomorphic to path graphs are being counted, we may write $r(n,i)$ instead of $r_{path,path}(n,i)$.

	Section \ref{sec:specialcases} lists different families of hypercube-graphs and focuses on interesting combinatorial properties of each hypercube-graph associahedron, as well as giving results enumerating the $f$-polynomials of each hypercube-graph associahedron. The results are sorted by hypercube graph type, with subsections \ref{sub:cubeahedra} and \ref{sub:doublecubeahedra} providing results for cubeahedra and double cubeahedra graphs respectively, and other subsections providing results for other hypercube graphs. The proofs for many of these enumeration results are not given directly in this section. Instead, Section \ref{sec:proofs} contains proofs for these results. One reason to separate the results from the proofs is the fact that many of the enumerative proofs are interdependent, and so the proofs are easier to follow if kept in a different section.

	As for notation: we say a graph contains a path $(v_1,\ldots,v_k)$ if it contains the edges $\{v_1,v_2\}, \{v_2,v_3\}, \ldots, \{v_{k-1},v_k\}$. A graph contains the cycle $(v_1,\ldots,v_k)$ if it contains the path $(v_1,\ldots,v_k)$ and the edge $\{v_1,v_k\}$.

\section{Families of Hypercube Graph Associahedra}\label{sec:specialcases}

	This section lists interesting hypercube graphs and states results calculating the bivariate $f$-polynomials for their associahedra. We note that some of these hypercube-graph associahedra are related to existing known classical graph associahedra. Figure \ref{polytopetable} lists four graphs with well-studied graph associahedra: the path graph, the complete graph, the cycle graph, and the star graph. We find three basic ways to turn a graph on $n$ vertices into a hypercube graph: the cubeahedron case, the double cubeahedron case, and the omni-graph case. In several cases, the cubeahedron or double cubeahedron of an existing graph corresponds to well-known polyhedron, and for some of these cases, we are able to find enumeration results which were not previously known. In other cases, we have defined what appear to be un-discovered polytopes.

	In addition to these hypercube graphs derived from simplex-graphs, we find several hypercube graphs which are not derived by taking copies of simplex-graphs. These include the twisted path and twisted cycle hypercube-graphs, the Pell and companion Pell hypercube-graphs, and the near double path graph. Figure \ref{graphs} shows a gallery of hypercube graphs.
	
	\begin{figure}
		\centering
		\caption{Notable hypercube graphs for the $4$-dimensional hypercube. The top left graph has vertices labeled with members of $\pm[4]$, with dashed lines connecting vertices corresponding to opposing facets; these are not actual edges in the hypercube graph. From top left to bottom right: an empty graph, a full adjacency graph, a $2K_n$ graph, a single path graph, a double path graph, a twisted path graph, a twisted cycle graph, and a single $K_n$ graph.}
		\label{graphs}
	\end{figure}


	
%

	\begin{figure}
	\begin{center}
\scalebox{.9}{	\begin{tabular}{|c|c|c|c|}
	\hline
	 & Graph Associahedron & Graph Cubeahedron \cite{devadoss2011} & Graph Double Cubeahedron \\
	\hline
	Path & $A_{n-1}$ associahedron & $A_n$ associahedron \cite{devadoss2011} & type $A_n$ linear biassociahedron\cite{barnard} \\
	$K_n$ & $A_{n-1}$ permutahedron & Stellohedron \cite{devadoss2011} & $A_n$ permutahedron \cite{multiplihedron}\\
	$n$-cycle & $B_{n-1}$ associahedron & Halohedron \cite{devadoss2011} & {\color{red} Cycle double cubeahedron} \\
	$n$-star & 	stellohedron & {\color{red} Stellar Cubeahedron} & {\color{red} Stellar Double Cubeahedron}\\
	\hline
	\end{tabular}}
	\caption{Polytopes obtained from graph associahedra, cubeahedra, and double cubeahedra of graphs with well-studied graph associahedra. Citations are given for cases where these polytopes have arisen before; entries in red apparently do not exist in the literature.} \label{polytopetable}
	\end{center}
\end{figure}

	\subsection{Graphs on Positive Vertices (Cubeahedra)} \label{sub:cubeahedra}

	Given a graph $G$ on vertices $[n]$, we define the hypercube graph $G^+$ as the graph on $\pm[n]$ such that for positive vertices $i,j \in [n]$, $\{i,j\} \in G^+$ if and only if $\{i,j\} \in G$, and there are no edges incident to negative vertices. We will prove that the hypercube graph associahedron of $G^+$ is isomorphic to the \emph{cubeahedron} of $G$ defined in \cite{devadoss2011}. The cubeahedron is defined by design tubes, which are defined as follows.

	\begin{definition}
		A \emph{round tube} is any subset of $[n]$ which induces a connected subgraph of $G$. A \emph{square tube} is a single element in $[n]$. Both are called \emph{design tubes}. Two design tubes are compatible if either
	\begin{enumerate}
		\item $t_1, t_2$ are both round, and either $t_1 \subset t_2, t_2 \subset t_1$, or $t_1, t_2$ are disjoint and not adjacent.
		\item One or both of $t_1, t_2$ is square, and $t_1, t_2$ are disjoint.
	\end{enumerate}
		A \emph{design tubing} is any set of pairwise compatible design tubes.
	\end{definition}

%

	\begin{proposition}
		For any graph $G$, the design tubing complex of $G$ is isomorphic to the hypercube-graph tubing complex of $G^+$.
	\end{proposition}

	\begin{proof}
		Recall from Proposition \ref{prop:hypercube-pairwise-compatibility} that a set of tubes in a hypercube-tube graph are a tubing if and only if they are pairwise compatible. Because design tubings are also defined by pairwise compatibility of design tubes, we find that the design tubing complex of $G$ and the hypercube-graph tubing complex of $G^+$ are isomorphic if and only if there exists a bijection $\phi$ mapping design tubes of $G$ onto hypercube-graph tubes of $G^+$, such that $t_1, t_2$ are compatible if and only if $\phi(t_1), \phi(t_2)$ are compatible.
	
		We define a bijection $\phi$ between design tubes of $G$ and hypercube-graph tubes of $G^+$. If $t$ is a round tube, then $\phi(t)=t$. We find this is a tube in $G^+$. If $t$ is a square tube, then $\phi(t)=-t$ is a singleton tube in $G^+$. We will show that pairwise compatibility is preserved by this mapping. If $t_1, t_1$ are both round tubes, then $t_1, t_2$ are compatible if and only if $t_1 \subset t_2, t_2 \subset t_1$, or $t_1, t_2$ are disjoint and not adjacent. These are exactly the terms of compatibility for tubes on positive vertices of $G^+$, and so $t_1, t_2$ are compatible as design tubes if and only if they are compatible in $G^+$. We find that any two square tubes are compatible, which is true of negative singleton tubes in $G^+$. Finally, we find that if $t_1$ is round and $t_2$ is square, we find that $t_1, t_2$ are compatible as design tubes if and only if $t_1, t_2$ are disjoint. This is equivalent to saying that $-t_2$ is not in $t_1$, which is the term of compatibility for a positive and negative term. As a result, $t_1, t_2$ are compatible if and only if $\phi(t_1), \phi(t_2)$ are. This proves the proposition.
	\end{proof}

	Figure \ref{fig:design-tubing} is an illustration of the map taking a design tubing for a path graph to its equivalent hypercube graph tubing. We define the \emph{cubeahedron} of a graph $G$ as a polyhedron obtained by truncation of a hypercube, and which is dual to the design tubing complex of $G$. We then see that the cubeahedron is combinatorially isomorphic to the hypercube-graph associahedron of $G^+$.

	\begin{figure}
\centering
\includegraphics[width=\textwidth]{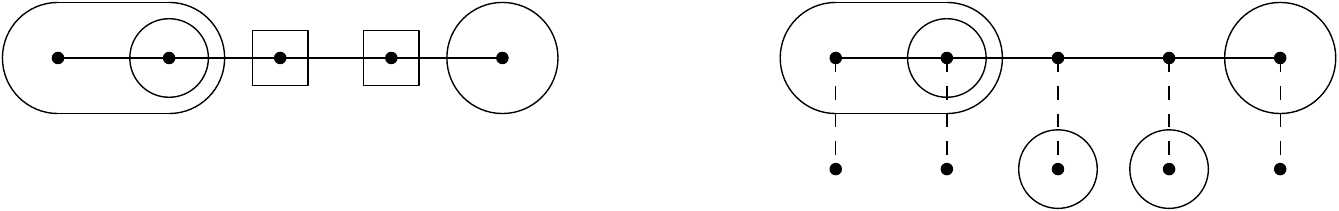}
\caption{A design tubing and an equivalent hypercube graph tubing.} \label{fig:design-tubing}
\end{figure}

	\subsubsection{Path on Positive Vertices and the Associahedron}

	
	Define the \emph{single path hypercube graph} to be the hypercube graph consisting of a path on the vertices $(1, 2, \ldots, n)$. We know that the type $A_n$ associahedron is the simplex graph associahedron for the path graph on $n+1$ vertices. The following is \cite[Proposition 15]{devadoss2011}, along with a sketch of the proof:

	\begin{proposition} \label{prop:cited-associahedron-result}
		The single path hypercube graph associahedron is combinatorially isomorphic to the associahedron.
	\end{proposition}
	\begin{proof}
		Define a map $f$ on the set of design tubes of a path on $n$ vertices, such that $f(t)=t$ for all round tubes, and $f(\{i\})=\{i+1,\ldots,n+1\}$ for all square tubes. This map defines a bijection between design tubes of the path graph on $n$ vertices, and tubes of the simplex-path graph on $[n+1]$. This map preserves compatibility, and induces an isomorphism between tubing complexes.
	\end{proof}

We will make a stronger statement in Proposition \ref{prop:linear-cluster-fan}. Recall Definition \ref{def:standard-cut} of a standard cut hypercube graph associahedron. We will further prove that the normal fan of a standard cut single path hypercube graph associahedron is linearly isomorphic to the linear $c$-cluster fan. In preparation, we will give some background of Coxeter theory and root systems.

Consider a finite Coxeter group $W$, represented in the usual way by a group of reflections in a vector space. A \emph{root system} associated with $W$ is a set of normal vectors to the set of reflecting hyperplanes of $W$. The group $W$ has a distinguished set of reflections $s_1, \ldots, s_n$, with simple roots $\alpha_1, \ldots, \alpha_n$ such that each root $\alpha_i$ is normal to the reflecting hyperplane of $s_i$. We note that every root in a root system can be written as a sum of integer multiples of simple roots.


The type $A_n$ Coxeter group can be realized as the set of permutations on $[n+1]$, and is generated by transpositions $s_1, \ldots, s_n$ such that each element $s_i$ for $1 \le i \le n$ is the transposition $(i \;\; i+1)$. We can define a type $A_n$ root system containing roots of the form $e_i-e_j$ for $i \ne j$ in $[n+1]$, and define the group action of $A_n$ on the vector space $\R^{n+1}$ such that for any vector, the permutation $\sigma$ permutes the coordinates of that vector. We may define a set of simple roots $\alpha_i = e_i-e_{i+1}$ for each $i \in [n]$, and note that the root $e_i-e_j$ for $i < j$ is equal to $\alpha_i + \cdots + \alpha_{j-1}$. We then define $\beta_{i,j}= \alpha_i + \alpha_{i+1} + \cdots + \alpha_{j}= e_i-e_{j+1}$ for any $1 \le i \le j \le n$. We note that the action of $s_i$ on $\alpha_i$ is $s_i \alpha_i = -\alpha_i$.

Given a Coxeter system with generators $s_1, \ldots, s_n$, a \emph{Coxeter element} is any element which can be written as $c=s_{w(1)}, \ldots, s_{w(n)}$, where $w$ is a permutation on $[n]$.

There are multiple types of Coxeter elements. One type, a \emph{bipartite} Coxeter element, sorts the elements such that $\{s_{\sigma(1)}, \ldots, s_{\sigma(k)}\}$ and $\{s_{\sigma(k+1)}, \ldots, s_{\sigma(n)}\}$ partition the generators of $W$ such that any two elements from the same set commute. These elements and their cluster fans are studied extensively in \cite{ysystems}.

When $W$ is the type $A_n$ Coxeter group, we can represent each element $s_i$ as the transposition $s_i = (i \;\; i+1)$ on $[n+1]$. The \emph{linear element} is the element $s_1 s_2 \cdots s_n$.

\begin{definition}
Given a set of simple roots $\alpha_1, \ldots, \alpha_n$ in a root system $\Phi$, the set of \emph{almost positive roots} is the set of roots in $\Phi$ that can either be written as a sum of nonnegative multiples of simple roots, or are equal to $-\alpha_i$ for some $i$.
\end{definition}

We note that in type $A_n$ with a set of simple roots $\alpha_1, \ldots, \alpha_n$, the set containing $-\alpha_i$ for all $i \in [n]$, and $\beta_{i,j}$ for all $1 \le i \le j \le n$ is a set of almost positive roots.


	The $c$-compatibility rules for a Coxeter element $c$ are as follows: for a Coxeter system with simple roots $\{\alpha_1, \ldots, \alpha_n\}$ and simple reflections $\{s_1, \ldots, s_n\}$, we define an operation on almost-positive roots:
	\[
		\sigma_i(\beta) := \begin{cases} \beta & \text{ for $\beta = -\alpha_j, j\ne i$} \\
		s_i \beta & \text{ otherwise} \end{cases}
	\]
	We write $[\beta:\alpha_i]$ as the coefficient of $\alpha_i$ in the expansion of $\beta$ in the basis of simple roots. 	We define \emph{$c$-compatibility} as a family of symmetric binary relations of the form $||_c$ where $c$ is a Coxeter element. We write that a simple reflection $s_i$ is \emph{initial} in $c$ if there is a reduced word for $c$ starting with $s_i$. If $s_i$ is initial in $c$, then $s_i c s_i$ is another Coxeter element. There is a unique family of relations with the properties
	\begin{enumerate}
		\item $-\alpha_i ||_c \beta$ if and only if $[\beta:\alpha_i]=0$.
		\item Given $s_i$ initial in $c$, we have $\beta_1 ||_c \beta_2$ if and only if $\sigma_i(\beta_1) ||_{s_i c s_i} \sigma_i(\beta_2)$.
	\end{enumerate}

We also define an operation $\tau = \sigma_{w(1)}  \cdots  \sigma_{w(n)}$. We then find that $\beta_1 ||_c \beta_2$ if and only if $\tau(\beta_1) ||_{c c c^{-1}} \tau (\beta_1)$, which is then equivalent to $\tau(\beta_1) ||_{c} \tau (\beta_2)$.

	\begin{proposition}
		The tubing complex of the $n$-dimensional path hypercube graph associahedron is isomorphic to the complex of pairwise-compatible roots for the linear Coxeter element $c=s_1\cdots s_n$ in $A_n$.
	\end{proposition}
\begin{proof}

	We can characterize $s_j \beta_{j,k}$ as follows:
	\begin{itemize}
\item If $i=j=k$, then $s_i \beta_{j,k}=(i, i+1) (e_i-e_{i+1})=-\alpha_i$.
\item If $i=j < k$, then $s_i \beta_{j,k} = \beta_{j+1,k}$
\item If $i=j-1$, then $s_i \beta_{j,k} = \beta_{j-1, k}$
\item If $j < i=k$, then $s_i \beta_{j,k}=\beta_{j,k-1}$
\item If $i=k+1$, then $s_i \beta_{j,k}=\beta_{j,k+1}$
\item Otherwise, $s_i \beta_{j,k}=\beta_{j,k}$.
\end{itemize}

Consider the action of $\tau=\sigma_1 \cdots \sigma_n$ on almost positive roots.

\begin{itemize}
\item If $k < n$, then we find $\tau \beta_{j,k} = \beta_{j+1,k+1}$. This is because we notice that $\sigma_{i} \beta_{i,j} = s_i \beta_{j,k}=\beta_{j,k}$ for all $i > k+1$. We then calculate $\sigma_1 \cdots \sigma_{k+1} \beta_{j,k}=(\sigma_1 \cdots \sigma_k) s_{k+1} \beta_{j,k}=(\sigma_1 \cdots \sigma_k)\beta_{j,k+1}$. For all values of $i$ such that $j+1 < i \le k$, we note that $\sigma_i \beta_{j,k+1}=\beta_{j,k+1}$. From there, we calculate $(\sigma_1 \cdots \sigma_{j+1}) \beta_{j,k+1}=(\sigma_1 \cdots \sigma_{j-1}) s_j \beta_{j,k+1} = (\sigma_1 \cdots \sigma_{j-1}) \beta_{j+1, k+1}$, and from there notice that this will be equal to $\beta_{j+1, k+1}$.
\item If $k=n$, then we will find $\tau \beta_{j,n} = - \alpha_j$. We notice that if $j < k$, then $\sigma_k \beta_{j,k}=s_k \beta_{j,k}=\beta_{j,k-1}$. As a result, $(\sigma_{j+1} \cdots \sigma_n) \beta_{j,n} = \beta_{j,j}$. From there, $\sigma_j \beta_{j,j} = s_j \beta_{j,j} = -\alpha_j$, and $\sigma_i (-\alpha_j) = -\alpha_j$ for all $i \ne j$.

\item $\tau (-\alpha_i) = \sigma_1 \cdots \sigma_n (-\alpha_i)$, which will then be equal to $(\sigma_1 \cdots \sigma_{i-1}) \sigma_i (-\alpha_i)$. We then get $\sigma_i (-\alpha_i) = \beta_{i,i}$. Now, we note recursively that $\sigma_{j-1} \beta_{j,i} = \beta_{j-1,i}$, and so $(\sigma_1 \cdots \sigma_{i-1}) \beta_{i,i} = \beta_{1,i}$.
\end{itemize}

We will now define a map $\phi$ taking tubes of the $n$-dimensional path hypercube graph to almost positive roots of $A_n$. For any $1\le i \le j \le n$, define $[i,j]=\{i,i+1,\ldots,j\}$. We define $\phi([i,j])=\beta_{i,j}$, and define $\phi(\{-i\})=-\alpha_i$. This is a bijection taking positive tubes to positive roots and negative singleton tubes to negative simple roots.

Now, we begin to determine compatibility conditions. We note that a negative tube $\{-i\}$ and a positive tube $[j,k]$ are compatible if and only if $i \notin [j,k]$. This is exactly the $||_c$ compatibility relation between a negative simple root $-\alpha_i$ and a positive root $\beta_{j,k}$. Now note that any two negative tubes are compatible, and any two negative simple roots are compatible under $||_c$. Now, we must prove compatibility conditions for two positive tubes $\beta_{i,j}, \beta_{k,l}$. We wish to prove that any two distinct tubes $t_1, t_2$ are compatible if and only if $\phi(t_1) ||_c \phi(t_2)$.

Consider two positive tubes $[i,j], [k,l]$, such that $l \ge j$. We wish to prove these tubes are compatible if and only if $\beta_{i,j} ||_c \beta_{k,l}$.  Note that shifting these tubes rightward by $n-l$ units does not change compatibility, so these tubes are compatible if and only if the tubes $[i+(n-l),j+(n-l)]$ and $[k+(n-l),n]$ are compatible. The image of these tubes are the tubes $\tau^{n-l} \beta_{i,j}, \tau^{n-l} \beta_{k,l}$, which are $||_c$-compatible if and only if $\beta_{i,j}, \beta_{k,l}$ are. As a result, we can focus only on the case of compatibility for tubes $[i,j], [k,n]$.

We note that $[i,j] \subseteq [k,n]$ if and only if $k \le i$, and the two tubes are disjoint and not adjacent if and only if $j < k-1$. As a result, we say that $[i,j], [k,n]$ are incompatible if and only if $i+1 \le k \le j+1$. Furthermore, if $j = n$, then $[i,j], [k,n]$ must be compatible.

Now consider $\tau \beta_{i,j}$ and $\tau \beta_{k,n}$. We note that $\tau \beta_{k,n}=-\alpha_k$. If $j=n$, then $\tau \beta_{i,j}=-\alpha_j$, and the two roots are $||_c$ compatible. If $j \ne n$, then $\tau \beta_{i,j}=\beta_{i+1,j+1}$. We note that $\beta_{i+1,j+1}$ and $-\alpha_k$ are compatible if and only if $k \notin [i+1,j+1]$. As a result, $\beta_{i,j}$ and $\beta_{k,n}$ are compatible if and only if $[i,j]$ and $[k,n]$ are compatible.

As a result, we have proven for every case that the tubes $t_1, t_2$ are compatible in the $n$-dimensional path hypercube graph if and only if $\phi(t_1)||_c \phi(t_2)$.
\end{proof}

\begin{definition}
The \emph{$c$-cluster fan} for a Coxeter element $c$ is the fan containing all cones $\conichull\{\beta_1,\ldots,\beta_k\}$ for every set $\{\beta_1, \ldots, \beta_k\}$ of pairwise $||_c$-compatible almost positive roots.
\end{definition}

\begin{proposition}\label{prop:linear-cluster-fan}
The normal fan of the standard cut hypercube graph associahedron of the $n$-dimensional path hypercube graph associahedron is linearly isomorphic to the linear $c$-cluster fan of type $A_n$.
\end{proposition}
\begin{proof}
Define a linear isomorphism $h$ taking $e_i$ to $\alpha_i$ for each $i \in [n]$. Each facet $\Phi_t$ for a tube $t$ is normal to a vector $v_t$, and we find $h(v_t) = \phi(t)$, where $\phi$ is the map taking the tube $[i,j]$ to root $\beta_{i,j}$ and $\{-i\}$ to $-\alpha_i$. For each tubing $T$, the cone dual to $\Phi_T$ is the conic hull of vectors $v_t$ for all $t \in T$. The image of this cone under the map $h$ is the cone $\conichull\{\phi(t)|t \in T\}$. These are exactly the cones in the linear $c$-cluster fan of type $A_n$.
\end{proof}


We have defined a bijection between sets of compatible almost-positive roots and hypercube graph tubings. We have also established a bijection between hypercube graph tubings of a path on $\pm[n]$ and simplex-graph tubings of a path on $[n+1]$. We will now define a bijection between diagonals of a polygon and almost-positive roots of type $A_n$ which matches with our choice of generators.

Consider a polygon with $n+3$ vertices labeled $\{0,1,\ldots,n+2\}$ in order around the polygon. Define $D_{i,j}$ as the diagonal connecting vertex $i$ to vertex $j$ with $i < j$. A diagonal $D_{i,j}$ exists for all $i,j \in \{0,\ldots,n+2\}$ such that $i +2 \le j$, with the exception that there exists no diagonal $D_{0,n+2}$. We will define a bijection between diagonals and almost positive roots. We define $g(D_{i,n+2})=-\alpha_i$, and $g(D_{i,j})=\alpha_{i+1} + \cdots + \alpha_{j-1}$ when $j \le n+1$. Note that $g$ maps diagonals to positive roots, and is different from $f$, which maps design tubes/hypercube-graph tubes to path-graph tubes.

\begin{proposition} \label{prop:linear-cluster-compatibility}
Two diagonals $D, D'$ are non-crossing if and only if $g(D), g(D')$ are linear-$c$ compatible. 
\end{proposition}

\begin{proof}

Consider two diagonals $D_{i,j}$ and $D_{k,n+2}$. If $j=n+2$, then the two diagonals share a point and are non-crossing, and we note that $g(D_{i,j})=-\alpha_i$ and $g(D_{k,n+2})=-\alpha_k$, so the two roots are linear-$c$ compatible. Consider the case where $j < n+2$. We note that $g(D_{i,j})=\alpha_{i+1}+\cdots+\alpha_{j-1}$. We also note that $g(D_{k,n+2})=-\alpha_k$. These two roots are incompatible if and only if $i+1 \le k \le j-1$. In addition, we note that $D_{k,n+2}$ splits the polygon into two smaller polygons with vertex sets $\{0,1,\ldots, k,n+2\}$ and $\{k,k+1,\ldots,n+2\}$. We note that $D_{i,j}$ crosses $D_{k,n+2}$ if and only if each vertex $i, j$ lies in a different polygon and neither vertex is shared by the two polygons. As a result, we find the two diagonals cross if and only if $i \in \{0,\ldots, k-1\}$ and $j \in \{k+1, \ldots, n+1\}$. This is equivalent to the condition that $i+1 \le k \le j-1$, and so $g(D_{i,j})$ and $g(D_{k,n+2})$ are compatible if and only if $D_{i,j}$ and $D_{k,n+2}$ are non-crossing.

We then consider rotational symmetry. For two diagonals $D_{i,j}, D_{k,l}$ with $j \le l < n+2$, we find the two are compatible if and only if $D_{i+1,j+1}, D_{k+1,l+1}$ are. We however can find that $g(D_{i+1,j+1})=\tau(g(D_{i,j}))$, which preserves compatibility. As a result, we can find that $D_{i,j}, D_{k,l}$ are noncrossing if and only if $D_{i+(n+2-l),j+(n+2-j)}, D_{k+(n+2-l), n+2}$ are, which are noncrossing if and only if $g(D_{i+(n+2-l),j+(n+2-j)}),  g(D_{k+(n+2-l), n+2})$ are compatible, which are compatible if and only if $g(D_{i,j}), g(D_{k,l})$ are.
\end{proof}

Now we have established that there are bijections between partial triangulations of a polygon on $n+2$ vertices, almost positive roots of type $A_n$, single path hypercube-graph tubings on $\pm[n]$ vertices, and tubings of the path graph on $n+1$ vertices. We illustrate these relations in several diagrams. Recall that $g$ is the map taking each diagonal $D$ to an almost positive root, $f$ is the map taking a design path-tube to a simplex-path tube, and $\phi$ is the map taking design path-tubes to almost-positive roots.

First, figure \ref{fig:negative-roots-triangulation} shows the negative simple roots in the $n=4$ case. We note that a diagonal intersects with $D_{i,n+2}$ if and only if the vector $g(D_i,n+2)$ contains $-\alpha_i$ in its support. This agrees in essence with the method of assigning almost-positive roots to diagonals in a polygon used in \cite{ysystems}, except their method chooses a different set of diagonals to be the negative simple roots, forming what is known as a "snake path." Figure \ref{fig:diagonals-2} shows a triangulation consisting of diagonals and the almost-positive root associated with each diagonal, noting that the root labeled $\alpha_1+\alpha_2$ crosses exactly the diagonals labeled $-\alpha_1$ and $-\alpha_2$.

This shows  $g(D)$ for each diagonal in the triangulation. Figure \ref{fig:diagonals-4} shows the hypercube graph tube $\phi^{-1}(g(D))$ associated with each diagonal in the same tubing. If we apply the bijection $f$ used in Proposition \ref{prop:cited-associahedron-result}, we find that Figure \ref{fig:diagonals-3} shows the path tube $f^{-1}(\phi^{-1}(g(D)))$ for each diagonal $D$. Define $\theta(D)=f^{-1}(\phi^{-1}(g(D)))$. Because each map $f, \phi, g$ is a bijection which preserves compatibility or crossing conditions, we find that the following lemma holds.

\begin{lemma} \label{lemma:diagonal-tubes-map}
For a polygon on $n+3$ vertices, two diagonals $D, D'$ cross if and only if the two tubes $\theta(D), \theta(D')$ of the path graph on $[n+1]$ are incompatible.
\end{lemma}


\begin{figure}
\centering
\includegraphics[width=.5\textwidth]{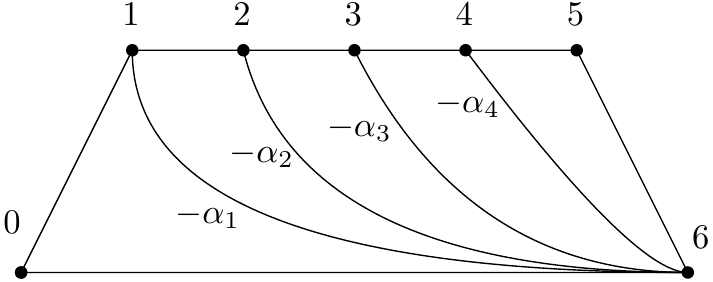}
\caption{Triangulation associated with negative simple roots.} \label{fig:negative-roots-triangulation}
\end{figure}


\begin{figure}
\centering
\includegraphics[width=.5\textwidth]{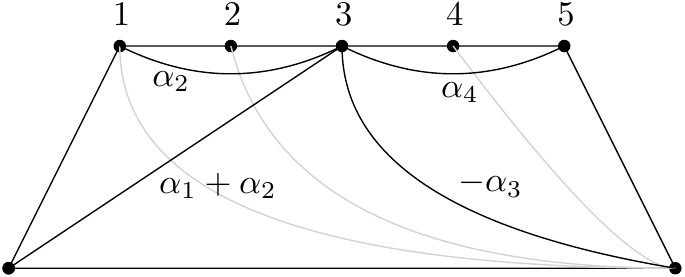}
\caption{Almost positive roots associated with diagonals of a triangulation.} \label{fig:diagonals-2}
\end{figure}

\begin{figure}
\centering
\includegraphics[width=.75\textwidth]{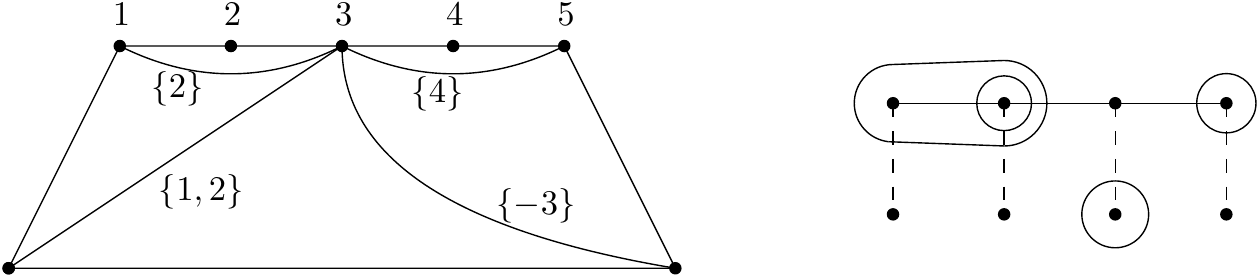}
\caption{Single path hypercube graph tubing associated with diagonals of the same triangulation.} \label{fig:diagonals-4}
\end{figure}

\begin{figure}
\centering
\includegraphics[width=.75\textwidth]{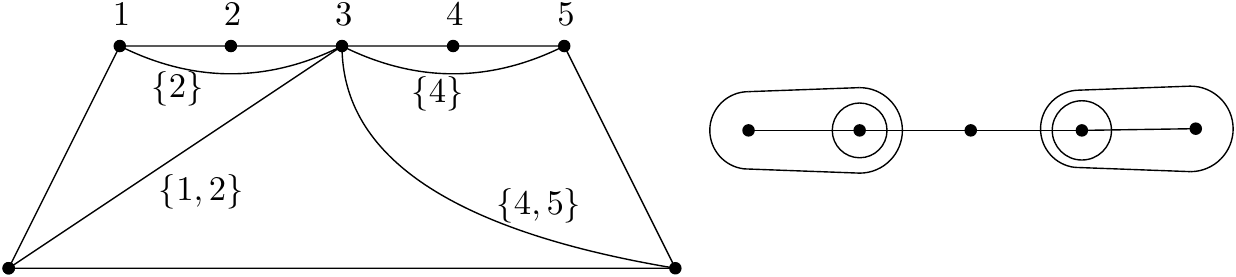}
\caption{Simplex path graph tubing associated with diagonals of the same triangulation.} \label{fig:diagonals-3}
\end{figure}

	\subsubsection{Complete Graph on Positive Vertices and the Stellohedron}
	
	Define the \emph{single $K_n$ hypercube graph} to be the graph on the hypercube consisting of the complete graph $K_n$ on vertices in $[n]$, and isolated vertices for each element in $-[n]$. This graph is $K_n^+$. The \emph{stellohedron} is the simplex graph associahedron of the complete bipartite $K_{n,1}$ graph; this graph associahedron is mentioned in \cite{postnikovfaces}.
	
	\begin{proposition}
		The hypercube graph associahedron for the single $K_n$ graph is isomorphic to the $n$-dimensional stellohedron.
	\end{proposition}
	
	\begin{proof}

		The tubes of $K_n^+$ are either subsets of $[n]$, or negative singleton subsets of $-[n]$. Label $K_{n,1}$ as the graph on $[n+1]$ such that $n+1$ is connected to every other vertex. The tubes of $K_{n,1}$ are either subsets containing the center vertex $n+1$, or they are singleton subsets of $[n]$.


		We define a map from tubes of $K_{n,1}$ to tubes of $K_n^+$. Define $\phi(t)=[n+1] \backslash t$ whenever $t$ contains $n+1$, and define $\phi(t)=-t$ whenever $t$ does not contain $n+1$. This is a bijection between tubes, and we need to prove that it preserves compatibility.

		For two tubes $t_1, t_2$ that contain $n+1$, we find $t_1, t_2$ are compatible if and only if one tube is a subset of the other. We find that this is true if and only if one tube out of $[n+1]\backslash t_1, [n+1]\backslash t_2$ is a subset of the other. As simplex-graph tubes in a clique are compatible if and only if one is contained in the other, this means that $t_1, t_2$ are compatible if and only if $\phi(t_1), \phi(t_2)$ are compatible.

		We find any two tubes $t_1, t_2$ not containing $n+1$ must be compatible, and we find as well that $-t_1, -t_2$ must be compatible. As a result, in this case $t_1, t_2$ are compatible if and only if $\phi(t_1), \phi(t_2)$ are compatible.

		For a tube $t_1$ containing $n+1$ and a singleton tube $t_2$ not containing $n+1$, we find the two are compatible if and only if $t_2 \subset t_1$. This is then only possible if $t_2 \not\subseteq [n+1]\backslash n_1$, which are the conditions that a negative singleton tube $-t_2$ is compatible with $t_1$ in $K_{n,1}$. As a result, we find that $\phi$ preserves compatibility conditions and the nested complexes of $K_{n,1}$ and $K_n^+$ are isomorphic.
	\end{proof}

	Figure \ref{fig:stellahedron-as-hypercube-graph} shows $K_{n,1}$ and $K_n^+$, with vertices arranged radially. Figure \ref{fig:stellahedron-as-hypercube-graph-2} shows two equivalent tubings of these graphs, with each tube $t$ and its image $\phi(t)$ colored the same.

	\begin{figure}
\centering
\includegraphics[width=.75\textwidth]{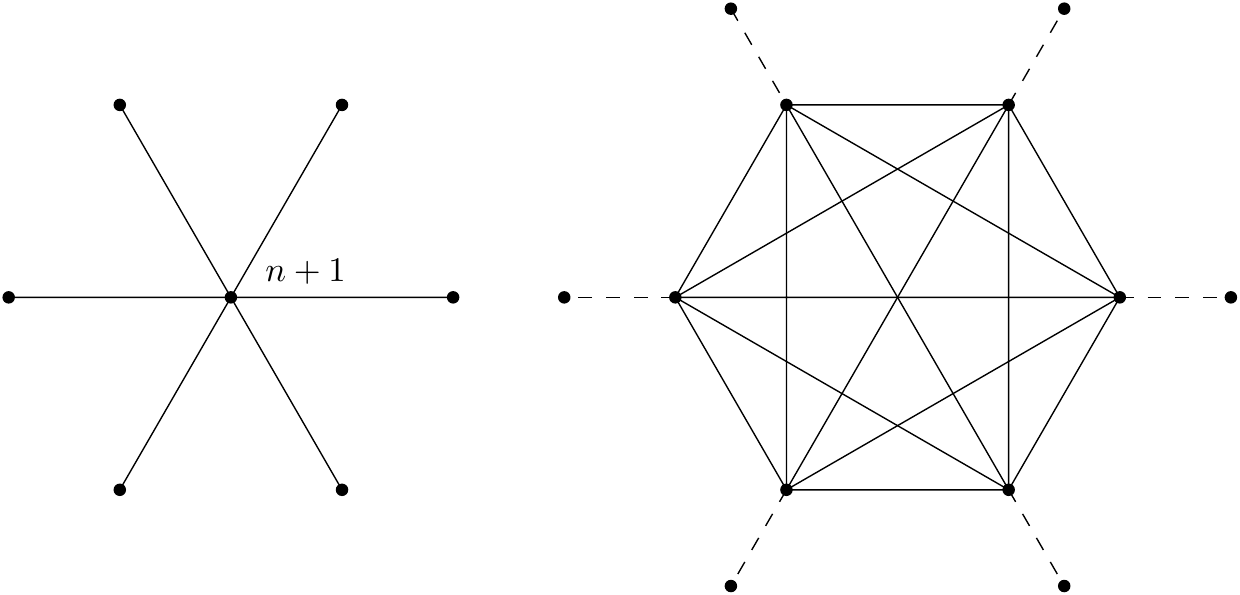}
\caption{The graphs $K_{n,1}$ and $K_n^+$.} \label{fig:stellahedron-as-hypercube-graph}
\end{figure}

	\begin{figure}
\centering
\includegraphics[width=.75\textwidth]{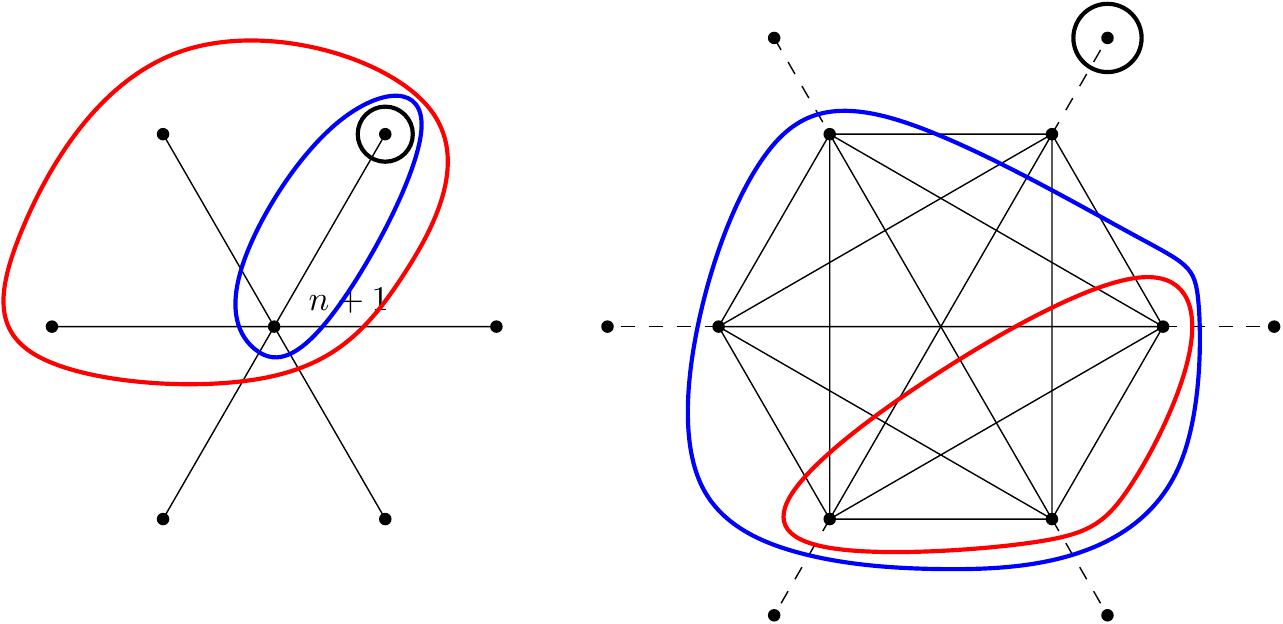}
\caption{Equivalent tubings of the graphs $K_{n,1}$ and $K_n^+$.} \label{fig:stellahedron-as-hypercube-graph-2}
\end{figure}

	\subsubsection{Cycle on Positive Vertices and the Halohedron}

	The single cycle hypercube graph is the graph on $\pm[n]$ with cycle $(1,2,\ldots,n)$. Its tubing complex is isomorphic to the design tubing complex of a cycle on $n+1$ vertices. The cubeahedron of the cycle graph is known as the \emph{halohedron} \cite{devadoss2011}. Up until now, its $f$-polynomial was not known, and not even an enumeration of the number of vertices of the halohedron was known. We provide here an enumeration of the faces of the halohedron:

	\begin{theorem} \label{thm:halohedron-formula}
		The bivariate generating function for the number of faces of dimension $k$ of the $n$-dimensional halohedron is

		\[
			f^H(x,y) = \frac{1+(2+x)y}{2\sqrt{1-2(2+x) y + x^2 y^2}}+\frac{1}{2} .
		\]
	\end{theorem}

		The proof of this theorem is presented in Subsection \ref{sub:halohedron} using maximal tube enumeration.

	When $x=0$, this generating function gives the number of vertices of each $n$-dimensional halohedron.
	
	\begin{proposition}
	The generating function for the vertex count of each $n$-dimensional halohedron is given by the function
	\[
		f^H(0,y) = \frac{1+2y}{2 \sqrt{1-4y}}+\frac{1}{2}
	\]

\end{proposition}

%
%
%
%

	\subsubsection{Stellocubahedra} \danger{stellacubohedron?}

	Define the \emph{single star hypercube graph} as the hypercube graph on $\pm[n]$ with edges $(1,i)$ for all $i=2,\ldots,n$. This is the graph $K_{n,1}^+$. The single star hypercube graph associahedron is isomorphic to the design graph associahedron of the bipartite $K_{n,1}$ graph. Call this polytope the \emph{stellocubahedron}.

	\begin{proposition}
		Facets of the stellocubahedron are either lower-dimensional stellocubeahera, lower-dimensional hypercubes, or products of stellohedra.
	\end{proposition}

	\begin{proof}
		There are three cases of tubes in the stellar hypercube graph associahedron: tubes containing $1$, the tube $\{-1\}$, and singleton tubes not containing $1$ or $-1$. The reconnected complement of a tube containing $1$ is the hypercube graph with a complete graph on positive vertices, whose hypercube-graph associahedron is isomorphic to a stellohedron, and so the facet associated with such a tube is combinatorially isomorphic to the product of two stellohedra. The reconnected complement of the tube $\{-1\}$ is an empty hypercube graph, and so the associated facet is isomorphic to a hypercube. The reconnected complement of a singleton tube not containing $1$ or $-1$ is a single star hypercube graph, and the associated facet is a stellocubeahedron.
	\end{proof}

	We do not have a formula for the $f$-polynomial of the stellocubeahedron. However, we do have a formula counting its vertices.

	\begin{proposition}
		The $n$-dimensional stellocubeahedron has
		\[
			2^{n-1} + \sum_{0 \le j \le k \le n-1} \frac{(n-1)!}{j!(n-1-k)!}
		\]
		vertices for all $n \ge 1$.
\end{proposition}
\begin{proof}
The maximal tubings of this graph can be counted as follows: either a tubing contains the tube $\{-1\}$, or it contains a maximal tube of size $k+1$ containing~$1$. This tube will induce a star graph on $k+1$ vertices, and will have $\sum_{i=0}^{k} \frac{k!}{i!}$ maximal tubings. There are $${n-1 \choose k}$$ possible tubes of size $k+1$ containing $1$, giving us the count
	\[
		sc(n) = 2^{n-1} + \sum_{k=0}^{n-1} {n-1 \choose k} \sum_{j=0}^{k} \frac{k!}{j!}
	\]
	maximal tubings for $n>0$, and $sc(0)=1$.
\end{proof}

	\begin{remark}
		As written now, the maximal tubing subcomplex enumeration method is not sufficient to calculate the $f$-polynomial of the stellocubeahedron. Methods such as atomic link sum enumeration and maximal tube sub-complex enumeration require the counting functions form $r_{\sigma,\tau}(n,i)$ to be written as finite sums of separable functions in $i$ and $n-1-i$. However, the number of tubes containing $1$ and $i$ other vertices is ${n-1 \choose i}$. This is written as $r_\sigma(n,i)=\frac{(n-1)!}{i! (n-1-i)!}$, which cannot be written as a finite sum of separable functions. We believe that a modification to the technique using mixed ordinary and exponential generating functions would be more capable of performing these operations.
	\end{remark}

	\subsection{Double Graphs on Positive and Negative Vertices (Double Cubeahedra)} \label{sub:doublecubeahedra}

	 Given a graph $G$ on $[n]$, the \emph{double graph} $2G$ is the hypercube graph on $\pm[n]$ that has edges $\{i,j\}$ and $\{-i,-j\}$ if $\{i,j\}$ is an edge of $G$, and no other edges. Tubes of double graphs can be expressed as \emph{signed tubes} and \emph{signed tubings}.

	\begin{definition}
		For a graph $G$, a \emph{signed tube} of $G$ is a subset of vertices of $G$ which induces a connected subgraph, and which is labeled with a sign, either positive or negative.
	\end{definition}

	\begin{definition}
		Two signed tubes of $G$ are \emph{compatible} if
		\begin{enumerate}
			\item They agree in sign, and either: one is contained in the other, or they are disjoint but not adjacent.
			\item They disagree in sign and are disjoint.
		\end{enumerate}
	\end{definition}

	A \emph{signed tubing} is a collection of pairwise compatible signed tubes. Signed tubes and tubings of a graph $G$ are in bijection with tubes and tubings of $2G$, with positive tubes corresponding to tubes on positive vertices and negative tubes corresponding to tubes on negative vertices. This bijection gives us the following proposition:

	\begin{proposition}
		The complex of signed tubings of a graph $G$ is isomorphic to the nested complex of the hypercube-graph $2G$.
\end{proposition}

	Figure \ref{fig:signed-tubing} shows a tubing of a double path graph and an equivalent signed tubing. 

	\begin{figure}[h]
	\begin{center}
	\includegraphics[width=.75\textwidth]{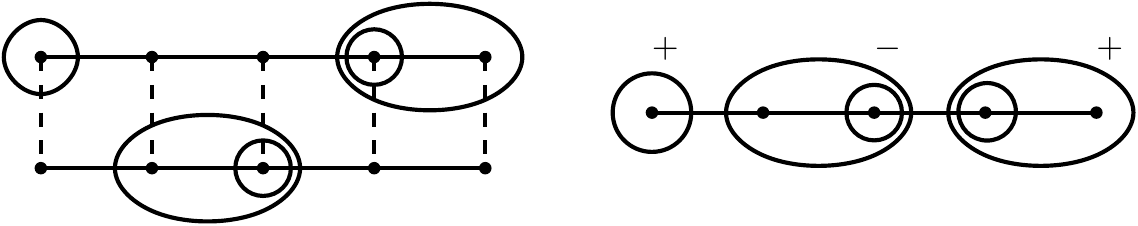}
	\end{center}
	\caption{A double path graph tubing and the corresponding signed tubing.} \label{fig:signed-tubing}
\end{figure}

	\begin{proposition}	\label{prop:double-graph-minkowski-sum}
		The hypercube graph associahedron of a double graph $2G$ can be realized as the Minkowski sum of the hypercube graph associahedron of $G^+$ and its antipodal inverse.
	\end{proposition}

	\begin{proof}
		Define $G^-$ as the hypercube graph containing edges $\{-i,-j\}$ for each edge $\{i,j\}$ in $G$, and no other edges. We note that the set of tubes of $2G$ is equal to the union of the set of tubes of $G^+$, and the set of tubes of $G^-$. As a result, the graphic building set of $2G$ is equal to the union of graphic building sets of $G^+$ and $G^-$. As a result of Corollary \ref{cor:union}, the normal fan of a standard cut $2G$ hypercube graph associahedron is equal to the coarsest common refinement of the normal fans of standard cut hypercube graph associahedra of $G^+$ and $G^-$. We also note that the antipodal inverse of a standard cut hypercube graph associahedron of $G^+$ is a standard cut hypercube graph associahedron of $G^-$. As a result, we find we can realize the hypercube graph associahedron of $2G$ as the Minkowski sum of the hypercube graph associahedron of $G^+$ and its antipodal inverse.
	\end{proof}

	\subsubsection{\texorpdfstring{$2K_n$ Graphs and the type $A_n$ permutahedron}{2Kn Graphs and the type An permutahedron}}
	
	The \emph{$2K_n$ hypercube graph} is the hypercube graph on $\pm[n]$ consisting of a complete graph on $[n]$ and a complete graph on $[-n]$. This is a special case of the double graph $2G$ we have just described. The type $A_n$ permutahedron is the orbit of a generic point under the type $A_n$ reflection group, and is isomorphic to the simplex graph associahedron of the $K_{n+1}$ graph.
	
	\begin{proposition} \label{prop:type-A-permutahedron}
		The hypercube graph associahedron of the $2K_n$ graph is combinatorially isomorphic to the type $A_n$ permutahedron.
	\end{proposition}
	
	\begin{proof}
		Consider the set of tubes in the graph $2K_n$; these consist of all subsets of positive vertices $[n]$, and all subsets of negative vertices $-[n]$. The tubes of the $K_{n+1}$ simplex-graph are precisely the set of proper subsets of $[n+1]$. We define a map $\phi$ on tubes of $2K_n$ as follows: for any tube $t \subseteq [n]$, define $\phi(t)=t$, and for any tube $t \subseteq -[n]$, define $\phi(t)=[n+1] \backslash -t$, where $-t=\{-i:i \in t\}$. The map $\phi$ is a bijection between tubes of $2K_n$ and tubes of the simplex-graph $K_{n+1}$. We then must prove that $\phi$ preserves compatibility conditions.

		If $t_1, t_2$ are both tubes on positive vertices in $2K_n$, then we note that they are compatible in $2K_n$ if and only if they are compatible in $K_{n+1}$. We note that a tube $t_1 \subseteq t_2$ if and only if $[n+1] \backslash -t_2$ is a subset of $[n+1]\backslash -t_1$, and vice versa. We note that for negative tubes in $2K_n$ and any tubes in $K_{n+1}$, two tubes are compatible if and only if one is contained in the other. As a result, $t_1, t_2$ are compatible if and only if $\phi(t_1), \phi(t_2)$ are compatible.

		Now consider the case where $t_1$ is a tube on positive vertices, and $t_2$ is a tube on negative vertices. We note that $t_1, t_2$ are compatible if and only if $t_1, -t_2$ are disjoint. This is then true if and only if $t_1 \subseteq [n+1] \backslash -t_2$. Now we note that $\phi(t_2)$ contains $n+1$, and $\phi(t_1)$ does not, so $\phi(t_1), \phi(t_2)$ are compatible if and only if $\phi(t_2)$ contains $\phi(t_1)$. As a result of this and the other cases, $t_1, t_2$ in $2K_n$ are compatible if and only if $\phi(t_1), \phi(t_2)$ are. As a result, $\phi$ preserves compatibility conditions.
		\end{proof}

	
		As an aside, we point out the connection between this hypercube-graph associahedron, and the \emph{graph multiplihedron}, as defined in \cite[Theorem 17]{multiplihedron}. The realization of the complete graph multiplihedron can be defined by truncation of a hypercube in that paper, and the reader will find that the complete graph multiplihedron is a standard cut hypercube graph associahedron of $2K_n$.



%
	
	\subsubsection{Double Path graph and Coxeter Bi-Catalan Combinatorics}
	
	Define the \emph{double path hypercube graph} to be the hypercube graph on $\pm[n]$ which consists of $(1,\ldots,n)$ and $(-1,\ldots,-n)$. This hypercube graph associahedron has connections to linear $c$-cluster fans, like the single path hypercube graph. Given a type $W$ Coxeter group and a Coxeter element $c$, there exists a $c$-cluster fan. We consider a definition from \cite{barnard}:

	\begin{definition}
		The \emph{$c$-bicluster fan} of a Coxeter element $c$ of a Coxeter group $W$ is a fan equal to the coarsest common refinement of a $c$-cluster fan and its antipodal inverse. 
\end{definition}

	We recall from Proposition \ref{prop:double-graph-minkowski-sum} that the standard cut normal fan of $2G$ can be realized as the coarsest common refinement of the normal fan of a standard cut hypercube graph associahedron of $G^+$ and its antipodal inverse. When $G$ is a path graph, we recall from Proposition \ref{prop:linear-cluster-fan} that the normal fan of an $n$-dimensional standard cut hypercube path graph associahedron is a linear $c$-cluster fan of type $A_n$. These two lead to the following proposition:

	\begin{proposition}
		The normal fan of a standard cut $n$-dimensional double path hypercube graph associahedron is equal to the linear $c$-bicluster fan of type $A_n$.
	\end{proposition}

	Theorem 2.20 of \cite{barnard} makes statements about the enumeration of cones in the bipartite $c$-bicluster fan, and Remark 2.14 of that paper makes explicit that the type $A_n$-biassociahedron, which is dual to the type $A_n$ bipartite $c$-bicluster fan, has the same $f$-vector as the $n$-dimensional cyclohedron, although the two are not combinatorially isomorphic. However, the paper does not attempt to provide enumeration for the linear $c$-bicluster fan. We find that the enumeration is the same as follows:

	\begin{proposition}\label{prop:doublepath}
		The $f$-vector of the $n$-dimensional double path hypercube graph associahedron is equal to the $f$-vector of the $n$-dimensional cyclohedron.
	\end{proposition}

		Subsection \ref{sub:doublepath} details the method by which we find the bivariate $f$-polynomial of the double path hypercube graph associahedron. While the two polyhedra share the same $f$-vectors, they are not combinatorially isomorphic, which becomes clear in $\ge 3$ dimensions. Figure \ref{fig:1-a3_biassociahedron-again} shows the double path hypercube-graph associahedron in 3 dimensions.

	\begin{figure}[h!]
\centering
\includegraphics[width=.3\linewidth]{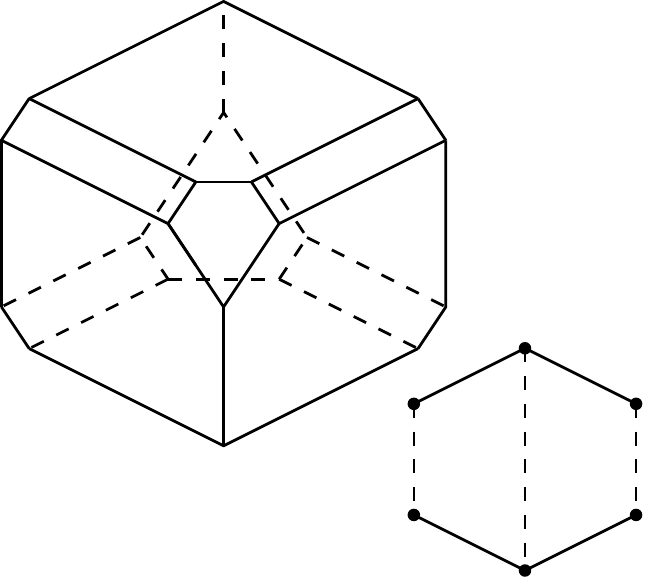}
\caption{Type $A_3$ linear biassociahedron as a hypercube-graph associahedron.} \label{fig:1-a3_biassociahedron-again}
\end{figure}

	In addition to the interpretation of the double path hypercube graph associahedron in terms of Coxeter combinatorics, we can realize the tubing complex of this graph via a construction we call \emph{linear triangulations with no trapped vertices.}

	Consider a line with $n+2$ labeled points, labeled $0,1,\ldots,n+1$	sequentially. We define a set of arcs to be the set of curves passing either over or under the line. We write $D_{i,j}^+$ to be the arc connecting point $i$ to point $j$ with $i < j$ passing over the line, and $D_{i,j}^-$ connecting $i < j$ passing under the line.

	Define a map $h$ from arcs to signed tubes in the path graph on $n$ vertices, such that $h(D_{i,j}^\pm) = [i+1,j-1]^\pm$. We say that the vertices in $h(D^+)$ are under a positive arc $D^+$, and the vertices in $h(D^-)$ are over a negative arc $D^-$. We say a vertex is \emph{trapped} if it is both under and over an arc.

	\begin{proposition}
		The collection of sets of noncrossing arcs with no trapped vertices on $n+2$ vertices is isomorphic to the double path hypercube graph tubing complex.
	\end{proposition}

	\begin{proof}
		We note that the map $h(D^\pm_{i,j})$ is the same as the map used in Proposition \ref{lemma:diagonal-tubes-map}, except here we map the diagonal $D^+_{i,j}$ to a positive tube and $D^-_{i,j}$ to a negative tube, whereas the map in that proposition maps $D_{i,j}$ to a tube in a path graph. Furthermore, we can define a map taking $(n-1)+3$ vertices of a polygon to the $n+2$ labeled points on our line, essentially 'unfolding' the polygon. As a result, we find that two arcs $D, D'$ matching in sign are compatible if and only if the two signed tubes $h(D), h(D')$ are compatible on the positive or negative path on $n$ vertices.

		Now consider two arcs with opposite signs. We find that the vertex $i$ is both under and over a pair of arcs if it is contained in $h(D^+)$ and $h(D^-)$ for some pair of arcs with opposite signs. As a result, for any set $C$ of arcs, we find that two arcs $D, D'$ in $C$ do not cross, and trap no vertex, if and only if the tubes $h(D), h(D')$ are compatible. As a result, $C$ is a set of noncrossing arcs trapping no vertex if and only if $\{h(D)|D \in C\}$ is a signed path tubing complex on $[n]$. 
	\end{proof}

	Figure \ref{fig:double-path-bijection} shows an example of a collection of noncrossing arcs with no trapped vertices and an equivalent double path hypercube graph tubing. 

\begin{figure}
\centering
\includegraphics[width=.8\textwidth]{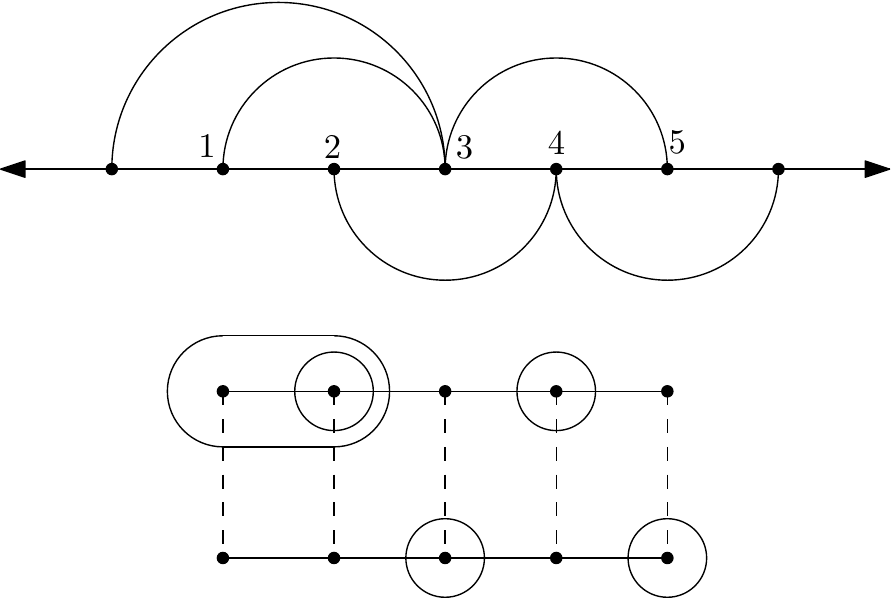}
\caption{A set of noncrossing arcs with no trapped vertices for $n=5$, and an equivalent double path hypercube graph tubing.}\label{fig:double-path-bijection}
\end{figure}

There is a standard Catalan recurrence between rooted binary trees and triangulations of a polygon, made by drawing an internal node in the center of every triangle, and drawing a leaf node or a root node on every edge of the polygon. Figure \ref{fig:double-path-bijection-2} shows an interesting variation of this, drawn on a linear triangulation with no trapped vertices.

\begin{figure}
\centering
\includegraphics[width=.8\textwidth]{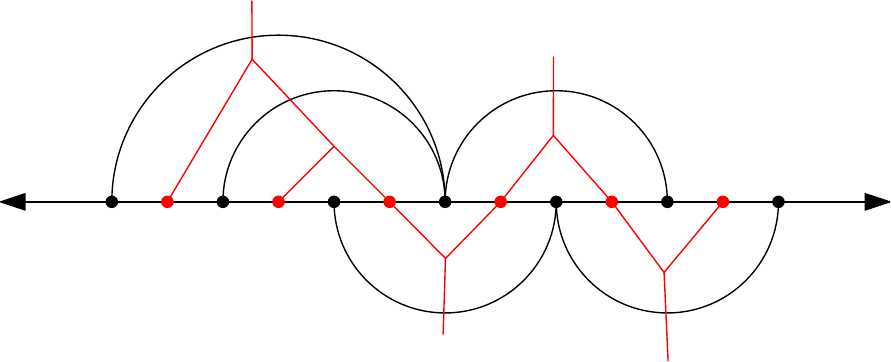}
\caption{A set of noncrossing arcs with no trapped vertices for $n=5$, with positive and negative rooted binary trees drawn.}\label{fig:double-path-bijection-2}
\end{figure}

	\subsubsection{Double Cycle Graph}

	Define the \emph{double cycle hypercube graph} as the double graph of the cycle on $[n]$ consisting of cycles $(1,\ldots,n)$ and $(-1,\ldots,-n)$. We find the following:

	\begin{proposition}\label{prop:doublecycle}
		The bivariate $f$-polynomial of the family of double cycle hypercube graph associahedra is equal to

		\[
			f^{DC}(x,y) = (xy+2y) f^B + 2y\left[(1+yD_y)f^A\right] (f^C-1)
		\]
where $f^A$ is the bivariate $f$-polynomial of the family of associahedra, $f^B$ is the bivariate $f$-polynomial of the family of cyclohedra, and $f^C$ is the bivariate $f$-polynomial of the cis-double path graph described in Subsection \ref{sub:cistransgraphs}. This function restricted to $x=0$ gives the generating function for the number of vertices of each polytope:

		\[
			f^{DC}(0,y) = 1 + \frac{1-\sqrt{1-4y}-2y}{1-4y} + \frac{2y}{\sqrt{1-4y}}.
		\]
	\end{proposition}

		Subsection \ref{sub:doublecycle} details the process by which the case by case method is used to count the tubings of $DC_n$.

%
%
%
%

	\subsubsection{Double Star Graph}

	Define the \emph{double star hypercube graph} as the double graph of the graph $K_{n-1,1}$. We can write it as the hypercube graph on vertices $\pm[n]$ with edges $\{1,i\}$ and $\{-1,-i\}$ for each $i \in \{2,\ldots,n\}$. We call its hypercube graph associahedron the \emph{double stellar cubeahedron}.

	\begin{proposition}
		Every facet of an $n$-dimensional double stellar cubeahedron is combinatorially isomorphic to the product of two stellohedra, or an $n-1$-dimensional double stellar cubeahedron.
\end{proposition}

	\begin{proof}
		Tubes of an $n$-dimensional double stellar cubeahedron either contain $1$ or $-1$, or they are singleton tubes not containing $1$ or $-1$. If a tube contains $1$ or $-1$, then its induced graph is a star graph, and its reconnected complement is a single complete-graph hypercube graph. If a tube is a singleton tube not containing $1$ or $-1$, its reconnected complement is an $n-1$-dimensional double star hypercube graph. Because the hypercube graph of a single complete-graph hypercube graph is a stellohedron, we find that each facet is isomorphic to the product of two stellohedra, or an $n-1$-dimensional double stellar cubeahedron.
	\end{proof}


	We can count the number of maximal tubings directly as follows:

	\begin{proposition}
		Vertices of the $n$-dimensional double stellar cubeahedron are counted by the expression
	\[
dsc(n) = 2\sum_{k=0}^{n-1} {n-1 \choose k} \left( \sum_{j=0}^{k} \frac{k!}{j!}\right) = 2 (n-1)!\sum_{i=0}^{n-1} \frac{2^i}{i!}
.
\]

\end{proposition}

\begin{proof}

Consider a maximal tubing. It either contains vertex $1$ or $-1$ in its support. The number of maximal tubings is equal to twice the number of maximal tubings containing $1$ in their supports, so by symmetry we can count the number of maximal tubings containing the vertex $1$ and multiply by $2$.

Consider a maximal tubing $T$ containing $1$ in its support. There exists a maximal tube $t$ of $T$ containing $1$, and all other maximal tubes will be negative singleton tubes. The remaining tubes will be tubings on the graph induced by $t$, which will be a star graph. To enumerate over the set of maximal tubings containing $1$ in their support, we need to enumerate over the set of all maximal tubes containing $1$, and then count the number of star graph tubings for each such tube.

For $0 \le k \le n-1$, there are ${n-1 \choose k}$ tubes of size $k+1$ containing $1$. Each such tube induces a star graph on $k+1$ vertices, and there are $\sum_{j=0}^{k} \frac{k!}{j!}$ maximal tubings on a star graph on $k+1$ vertices. As a result, we find the number of maximal tubings containing $1$ in their supports is equal to
\[
\frac{1}{2}dsc(n) = \sum_{k=0}^{n-1} {n-1 \choose k} \left( \sum_{j=0}^{k} \frac{k!}{j!}\right).
\]
Multiplying this by $2$ gives the total number of maximal tubings. We can rewrite this as
\[
dsc(n) = 2 (n-1)! \sum_{k=0}^{n-1} \sum_{j=0}^k \frac{1}{j!(n-1-k)!}.
\]
We note that this indexing gives all values of $j, k$ such that $j+(n-1-k) \le n-1$. We can then reindex with $(n-1-k)=l$, to get
\[
dsc(n) = 2(n-1)!\sum_{i=0}^{n-1} \sum_{j+l= i} \frac{1}{j! l!}.
\]
We note that $\sum_{j+l=i}\frac{1}{j!l!} = 2^i/i!$. This means we have calculated
\[
dsc(n) = 2 (n-1)!\sum_{i=0}^{n-1} \frac{2^i}{i!}.\qedhere
\]
\end{proof}

This is the same sequence as \cite[Sequence A195254]{oeis}.

%
%
%
%
%

	\subsection{Twisted Path and Twisted Cycle Graphs} \label{sub:twisted cycle and path}
	
	
	Define the \emph{twisted path graph}, or $TP_n$, as a hypercube graph on $\pm[n]$ consisting of a path graph along vertices $(1,\ldots, n, \allowbreak -1, \ldots, -n)$.

	\begin{proposition}\label{prop:twistedpath}
		The bivariate $f$-polynomial of the family of twisted path hypercube graph associahedra is equal to

		\[
			f^{TP}(x,y) =
		\left(	1 + \frac{xy}{(1-xy)(1+yf^A)-2yf^A} + yf^A(f^C-1) \right)
 \left(1+y\left[(2+D_y)f^A\right]\right)
		\]
		where $f^A$ is the bivariate $f$-polynomial for the family of associahedra, and $f^C$ is the bivariate $f$-polynomial for the family of cyclohedra. The generating function for the number of vertices in the $n$-dimensional twisted path hypercube graph associahedron is
		\[
			f^{TP}(0,y) = \frac{1-2y+\sqrt{1-4y}}{2(1-4y)}
		\]
	\end{proposition}
		Subsection \ref{sub:twistedpath} details the case by case method for calculating this bivariate $f$-polynomial, done by calculating the $f$-polynomials of several intermediate $\p$-graph associahedra.

	Define the \emph{twisted cycle graph} or $TC_n$ as the hypercube graph consisting of a cycle on vertices $(1,\ldots, n, -1, \ldots, -n)$.

	\begin{proposition}\label{prop:twisetdcycle}
		The bivariate $f$-polynomial of the twisted cycle graph is

		\[
			f^{TC}(x,y)=\frac{(1-xy)-2y}{(1-xy)^2-4y}.
		\]
	\end{proposition}
		By calculating $f$-vectors of lower-dimensional example cases, we were able to find a related integer series on OEIS \cite[Sequence A127674]{oeis}, the even rows of nonzero coefficients of Chebyshev polynomials. By reversing rows in this triangular series and making every entry positive, we were able to find a hypothesis bivariate $f$-polynomial. In Subsection \ref{sub:twistedcycle}, we use the facet sum polynomial method to define a partial differential equation
		\[
		(yD_y-xD_x)f^{TC}(x,y) = 2y \left[ f^A f^{TP} + (yD_yf^A)f^{TP}+f^A (yD_yf^{TP})\right].
		\]
		satisfied by $f^{TC}$, and prove Proposition \ref{prop:twisetdcycle} by verifying that the function satisfies the PDE.

	\begin{proposition}
	When $n \ge 0$, the number of maximal tubings of the $n$-dimensional twisted cycle hypercube graph is $2^{2n-1}$.
\end{proposition}

	\begin{proof}
		The number of vertices of each hypercube graph associahedron is counted by the generating function $f^{TC}(0,y)=\frac{1-2y}{1-4y} = \frac{1}{2} + \frac{1}{2} \frac{1}{1-4y}$.
	\end{proof}

	\subsection{Omni-graph, generating convex hull of \texorpdfstring{$2^n$}{2 to the n} copies of a graph associahedron}
	
	Consider a graph $G$ on $[n]$. Define the \emph{omni-graph} $\Omega G$ as the hypercube graph containing edges $\{i,j\}, \{i,-j\}, \{-i,j\}, \{-i,-j\}$ for each edge $\{i,j\}$ in $G$. We can consider tubings of the omni-graph $\Omega G$ as being similar to signed tubes of $G$, except that a sign is given to each vertex within a tube. Figure \ref{fig:omnigraph} shows one such tubing of a path graph. We call these tubes \emph{vertex-signed tubes}.

\begin{figure}
\centering
\includegraphics[width=\textwidth]{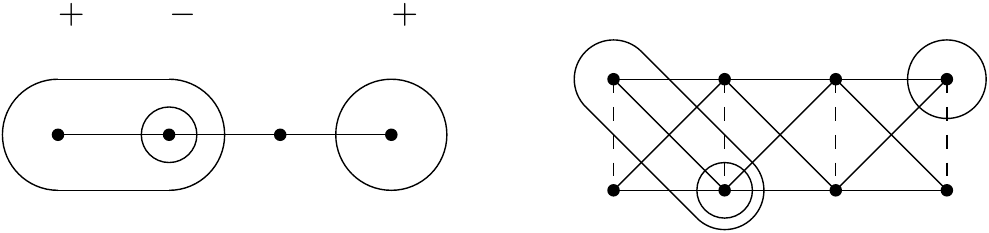}
\caption{Omni-graph tubing of a path graph.}\label{fig:omnigraph}
\end{figure}

\begin{proposition}
For a graph $G$ on $n$ vertices, the omni-graph $\Omega G$ has dual simplicial complex with $f$-polynomial
\[
	f_i^{\Delta(\Omega G)}= 2^n f_{i-1}^{\Delta(G)} + \sum_{T \in \mathcal{T}(G,i)} 2^{\left|\bigcup T\right|}.
\]
for every $1 \le i \le n$, where $\mathcal{T}(G,i)$ is the set of all tubings of the simplex-graph $G$ containing $i$ tubes.
\end{proposition}
\begin{proof}
If $T \in \mathcal{T}(G,i)$ is a simplex-graph tubing, then there are $2^{|\bigcup T|}$ vertex-signed tubing copies of $T$ in the omni-graph of $G$, recalling that $|\bigcup T|$ is the union of all tubes in $T$. In addition, there exist $2^n$ vertex-signed copies of the tube $[n]$, which is not a simplex-graph tube. As a result, the set of tubings of $\Omega G$ containing $i$ tubes consists of vertex-signed tubings of $G$ containing $i$ tubes, and tubings containing $i-1$ tubes of $G$ and one vertex-signed copy of $[n]$.
\end{proof}


There is no formula relating the number of tubes in a tubing, and the number of vertices in $\bigcup T$. However, we note that when $i=n$, the set $\mathcal{T}(G,i)$ is empty, allowing us to write a simpler formula.

		\begin{proposition}
		If a graph $G$ on $[n]$ has $k$ maximal tubings, the omni-graph $\Omega G$ has $2^n k$ maximal tubings.
	\end{proposition}

	Figure \ref{fig:omni-path} shows the omni-graph associahedron of a path graph on $3$ vertices. This realization is equal to the convex hull of $8$ copies of the graph associahedron of a path on $3$ vertices, which form $2$-dimensional faces. In general, for a connected graph $G$ on $n$ vertices, we find that there exist $2^n$ maximal tubes, each containing exactly one member of each set $\{i,-i\}$ for $i \in [n]$. We find that each maximal tubing must contain exactly one of these maximal tubes. Now dually, we find that every vertex of the hypercube-graph associahedron of $\Omega G$ is contained in exactly one facet $\Phi_t$ of a maximal tube $t$ of $\Omega G$, and so these $2^n$ facets partition the vertex set of the hypercube graph associahedron of $\Omega G$.


	\begin{figure}
\centering
\includegraphics[width=.5\textwidth]{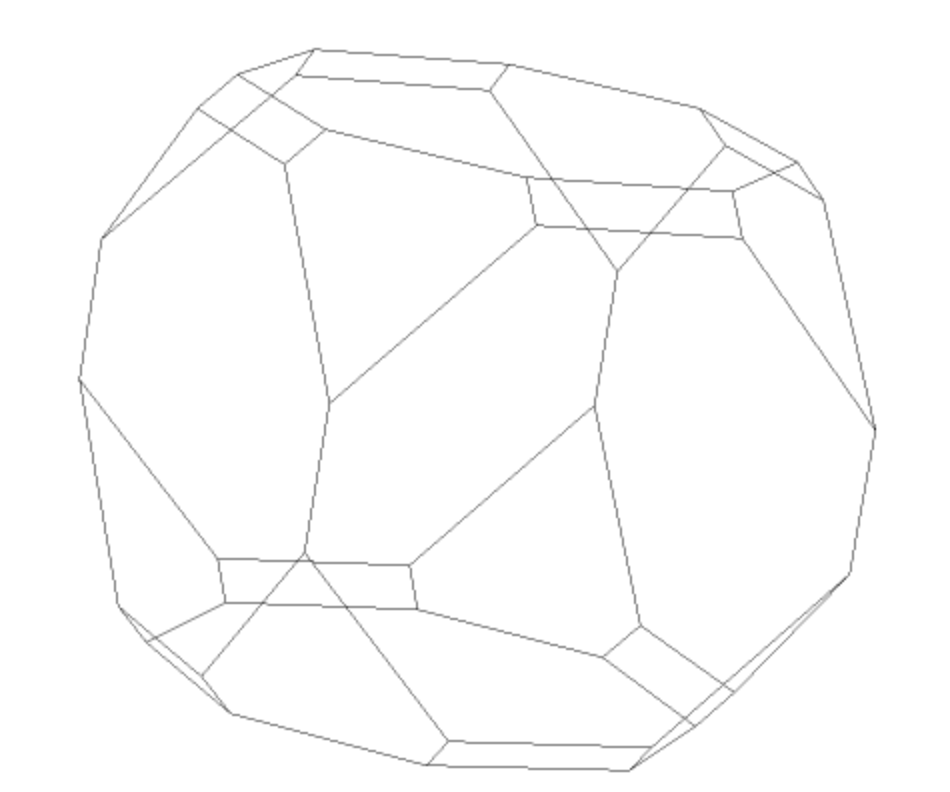}
\caption{Omni-graph associahedron of a path graph.} \label{fig:omni-path}
\end{figure}

	\subsection{Pell Numbers and Companion Pell Numbers} \label{sub:pell-graphs}
	
	Define the graph $G_n$ on $\pm[n]$ containing edges $\{i,-(i+1)\}$ for $1 \le i < n$. Define the graph $H_n$ as the graph $G_n$ with added edge $\{-1,n\}$. Figure \ref{fig:pell-graphs} shows the two graphs. Call $G_n$ the \emph{Pell hypercube graph}, and call $H_n$ the \emph{companion Pell hypercube graph}.

	\begin{figure}[h]
		\centering
		\includegraphics[width=.8\textwidth]{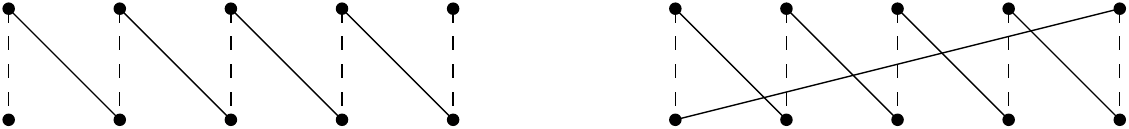}
		\caption{$G_5$ on left and $H_5$ on right.} \label{fig:pell-graphs}
	\end{figure}

	\begin{definition}
		The Pell numbers are a sequence defined by the recurrence $a(n) = 2 a(n-1) + a(n-2)$ with $a(0)=1, a(1)=2$. The companion Pell numbers are given by $b(n)=2(a(n+1)-a(n))$.
	\end{definition}
	
	\begin{theorem}
		The number of maximal tubings of the $G_n$ graph is the $n$th Pell number $a(n)$. The number of maximal tubings of the $H_n$ graph is the $(n-1)$th companion Pell number $b(n-1)$.
	\end{theorem}

	\begin{proof}

		We can count the number of maximal tubings here without going through the entire process of calculating the entire $f$-polynomials of either hypercube graph. We wish to define three tubings $T_1, T_2, T_3$ of $G_n$ such that any maximal tubing of $G_n$ must contain exactly one of them as a subset.

		Define $T_1$ to be the tubing containing only the tube $\{-1\}$ and define $T_2$ to be the tubing containing only the tube $\{1\}$. No tubing can contain both tubes. Now, if $T_1$ and $T_2$ are not in a maximal tubing $T$, then $T$ must contain some tube covering the vertex $1$ or $-1$, and so $T$ must contain $\{1,-2\}$. Now, note that if $T$ is maximal and contains $\{1,-2\}$, then $T$ must contain $\{1\}$ or $\{-2\}$, but because $T_2 \not\subseteq T$, we find $T$ must contain $\{-2\}$. As a result, define $T_3$ to be the tubing containing $\{1,-2\}$ and $\{-2\}$. Figure \ref{fig:pell-graph-tubings} shows these three tubings. The link of each tubing can be done by finding reconnected complements; the reconnected complements of $T_1$ and $T_2$ are isomorphic to the graph $G_{n-1}$, and the reconnected complement of $\{1,-2\}$ is isomorphic to $G_{n-2}$. As a result, if $g(n)$ is the number of maximal tubes of $G_n$, then $g(n)=2g(n-1)+g(n-2)$ for each $n \ge 2$. With $g(0)=1, g(1)=2$, we find $g(n)=a(n)$.

		Now consider the graph $H_n$. We will be using an altered version of the algorithm defined in Section \ref{sec:method-one} to count only maximal tubings. If $\mathcal{N}(H_n,Q_n)$ is the tubing complex of $H_n$, where $Q_n$ is the $n$-dimensional hypercube, then the sum calculated by adding the number of maximal tubings in each atomic link of $\mathcal{N}(H_n,Q_n)$ will be equal to $n$ times the number of maximal tubings of $H_n$. If $h(n)$ is the number of maximal tubes of $H_n$, then we find the atomic link maximal tubing sum is equal to $n h(n)$. We now characterize tubes of $H_n$. Every singleton tube has an atomic link isomorphic to its reconnected complement, which in this case is the $G_{n-1}$ graph, giving $g(n-1)$ maximal tubings. Every edge tube has an atomic link isomorphic to the product of its reconnected complement $G_{n-2}$ and an induced simplex-graph tubing complex of a $1$-simplex, which has $2$ possible maximal tubings. As a result, we find $2g(n-2)$ maximal elements in the atomic link of an edge tube. Because there are $n$ edges and $2n$ vertices, we find that $n h(n) = 2n g(n-1) + 2n g(n-2)$, or $h(n)=2g(n-1)+2g(n-2)$. Considering that $2g(n)-4g(n-1)-2g(n-2)=0$, we can rewrite $h(n)=2(g(n)-g(n-1))=2(a(n)-a(n-1))=b(n-1)$.
	\begin{figure}
\centering
\includegraphics[width=.5\textwidth]{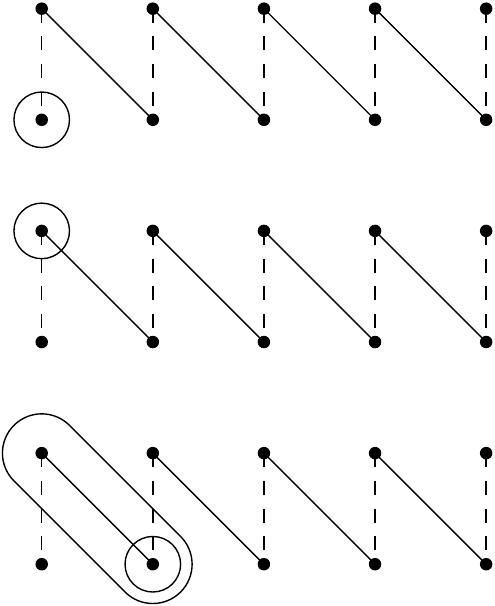}
\caption{Pell graph tubings used in counting.}\label{fig:pell-graph-tubings}
\end{figure}
		\end{proof}
	
	In \cite{sashes}, there is a lattice $\Sigma_n$ of sashes which correspond to the weak order on Pell permutations. We have previously written a program which realizes the Pell hypercube-graph associahedron for dimensions $n \le 7$, and compared the 1-skeleton of the polyhedron to the lattice of sashes, using SageMath. Some code for this appears in Appendix \ref{appendix:code}.

The following conjectures are true for $n\le 7$ dimensions, calculated via electronic computation.
	
	\begin{conjecture}\label{conj:pell-1}
		For every $n \ge 1$, the $1$-skeleton of the graph associahedron for $G_n$ is the undirected Hasse diagram of the lattice of Pell permutations $\Sigma_n$.
	\end{conjecture}

	\begin{conjecture}\label{conj:pell-2}
		The bivariate $f$-polynomial of the $G_n$ hypercube graph family is given by the generating function

\[
f^G = \frac{1-2st-s^2t-s^2t^2}{(1-2st-s^2t-s^2t^2)-s}
\]
	\end{conjecture}

	This polynomial comes from OEIS \cite[Sequence A209695]{oeis}, and is correct for all calculated dimensions.

%
%
%

	\subsection{Near-Double Path Graph} \label{sub:near double path definition}

	The near double path graph is the hypercube graph on $\pm[n]$ with paths $(1,2,\ldots,n-1,n,-1)$ and path $(-2,-3,\ldots,-n)$. The 4-dimensional case is shown in Figure \ref{fig:neardoublepathgraph}; note that here, we find it useful to arrange the vertices on a cycle, but this is the same underlying hypercube shape.

	\begin{figure}[h]

	\begin{center}
	\includegraphics[width=.3\textwidth]{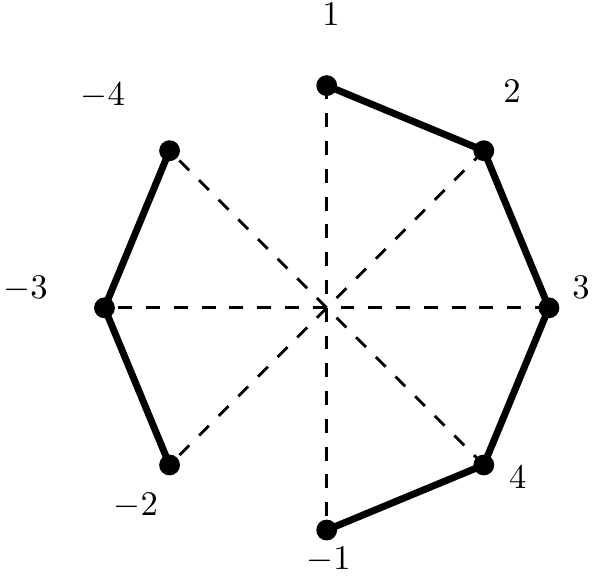}
	\caption{Near-double path graph for $n=4$}\label{fig:neardoublepathgraph}
\end{center}

	\end{figure}

	\begin{proposition}\label{prop:neardoublepath}
		The $n$-dimensional near double path graph hypercube graph associahedron has an $f$-vector equal to the $f$-vector of the $n$-dimensional cyclohedron.
	\end{proposition}

		This is proven using maximal tube sub-complex enumeration in Subsection \ref{sub:neardoublepath}. As we note there, the enumeration here is the same as the enumeration in the double path graph, giving us the same $f$-vector. 

	We note that for $n=1,2,3$, the $n$-dimensional near double path hypercube graph associahedron is combinatorially isomorphic to the cyclohedron of that dimension, but that pattern ends at $n=4$, when the two are not combinatorially isomorphic.

	\subsection{Is the cyclohedron a hypercube-graph associahedron?}

\begin{figure}
\begin{center}
\includegraphics[width=.75\textwidth]{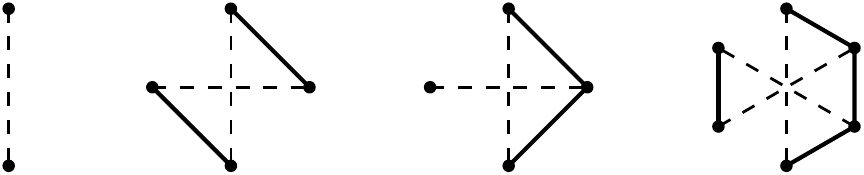}
\end{center}\caption{Known hypercube graphs whose hypercube graph associahedra are cyclohedra.}\label{hypercubecyclohedron}
\end{figure}
	
		Because hypercube graph associahedra are deformations of type $B_n$ permutahedra, it is natural to ask whether the cyclohedron can be expressed as a hypercube graph associahedron. Figure \ref{hypercubecyclohedron} shows the only hypercube graphs of dimension $n \le 3$ whose hypercube graph associahedra are cyclohedra up to isomorphism, proven by machine search; code for the $3$-dimensional case is provided in Appendix \ref{appendix:code}. We have however found several families of hypercube-graph associahedra whose $f$-vectors are equal to those of cyclohedra, despite not being combinatorially isomorphic to them: the double path graph associahedron, and the near-double-path graph associahedron, both of which appear as hypercube graphs in \ref{hypercubecyclohedron} for low dimensions.

\danger{cite appendix?}

		We can think of several features of cyclohedra which suggest they may not be hypercube-graph associahedra. Consider an $n$-dimensional cyclohedron realized as a simplex graph associahedron of a cycle on $n+1$ vertices. For each $0 \le i \le n-1$, there are $(n+1)$ tubes containing $i+1$ vertices. As a result, the cyclohedron has $n+1$ facets isomorphic to $P(A_i) \times P(B_{n-1-i})$, the product of an associahedron times a cyclohedron, and these are all the facets of $P(B_n)$. We know that $P(A_1)$ and $P(B_1)$ are isomorphic to a $1$-simplex.

		If $G$ is a hypercube graph on $\pm[n]$ whose associahedron is isomorphic to a cyclohedron, then we know that each edge induces a graph $A_1$, and each tube with $n-1$ vertices has a reconnected complement isomorphic to a $1$-dimensional hypercube graph as, so we find that the number of edges of $G$, plus the number of tubes of $G$ containing $n-1$ vertices, is equal to $2n+2$. We also know that any connected subgraph which is not an edge or a cycle will induce a simplex-graph associahedron that is not isomorphic to either an associahedron or cyclohedron, and so every connected component of $G$ must be a cycle or a path.

		We have attempted to perform a systematic search for hypercube graph associahedra of dimension $n \ge 4$ isomorphic to cyclohedra, but while this code is efficient enough to run for the case $n=3$, we have not been able to optimize the code to run in time to search all possible graphs.

		\begin{conjecture}\label{conj:no-cyclohedra}
			For all $n \ge 4$, the cyclohedron cannot be expressed as an $n$-dimensional hypercube graph associahedron.
		\end{conjecture}

	\section{Face enumeration Proofs}\label{sec:proofs}

	See the notes at the beginning of this chapter for notes on the structure of this section. Most relevantly, we recall that the notation $f^G$ refers to the bivariate $f$-polynomial of the family of $\p$-graph associahedra of the family of graphs $G=\{G_0,G_1,\ldots\}$, and that some indices have been removed from functions used in Propositions \ref{prop:p-graph-theorem} and \ref{prop:enumeration-maximal-tube-polyhedra}. In this section we use the methods of atomic link sum enumeration and maximal tube enumeration defined in these propositions.

	\subsection{Established $f$-polynomials}

	We recall that for a family of pointed simple polyhedra $P$ with family of dual simplicial complexes $\Delta(P)$, the change of bases $s=1/x, t=xy$, or $x=1/s, y=st$, allows us to write $f^{\Delta(P)}(s,t)=f^P(x,y)$. For most cases we are calculating the bivariate generating functions for polyhedra and not their dual simplicial complexes.

	From \cite[Sequence A033282]{oeis}, the bivariate generating function of the family of associahedra of type $A_n$ is given as follows:

\[
	f^A(x,y) = -\frac{\sqrt{x^2 y^2-2 (x+2) y+1}+(x+2) y-1}{2 (x+1) y^2}.
\]

	From \cite[Sequence A063007]{oeis}, the bivariate generating function of the family of cyclohedra, or associahedra of type $B_n$, is given as follows:

\[
	f^B(x,y) = \frac{1}{\sqrt{x^2 y^2-2 x y-4 y+1}}
\] 

	Define the \emph{point family} $0$ as a family of polyhedra which only contains one polyhedron, a point. There are two ways to conceptualize a point. The first is to define that a point has no facets, and so its dual simplicial complex is an empty set, and any $\p$-graph of a single point is an empty graph. The second is to note that every $n$-dimensional simplex has $n+1$ facets which together form a forbidden subset, and so the forbidden subset diagram of a $0$-simplex consists of a single empty facet. In either case, we can define that the dual simplicial complex consists of only the set $\emptyset$, and we define $f^0(x,y)=1$. We use this as an abstraction; for instance, when we count the number of tubings of a graph with no edges, we may define $\sigma$ to be the shape of tube including only singleton tubes, and we define $\Ag=0$ and $r(n,i)=\delta(i)$.

	We note that the $f$-polynomial of a ray is $(1+s)$, and the $f$-polynomial of a $1$-simplex is $(2+s)$. The bivariate $f$-polynomial of the family of hypercubes is $1+(2+s)t+(2+s)^2t^2+\cdots = \frac{1}{1-(2+s)t}$.

	\subsection{ Halohedron } \label{sub:halohedron}

	We will use maximal tube enumeration to calculate the bivariate generating function for the family of halohedra, proving Theorem \ref{thm:halohedron-formula}. We define the kingmaker set $X_n=\{1,-1\}$ for all $n \ge 1$. We now apply Proposition \ref{prop:enumeration-maximal-tube-polyhedra} to calculate the bivariate $f$-polynomial $f^H$ for the family of halohedra. We find $intset_{X_n}$ contains three distinct types of tubes: path-shaped tubes containing the vertex $1$, the cycle tube equal to $[n]$, and the negative singleton tube $\{-1\}$.

	We consider the set of negative singleton tubes. We typically count the number of tubes in a shape containing $i+1$ vertices, but here we are restricted to the case where $i=0$ and one tube contains $1$ vertex. The singleton tube has a neighborless complement isomorphic to the $(n-1)$-dimensional single path hypercube graph. The induced simplex-graph is a single vertex when $i=0$, and does not exist otherwise. As a result, we find that for the tube containing $i+1$ vertices of this type, we count $r(n,i)=\delta(i)$ tubes which decompose into graphs of the type $\Ag_i \times \Bg_{n-1-i}$, where $\Ag$ is the 0-dimensional simplex graph induced by a tube containing one vertex, and $\Bg$ is the single path hypercube graph. We find $f^\Ag=1$ and $f^\Bg=f^A$. We find $r(n,i)=r^\Ag(i)r^\Bg(n-1-i-\gamma)=(\delta(i))(1)$. This creates operators $O_y^\Ag$ representing the evaluation operator at $0$, and $O_y^\Bg$ representing the identity operator. As a result, we find this tube contributes $f^A$ to a later sum.

	Consider the set of cyclic tubes in $intset_{X_n}$. We immediately see that there exists only $1$ cyclic tube for each $n$, a cycle containing $n=(n-1)+1$ vertices, and so the cycle tube shape is counted by $r(n,i)=\delta(n-1-i)$. For a tube containing $i+1$ vertices, the tube induces a graph $\Ag_i$ which is a simplex cycle graph. The neighborless complement $\Bg_{n-1-i}$ is the empty graph. We find $r(n,i)=(1)(\delta(n-1-i))$, and so we define operators $O_y^\Ag=1$ and $O_y^\Bg$ to be the evaluation operator at $0$. Since $f^\Ag=f^C, f^\Bg=1$, we find this tube contribues $f^C$ to our sum.

	Now consider a path shaped tube containing $i+1$ vertices. When $i < n-1$, there are $i+1$ such tubes, but there are $0$ such paths when $i+1=n$, so we write $r(n,i) = (i+1)(1 - \delta(n-1-i))$. The induced graph $\Ag_i$ is a path graph, so $f^\Ag=f^A$. The neighborless complement is interesting, and is illustrated in Figure \ref{fig:halohedron-complements}. The neighborless complement of a path on $(n-1)$ vertices is a single vertex. The neighborless complement of a path on $(n-2)$ vertices is a pair of vertices. The neighborless complement of a path on $j < (n-2)$ vertices is a pair of two vertices, and a single-path graph on $(n-2-j)$ vertices. We can define this graph $\Bg$. Label these $\p$-graphs with the family $\Bg$, defining $\Bg_0$ as the empty graph, $\mathbb{B}_1$ as the graph on one vertex, $\mathbb{B}_2$ as the graph on two vertices, and so on. We note that the underlying forbidden subset diagrams of these graphs are polyhedral, defining $P_0$ as a vertex, $P_1$ as a ray, $P_2$ as a product of rays, and $P_n$ for $n \ge 2$ as the product of two rays and an $(n-2)$-dimensional hypercube. Taking this into consideration, we find the $P$-graph associahedron of $\Bg_i$ to be a vertex for $n=0$, a ray if $n=1$, and the product of two rays and an $(n-2)$-dimensional associahedron if $n \ge 2$. We then can compute $f^\Bg = 1 +(1+x)y + (1+x)^2 y^2 f^B$. As a result, we find the path tube shape in $intset_{X_n}$ containing $i+1$ vertices splits the graph into a path graph $\Ag_i$ on $i+1$ vertices and a neighborless complement $\Bg_{n-1-i}$. We split $r(n,i)=(i+1)(1-\delta(n-1-i))$ into separable parts, obtaining $r^\Ag(i)=i+1$ and $r^\Bg=1-\delta(n-1-i)$. We calculate linear operators $O_y^\Ag[f(x,y)]=f(x,y)+yD_y f(x,y)$ and $O_y^\Bg[f(x,y)]=f(x,y)-f(x,0)$, and find $O_y^\Ag[f^\Ag(x,y)]=f^A+y D_y f^A$, while $O_y^\Bg[f^\Bg(x,y)]=f^\Bg(x,y)-f^\Bg(x,0)$. We note that $f^\Bg(x,0)=1$, so $O_y^\Bg[f^\Bg(x,y)]=f^\Bg-1 = (1+x)y+(1+x)^2y^2 f^B$, giving $O_y^\Ag[f^\Ag] O_y^\Bg[f^\Bg] = \left( f^A + yD_y f^A\right) \left(f^\Bg-1\right) = \left( f^A + yD_y f^A\right)\left((1+x)y+(1+x)^2y^2 f^B\right)$.

\begin{figure}
\centering
\includegraphics[width=.5\textwidth]{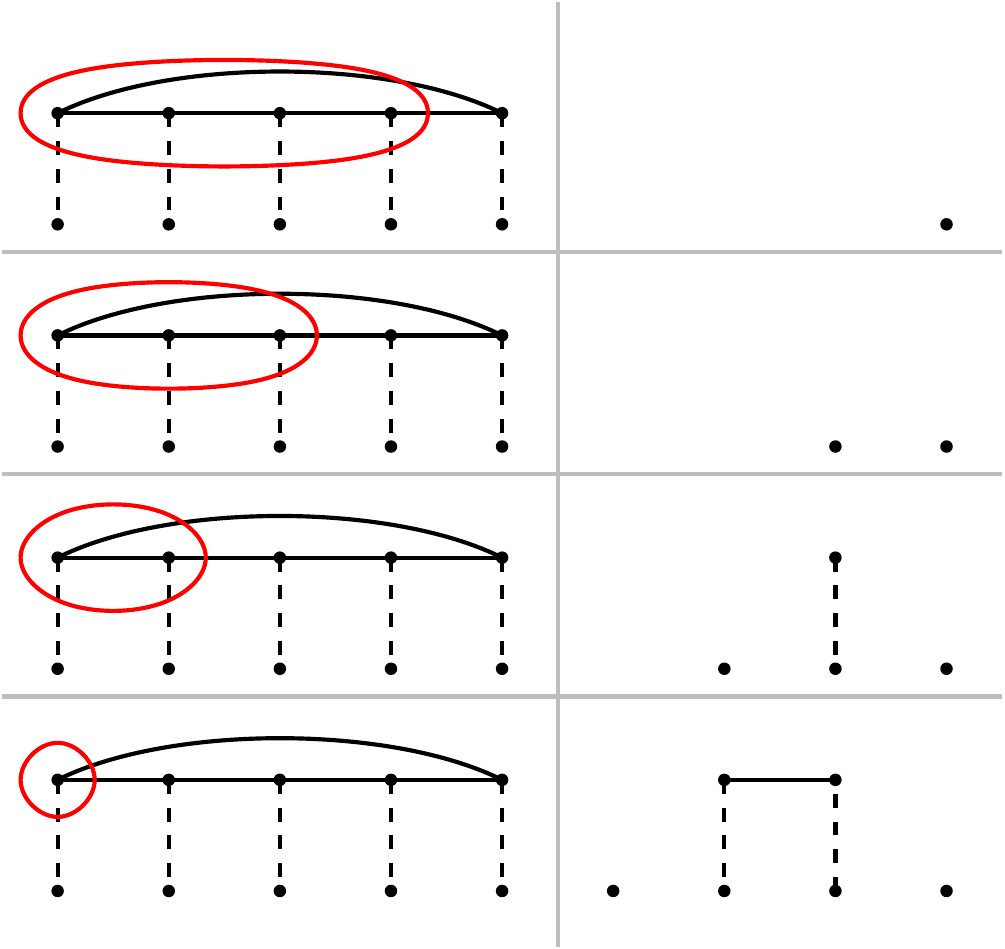}
\caption{The $n=5$ case, and neighborless complements of path tubes of each possible size.} \label{fig:halohedron-complements}
\end{figure}


	Finally, the hypercube graph minus $X_n$ is equal to an $n$-dimensional single path hypercube graph, which has been proven to be isomorphic to an associahedron. As a result, we write $\rho=1$ and get generating function $f^A$. We then plug all this information into Proposition \ref{prop:enumeration-maximal-tube-polyhedra} and obtain:

	\[
		f^H(x,y) = 1+f^A(xy)^1 + y \left[
(f^A) + (f^B) + \left( f^B + y D_y f^B\right) \left( (1+x)y+(1+x)^2y^2 f^B \right)
\right]
	\]


This expression simplifies to the expression in Theorem \ref{thm:halohedron-formula}.

	\subsection{Note on double path graph variants}

	We wish to use the maximal tube enumeration method to calculate the bivariate $f$-polynomial of the double path graph, but in order to do so we must define $\p$-graphs which will correspond to neighborless complements of tubes in the double path graph. Several of these graphs are used in multiple calculations. If $D_n$ is the double path graph on $\pm[n]$, define the \emph{missing vertex double path graph} $M_n$ as the graph obtained by removing the vertex $-1$ from $D_n$. Define the \emph{trans double path graph} as the graph obtained by removing $-1$ and $n$ from $D_n$, and define the \emph{cis double path graph} as the graph obtained by removing $-1$ and $-n$ from $D_n$. Example graphs are shown for the $n=5$ case in Figure \ref{fig:double-path-variants} These graphs are $\p$-graphs, with their respective polyhedra isomorphic to products of hypercubes and rays.


	\begin{figure}[h]
	\centering
	\includegraphics[width=.8\textwidth]{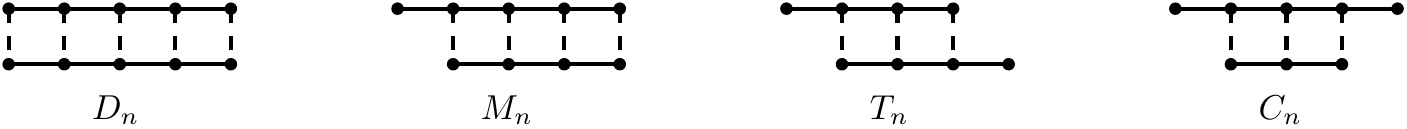}
	\caption{$\p$-graphs for $5$-dimensional double path, missing vertex double path, trans double path, and cis double path graphs.}\label{fig:double-path-variants}
\end{figure}

	\subsection{Missing vertex double path graph}

	We define the \emph{missing vertex double path graph } $M_n$ as the double path hypercube graph $D_n$ on $\pm[n]$, except with the vertex $-1$ removed. This is the graph on vertex set $\{1,2,\ldots, n, -2, \ldots, -n\}$, with paths $(1,2,\ldots,n)$ and $(-2,\ldots,-n)$, and forbidden subsets $\{i,-i\}$ for $2 \le i \le n$. We note that the forbidden subset diagram is dual to the product of a ray and $n-1$ $1$-simplices.

	We calculate the bivariate generating function for the $f$-polynomials of this family of graphs as follows. We define $X_n=\{1\}$ for each $n \ge 1$. We find that the graph $M_n \backslash X_n$ is isomorphic to the hypercube graph $D_{n-1}$.

	Now we consider each tube in $intset_{X_n}$. Consider a tube $t_i \in intset_{X_n}$ containing $i+1$ vertices

 For each $n$ and $0 \le i \le n-1$, there is exactly one path tube in $intset_{X_n}$ containing $i+1$ vertices, which will have graph $G_i$ isomorphic to the path graph $A_i$. We define $r(n,i)=1$, and find the neighborless complement of a path tube containing $i+1$ vertices is a graph on vertices $\{i+3, \ldots, n, -(i+2),\ldots,-(n)\}$. This neighborless complement is isomorphic to the graph $M_{n-1-i}$. We then note that having $r(n,i)=1$ copies of $G_i \times H_{n-1-i}$.

	From Proposition \ref{prop:enumeration-maximal-tube-polyhedra}, we find
	\[
		f^M = 1 +xy f^D + y f^A f^M.
	\]

	We note that this result is used in Subsection \ref{sub:doublepath} to calculate $f^D$, and in the process we calculate
		\[
			f^M = \frac{1}{(1-xy)(1+yf^A)-2yf^A}.
		\]
		We note that the dual function $F^{\Delta(M)}(s,t)$ is equal to the generating function of the triangle listed as Sequence A123160 in OEIS \cite[Sequence A123160]{oeis}.

	\subsection{Double path graph associahedron}\label{sub:doublepath}

	Recall that $D_n$ is the double path hypercube graph on vertices $\pm[n]$. We can apply maximal tube subcomplex enumeration to calculate the bivariate generating function $f^D$. First, we define $X_n = \{1,-1\}$ for $n \ge 1$. We note that $D_n \backslash X_n$ is the graph $D_{n-1}$.

	The only tubes in $intset_{X_n}$ are path shaped, so consider a path-shaped tube $t_i$ containing $i+1$ vertices. There are $r(n,i)=2$ such tubes in $intset_{X_n}$ for each $n$. We find each tube $t_i$ induces a path graph $A_i$. The neighborless complement of $t_i$ is the graph containing vertices $\{i+3, \ldots, n, -(i+2),\ldots,-(n)\}$, and is isomorphic to $M_{n-1-i}$. As a result, the tube $t_i$ splits the graph into $A_i$ and $M_{n-1-i}$. The result of operators $O_y^A[f^A] O_y^M[f^M]$ is $2 f^A f^M$.

	As a result, we apply Proposition \ref{prop:enumeration-maximal-tube-polyhedra}, and find

	\[
		f^D = 1 + xy f^D + y \left(2 f^A f^M\right).
	\]
	We recall that $f^M = 1+xy f^D + y f^A f^M$. We can solve these two algebraic equations to find two new equations:
	\[
	f^M(x,y) = \frac{f^D}{1+y f^A} = \frac{1}{(1-xy)(1+yf^A)-2yf^A}
	\]
	\[
		f^D(x,y) = \frac{1+y f^A}{(1-xy)(1+yf^A)-2yf^A}.
	\]

	Plugging in the known equation $f^A(x,y)$ and simplifying heavily gives us the equation $f^D(x,y)=f^B(x,y)$, proving Proposition \ref{prop:doublepath}.

	\subsection{Near double path}\label{sub:neardoublepath}

	Recall the near double path hypercube graph $NDP_n$ is the hypercube graph on $\pm[n]$ with paths $(1,2,\ldots,n,-1)$ and $(-2,-3,\ldots,-n)$. We will use the maximal tube complex enumeration method to calculate the bivariate generating function of the near double path hypercube graph associahedron.

	Define $X_n=\{-1,1\}$ for all $n \ge 1$. We find that $NDP_n \backslash X_n$ is isomorphic to the double path graph $D_{n-1}$. Consider $intset_{X_n}$. There are $r(n,i)=2$ path tubes in $intset_{X_n}$ containing $i+1$ vertices. For any such tube $t_i$, the tube $t_i$ induces a graph isomorphic to $A_i$, and has a neighborless complement isomorphic to the graph $M_{n-1-i}$.

	From Proposition \ref{prop:enumeration-maximal-tube-polyhedra}, we find

	\[
	f^{NDP}(x,y) = 1 + xy f^D + y \left(2 f^A f^M\right).
	\]

	If we recall the equation $f^D = 1 + xy f^D + 2y f^A f^M$, then we note that $f^{NDP}(x,y)=f^D(x,y)$, which is equal to $f^B(x,y)$, therefore proving Proposition \ref{prop:neardoublepath}.

	\subsection{Cis and trans double path graphs}\label{sub:cistransgraphs}

	Define the \emph{trans double path graph} $T_n$ as the graph obtained by removing $-1$ and $n$ from the double path hypercube graph $D_n$, and define the \emph{cis double path graph} $C_n$ as the graph obtained by removing $-1$ and $-n$ from $D_n$.

	We will perform maximal tube complex enumeration for the cis double path graph $C_n$. First, define $X_n=\{1\}$ for $n \ge 1$. We note $C_n \backslash X_n$ is isomorphic to $M_{n-1}$. The set $intset_{X_n}$ is the set of tubes containing the vertex $1$. There is exactly one tube $t_i$ for each $0 \le i \le n-1$ containing the vertex $1$ and containing $i+1$ vertices, and this is a path tube inducing the graph $A_i$. The neighborless complement of this tubes is a graph containing vertices $\{i+3, \cdots, n, i+2, \ldots, n-1\}$. This graph is isomorphic to $T_{n-1-i}$.

	As a result, from Proposition \ref{prop:enumeration-maximal-tube-polyhedra}, we find

	\[
		f^C = 1 + xy f^M + y \left(f^A f^T\right)
	\]

	Now consider the maximal tube complex enumeration for $T_n$. Define $X_n = \{1\}$. We note that $T_n \backslash X_n$ is isomorphic to $M_{n-1}$.

	Consider the set $intset_{X_n}$. This set contains the tubes $t_i=\{1,\ldots,i+1\}$ for $1 \le i \le n-2$. This means that we have to define $r(n,i) = (1-\delta(n-1-i))$. The tube $t_i$ induces a path graph $A_i$. The neighborless complement is a graph on vertices $\{i+3, \ldots, n-1,i+2,\ldots,n\}$. This is a complex isomorphic to $C_{n-1-i}$. We find that $r(n,i)$ is separable into $r^A(i)=1$ and $r^C(n-1-i)=1-\delta(n-1-i)$. This corresponds to linear operators $O_y^A[f^A]=f^A$ and $O_y^C[f^C] = f^C(x,y)-f^C(x,0)$. The function $f^C(x,0)=1$, giving us the resulting equation:

	\[
		f^T = 1 + xy f^M + y \left(f^A(f^C-1)\right).
	\]

	Solving these two equations together gives the formulas

		\[
	f^T = \frac{1+xy f^M}{1-y f^A} - \frac{y f^A}{1-(yf^A)^2}
\]
\[
	f^C = \frac{1+xy f^M}{1-y f^A} - \frac{(y f^A)^2}{1-(y f^A)^2}.
\]

	\subsection{Twisted Path}\label{sub:twistedpath}

	Recall that the twisted path graph $TP_n$ is the hypercube graph on $\pm[n]$ with path $(1,\ldots,n,-1,\ldots,-n)$. We will perform maximal tube complex enumeration on this hypercube graph.

	We define $X_n=\{n,-1\}$. The graph $TP_n\backslash X_n$ is the trans double path hypercube graph $T_n$, meaning we set $\rho=0$ in our formula.

	The set $intset_{X_n}$ consists of path tubes. We note that for each $0 \le i \le n-1$, there are $r(n,i)=i+2$ path tubes in $intset_{X_n}$ containing $i+1$ vertices. Consider one such tube $t_i$. We find $t_i$ induces a path graph isomorphic to $A_i$. The neighborless complement of $t_i$ is isomorphic to a trans path graph $T_{n-1-i}$. We find that the product $O_y^A[f^A] O_y^T[f^T]$ is equal to $(2f^A+y D_y f^A) (f^T)$. Together, this calculation gives

	\[
		f^{TP} = f^T +y \left(2f^A+2yD_y f^A\right) (f^T)
	\]

	Expanding $f^T$  and then $f^M$ gives the following result:

			\[
			f^{TP}(x,y) =
		\left(	1 + \frac{xy}{(1-xy)(1+yf^A)-2yf^A} + yf^A(f^C-1) \right)
 \left(1+y\left[(2+D_y)f^A\right]\right)
		\]
	which proves Proposition \ref{prop:twistedpath}.

	\subsection{Twisted Cycle}\label{sub:twistedcycle}
	
	Recall that the twisted cycle hypercube graph is the hypercube graph on $\pm[n]$ with cycle $(1,\ldots,n,-1,\ldots,-n)$. Call this graph $TC_n$.

	We use facet sum polynomial enumeration to calculate the bivariate $f$-polynomial for the face counts of twisted cycle hypercube graph associahedra.

	We note that $TC_n$ only has path-shaped tubes. For each $0 \le i \le n-1$, there are $r(n,i)=2n$ path tubes containing $i+1$ vertices. Each such tube $t_i$ induces a path graph isomorphic to $A_i$, and a reconnected complement isomorphic to $TP_{n-1-i}$. We note that $r(n,i)$ can be written as a sum of separable functions for terms $(n-1-i,i)$ if we rewrite it as $r(n,i)=(2)(n-1-i)+(2i+2)(1)$. We now apply Proposition \ref{prop:enumeration-maximal-tube-polyhedra} to define a differential equation
	\[
		(yD_y-xD_x) f^{TC}(x,y) = y \left(2f^A(yD_y f^{TP}) + (2f^A + 2 y D_y f^A)(f^{TP})
				\right).
	\]
	We then consider a hypothesis:

	\[
		f^{TC}(x,y) = \frac{(1-xy)-2y}{(1-xy)^2-4y}.
	\]

	This function satisfies our differential equation. It also satisfies initial conditions; we note that $f^{\Delta(TC)}(s,t)=\frac{(1-t)-2st}{(1-t)^2-4st}$, and so $f^{TC}(0,t)=\frac{1}{1-t}$.

, proving Proposition \ref{prop:twisetdcycle}. This function can be found in an altered form at \cite[Sequence A127674]{oeis}, which is the list of even rows of nonzero coefficients of Chebyshev polynomials, with absolute value taken.

	\subsection{Double Cycle}\label{sub:doublecycle}

	We can use maximal tube complex enumeration on the double cycle graph $DC_n$, defining $X_n=\{-1,1\}$ for this graph. The graph $DC_n \backslash X_n$ is isomorphic to the double path graph $D_{n-1}$, whose generating function is equal to that of $B_{n-1}$.

	There are two tube shapes in $intset_{X_n}$. For the path tube shape, there are $2(i+1)$ tubes in $intset_{X_n}$ containing $i+1$ vertices for $0 \le i \le n-2$. This gives a count of $r_{path}(n,i)=(2i+2)(1-\delta(n-1-i))$ path tubes with $i+1$ vertices. Each such tube induces a simplex-graph isomorphic to a path graph $A_i$, and the neighborless complement of each such tube is a cis double path graph $C_{n-1-i}$. As a result, we find $O_y^A[f^A] O_y^C[f^C]=(2f^A+2yD_y f^A) (f^C-1)$.

	Now there are two cycle tubes in $DC_n$, each containing $n$ vertices, so $r_{cycle}(n,i)=(2)(\delta(n-1-i))$. This induces a cycle graph $B_{n-1-i}$ and an empty neighborless complement $0$. As a result we find $O_y^B[f^B] O_y^0[f^0]=2f^B$.

	We then calculate
	\[
		f^{DC} =1 + xy f^{B} + y \left[(2f^A+2yD_y f^A) (f^C-1)		+		2 f^B
\right]
	\]
	which proves Proposition \ref{prop:doublecycle}.

\chapter{Future Work}
\label{chap:res3}

The work outlined in this thesis presents several possible avenues for exploration in the future. There are several conjectures which we wish to be able to confirm or disprove in the future. Conjecture \ref{conj:no-cyclohedra} suggests that the cyclohedron is not isomorphic to any hypercube graph associahedron for dimensions $n \ge 4$. Conjectures \ref{conj:pell-1} and \ref{conj:pell-2} are contingent upon the existence of a bijection between the poset of maximal tubings of the Pell hypercube-graph and the lattice of sashes. In addition, we have only explored in depth $\p$-nestohedra for hypercubes, but $\p$-nestohedra exist for any simple polyhedron. The following subsections suggest several other directions for further research.

\section{Type $B_n$ signed posets}

The type $A_n$ Coxeter arrangement is defined by hyperplanes of the form $x_i = x_j$ for any $i, j \in [n+1]$. The type $B_n$ Coxeter arrangement is defined by hyperplanes of the form $x_i=x_j, x_i=-x_j$, or $x_i=0$ for any $i, j \in [n+1]$. Section \ref{sec:classical-braid-cones} describes the correlation between braid cones of the type $A_n$ fan, and preposets on the set $[n+1]$. Cones coarsening the type $B_n$ Coxeter fan can be defined by collections of facet-defining inequalities, and each such inequality is in bijection with a root of the type $B_n$ root sytem. A generalization of posets called \emph{signed posets} is defined in \cite{signedposets}, where it is known that these signed posets are in bjiection with full-dimensional type $B_n$ braid cones.

Facial preposets were defined with generality in mind, and not just to work with hypercubes. in the hypercube case no assumptions are made that the vectors $v_i, v_{-i}$ are inverses of each other, let alone parallel. There is no guarantee that the bounding walls of maximal cones align into hyperplanes. However, when using a standard set of vectors, we find that every cone in the normal fan of a standard cut hypercube graph associahedron is a type $B_n$ braid cone. We note that $x_i < x_j$ implies that $-x_j < -x_i$. This symmetry is reflected in signed preposets, where we can say that $i < j$ implies $-j < -i$, but we note that it is absent in the case of facial preposets of hypercubes, where $i \prec j$ does not imply $-j \prec -i$. A treatment of facial preposets of hypercubes as signed posets may be useful for future research.

Polyhedra whose normal fans coarsen a type $W$ Coxeter fan are known as \emph{deformations of Coxeter permutahedra}, and there exists research into these polyhedra \cite{coxeterpermutahedra}. We believe that knowledge of properties of deformations of type $B_n$ Coxeter permutahedra will be useful in understanding hypercube-graph associahedra, and finding applications for hypercube-graph associahedra.

	\section{Posets of Maximal Nested Sets} \label{maximalnestedsets}

	Given a polytope $P \in \R^n$, consider the normal fan of $P$. We can define a poset on the maximal cones of the normal fan, called regions, by defining a poset on the vertices of $P$. These posets of regions can be interesting generalizations of the Coxeter weak order. For example, if $P$ is the permutahedron of a Coxeter group $W$, then the poset of regions of the normal fan is isomorphic to the weak order on $W$.

	\begin{definition}
		The \emph{poset of regions} of a fan $F$ with respect to a vector $\lambda$ is the poset on regions of $F$ generated by the set of relations of the form $R_1 \preceq R_2$ whenever there exists points $v_1 \in R_1, v_2 \in R_2$ such that $v_2= \alpha \lambda+v_1$ for some nonnegative number $\lambda$.
	\end{definition}



	If we have two fans $F_1$, $F_2$ such that $F_2$ coarsens $F_1$, then there exists a map on regions of $F_1$ to regions of $F_2$ defined such that $R_1 \mapsto R_2$ if $R_1 \subseteq R_2$. Now consider that for a classical graph associahedron, every maximal cone of the normal fan is dual to a maximal tubing, and as a result we may consider this poset to be the \emph{poset on maximal tubings.} When the classical graph associahedron is realized as a generalized permutahedron, its normal fan coarsens the type $A_n$ Coxeter fan, and so there is defined a map from the weak order of type $A_n$ to the poset of maximal tubings of a graph. The work in \cite{barnard2018lattices} characterizes every graph whose classical graph associahedron defines a map from the weak order of $A_n$ to the poset of maximal tubings is a lattice congruence.

	We recall that every standard cut hypercube graph associahedron has a fan that coarsens the type $B_n$ Coxeter fan. As a result, every hypercube graph can be associated with a map from the type $B_n$ weak order to the poset of maximal tubings of that hypercube graph. In the future, we wish to study the poset of maximal nested sets. One major unanswered question is as follows: is there a similar condition to that found in \cite{barnard2018lattices} which characterizes the maps on the type $B_n$ weak order induced by hypercube graph associahedra which are lattice congruences?

	We also wish to prove Conjecture \ref{conj:pell-1} in terms of defining the poset on maximal tubings for the Pell graph, and seeing it how it relates to the lattice of sashes.

	\section{Improved enumeration techniques}

	The  methods used in this paper are very effective when the counting function $r(n,i)$ is linear, but work less well whenever a term such as $r(n,i)={n \choose i}$ appears. We anticipate that expanding our enumeration method to count functions with exponential generating functions will allow us to perform more computations.

\subsection{Wand Graphs}

	The wand graph is not a hypercube graph, but is instead a simplex-graph we encountered while searching for families of simplex-graphs which are atomically closed. For any $j, k \ge 0$, define the graph $W_{j,k}$ on $[j+k]$ with a clique on vertices $1,\ldots, j$, and edges $\{i,i+1\}$ for all $j \le i < j+k$. When $j=0$ this is a path graph on $k$ vertices, and when $k=0$ this is a complete graph on $j$ vertices. We note that every reconnected complement of a wand graph is itself a wand graph. As a result, we anticipate that expanding our method of atomic link sum enumeration to allow families of graphs indexed by more than one variable, and compensating for exponential generating functions, may allow us to adapt this method to this case and calculate the trivariate $f$-polynomial for this family of graphs.

	\begin{conjecture}
		If $w_{j,k}$ is the number of maximal tubings of the $W_{j,k}$ graph, the mixed ordinary/exponential generating function is
		\[
			\sum_{j,k \ge 0} w_{j,k} (x^j/j!) y^k =\frac{2}{1+\sqrt{1-4y}-2x}.
		\]
	\end{conjecture}
We note that this is equal to $j! T(j,j+k)$, where $T(j,j+k)$ is an entry in Catalan's triangle as defined in \cite[Sequence A009766]{oeis}.



\printbibliography

\restoregeometry
\appendix
%
\chapter{SageMath Code} \label{appendix:code}

A good amount of code was used to directly compute certain hypercube-graph associahedra. The following SageMath code gives commands to define hypercube graphs and create and plot hypercube-graph associahedra and simplex-graph associahedra, as well as plot them. My copy of Jupyter uses JSmol to display graphics. SageMath code will also be available to download from my personal website.

\begin{verbatim}

# Defines a HypercubeGraph object containing graph/edge data
# which can call a command to create a hypercube graph associahedron
# polytope.

class HypercubeGraph(object):
    def __init__(self,en,edgelist):
        self.G = Graph()
        self.n = en
        for i in (1..en):
            self.G.add_vertex(i)
            self.G.add_vertex(-(i))
        self.G.add_edges(edgelist)
    
    # Returns the underlying graph of this hypercube graph
    
    def get_graph(self):
        return self.G
    
    def get_n(self):
        return self.n

    # Used to iterate over all possible connected subgraphs, and check
    # which of them are tubes. Very unoptimized. 
    
    def tube_iterator_1(self):
        subgraphs = self.G.connected_subgraph_iterator(vertices_only=True)
        for graph in subgraphs:
            isTube= True
            for i in range(1,self.n+1):
                if {i,-i}.issubset(graph):
                    isTube=False
            if isTube:
                yield graph

    # Return a hypercube-graph associahedron of G.
    
    def associahedron(self):
        tubes = self.tube_iterator_1()
        P = Polyhedron(ieqs = [self.inequality(tube) for tube in tubes])
        return P

    # Return a hypercube-graph associahedron of G, with custom specified
    # truncation depth function based on tube size/face dimension.
    
    def custom_associahedron(self,distance):
        tubes = self.tube_iterator_1()
        P = Polyhedron(ieqs = [self.custom_inequality(tube,distance)
            for tube in tubes])
        return P
    
    def inequality(self,tube):
        ieq = [(len(tube)*3^(self.n-1) -3^(len(tube)-2))]
        #We swap positive and negative so we get correct orientation
        for i in range(1,self.n+1):
            if i in tube:
                ieq.append(-1)
            elif (-i) in tube:
                ieq.append(1)
            else:
                ieq.append(0)
        return ieq

    def custom_inequality(self,tube,distance):
        ieq = [distance(len(tube))]
        for i in range(1,self.n+1): #swapping positive and negative...
            if i in tube:
                ieq.append(-1)
            elif (-i) in tube:
                ieq.append(1)
            else:
                ieq.append(0)
        return ieq

###############################################################

# These two functions are used to define simplex-graph associahedra.
# simplex_custom_inequality has you input the list of all tubes yourself.
# This is often optimal for well-understood graphs in high dimensions.
#
# simplex_graph_associahedron inputs a graph on [n+1] and iterates
# over all tubes automatically.

###############################################################
def simplex_custom_inequality(tube,n,f):
    ieq = [f(len(tube))]
    for i in range(1,n+1):
        if i in tube:
            ieq.append(1)
        else:
            ieq.append(0)
    return ieq
def simplex_graph_associahedron(graph,f=lambda x:3-3^(-x)):
    n = len(graph.vertices())
    tubes = graph.connected_subgraph_iterator(vertices_only=True)
    ieq_list = [simplex_custom_inequality(tube,n,f) for tube in tubes]
    eqn_list = [[1 for i in range(0,n+1)]]
    P = Polyhedron(ieqs = ieq_list,eqns=eqn_list)
    return P

# Plots polytope with pleasing ratios, then outputs a pretty spinning
# plot in addition to showing the original hypercube graph.

def niceplot(hg,f=lambda x: x-3^(x-3),txt=""): 
    hg.get_graph().show(layout='circular')     
    return text(txt,(0,0)).plot()+hg.custom_associahedron(f).plot(
        frame=false,wireframe='black',polygon=False,point=False,spin=True)

\end{verbatim}

The following commands define hypercube-graphs by providing lists of edges.

\begin{verbatim}
G1 = HypercubeGraph(3,[(1,2),(2,3)])
G2 = HypercubeGraph(3,[(1,2),(2,3),(1,3),(-1,-2),(-2,-3),(-1,-3)])
G3 = HypercubeGraph(3,full_edge_list(3))
\end{verbatim}

Each of these commands individually will display a labeled 3-dimensional hypercube-graph associahedron.

\begin{verbatim}
niceplot(G1,txt="Associahedron")
niceplot(G2,txt="Permutahedron")
niceplot(G3,txt="B3 Permutahedron")
\end{verbatim}

Several commands have been very helpful. Two of these are 
\verb|f_vector()| \linebreak and \verb|is_combinatorially_isomorphic()|. The following code calculates the $f$-vectors of hypercube-graph associahedra of Pell graphs up to $6$ dimensions.

\begin{verbatim}
def partialpellgraph(n):
    edgelist=[]
    for i in range(1,n):
        edgelist.append([i,-i-1])
    return HypercubeGraph(n,edgelist)
for n in range(1,7):
    print partialpellgraph(n).associahedron().f_vector()
\end{verbatim}

The following code iterates over all $3$-dimensional hypercube graphs, and returns the list of hypercube-graphs whose associahedra are combinatorially isomorphic to the $3$-dimensional associahedron.

\begin{verbatim}
A3_associahedron = simplex_graph_associahedron(Graph([(1,2),(2,3),(3,4)]))
count = 0
for edgelist in subsets(full_edge_list(3)):
    P = HypercubeGraph(3,edgelist).associahedron()
    if P.is_combinatorially_isomorphic(A3_associahedron):
        count = count+1
        print edgelist
print count 
\end{verbatim}




\restoregeometry


\backmatter

\end{document}